\documentclass{article}
\usepackage[english]{babel}
\usepackage{amsmath,amssymb,graphicx,stmaryrd,enumerate,hyperref,mathtools,cleveref,latexsym,makeidx,xcolor,theorem}
\makeindex

\newcommand{\assign}{:=}
\newcommand{\backassign}{=:}
\newcommand{\divides}{\mathrel{|}}
\newcommand{\dottimes}{\mathbin{\dot{\times}}}
\newcommand{\ggeq}{\geq\!\!\!\geq}
\newcommand{\glossaryentry}[3]{\item[{#1}\hfill]#2\dotfill#3}
\newcommand{\infixand}{\text{ and }}
\newcommand{\infixor}{\text{ or }}
\newcommand{\leangle}{\mathrel{\angle}}
\newcommand{\lebar}{\mathrel{\Yleft}}
\newcommand{\legeangle}{\mathrel{\substack{\leangle\\\anglege}}}
\newcommand{\leqangle}{\mathrel{\angle \llap {\raisebox{-1ex}{$-$}}}}
\newcommand{\lleq}{\leq\!\!\!\leq}
\newcommand{\mathe}{\mathrm{e}}
\newcommand{\matheuler}{\gamma}
\newcommand{\mathpi}{\pi}
\newcommand{\nin}{\not\in}
\newcommand{\nobracket}{}
\newcommand{\nocomma}{}
\newcommand{\nosymbol}{}
\newcommand{\op}[1]{#1}
\newcommand{\pplus}{\mathbin{+\!\!\!\!+}}
\newcommand{\suchthat}{:}
\newcommand{\tmabbr}[1]{#1}
\newcommand{\tmaffiliation}[1]{\\ #1}
\newcommand{\tmdummy}{$\mbox{}$}
\newcommand{\tmem}[1]{{\em #1\/}}
\newcommand{\tmemail}[1]{\\ \textit{Email:} \texttt{#1}}
\newcommand{\tmmathbf}[1]{\ensuremath{\boldsymbol{#1}}}
\newcommand{\tmop}[1]{\ensuremath{\operatorname{#1}}}
\newcommand{\tmrsub}[1]{\ensuremath{_{\textrm{#1}}}}
\newcommand{\tmstrong}[1]{\textbf{#1}}
\newcommand{\tmtextbf}[1]{\text{{\bfseries{#1}}}}
\newcommand{\tmtextit}[1]{\text{{\itshape{#1}}}}
\newcommand{\tmtextrm}[1]{\text{{\rmfamily{#1}}}}

\newcommand{\tmtextup}[1]{\text{{\upshape{#1}}}}

\newcommand{\udots}{{\mathinner{\mskip1mu\raise1pt\vbox{\kern7pt\hbox{.}}\mskip2mu\raise4pt\hbox{.}\mskip2mu\raise7pt\hbox{.}\mskip1mu}}}
\newenvironment{descriptioncompact}{\begin{description} }{\end{description}}
\newenvironment{enumeratealpha}{\begin{enumerate}[a{\textup{)}}] }{\end{enumerate}}
\newenvironment{enumeratenumeric}{\begin{enumerate}[1.] }{\end{enumerate}}
\newenvironment{enumerateroman}{\begin{enumerate}[i.] }{\end{enumerate}}
\newenvironment{itemizedot}{\begin{itemize} }{\end{itemize}}
\newenvironment{proof}{\noindent\textbf{Proof\ }}{\hspace*{\fill}$\Box$\medskip}
\newenvironment{proof*}[1]{\noindent\textbf{#1\ }}{\hspace*{\fill}$\Box$\medskip}
\newenvironment{theglossary}[1]{\begin{list}{}{\setlength{\labelwidth}{6.5em}\setlength{\leftmargin}{7em}\small} }{\end{list}}
\newtheorem{theorem}{Theorem}[section]
\newtheorem{corollary}[theorem]{Corollary}
\newtheorem{definition}[theorem]{Definition}
{\theorembodyfont{\rmfamily}\newtheorem{example}[theorem]{Example}}
\newtheorem{lemma}[theorem]{Lemma}
\newcommand{\nonconverted}[1]{\mbox{}}
\newtheorem{proposition}[theorem]{Proposition}
{\theorembodyfont{\rmfamily}\newtheorem{remark}[theorem]{Remark}}
{\theorembodyfont{\rmfamily}\newtheorem{warning}{Warning}}
\providecommand{\xLeftrightarrow}[2][]{\mathop{\Longleftrightarrow}\limits_{#1}^{#2}}

\newcommand{\open}{\textrm{?}}

\hypersetup{
    colorlinks = true,
    linkcolor = {blue},
    citecolor = {cyan}
}

\begin{document}

\title{Surreal numbers as hyperseries}

\author{
  Vincent Bagayoko
  \tmaffiliation{imj-prg}, France
  \tmemail{bagayoko@imj-prg.fr}
  \and
  Joris van der Hoeven
  \tmaffiliation{CNRS, LIX}, France
  \tmemail{vdhoeven@lix.polytechnique.fr}
}

\maketitle

\begin{abstract}
  Surreal numbers form the ultimate extension of the field of real numbers
  with infinitely large and small quantities and in particular with all
  ordinal numbers. Hyperseries can be regarded as the ultimate formal device
  for representing regular growth rates at infinity. In this paper, we show
  that any surreal number can naturally be regarded as the value of a
  hyperseries at the first infinite ordinal $\omega$. This yields a remarkable
  correspondence between two types of infinities: numbers and growth rates.
\end{abstract}

\section{Introduction}

\subsection{Toward a unification of infinities}\label{unif-inf-sec}

At the end of the 19-th century, two theories emerged for computations with
infinitely large quantities. The first one was due to
du~Bois-Reymond~{\cite{DBR1870,DBR1875,DBR1877}}, who developed a~``calculus
of infinities'' to deal with the growth rates of functions in one real
variable at infinity. The second theory of ``ordinal numbers'' was proposed by
Cantor~{\cite{Cantor}} as a~way to count beyond the natural numbers and to
describe the sizes of sets in his recently introduced set theory.

Du Bois-Reymond's original theory was partly informal and not to the taste of
Cantor, who misunderstood it~{\cite{Fish}}. The theory was firmly
grounded and further developed by Hausdorff and Hardy. Hausdorff formalized
du~Bois-Reymond's ``orders of infinity'' in Cantor's set-theoretic
universe~{\cite{Fel02}}. Hardy focused on the computational aspects
and introduced the differential field of \tmtextit{logarithmico-exponential
functions}~{\cite{H1910,H1912}}: such a function is constructed from the real
numbers and an indeterminate $x$ (that we think of as tending to infinity)
using the field operations, exponentiation, and the logarithm. Subsequently,
this led to the notion of a Hardy field~{\cite{Bou61}}.

As to Cantor's theory of ordinal numbers, Conway proposed a dramatic
generalization in the 1970s. Originally motivated by game theory, he
introduced the proper class~$\mathbf{No}$ of {\tmem{surreal
numbers}}~{\cite{Con76}}, which simultaneously contains the set $\mathbb{R}$
of all real numbers and the class $\mathbf{On}$ of all ordinals. This class
comes with a natural ordering and arithmetic operations that turn
$\mathbf{No}$ into a non-Archimedean real closed field. In particular, $\omega
+ \mathpi$, $\omega^{- 1}$, $\sqrt{\omega}$, $\omega^{\omega} - 3 \omega^2$
are all surreal numbers, where $\omega$ stands for the first infinite ordinal.

Conway's original definition of surreal numbers is somewhat informal and draws
inspiration from both Dedekind cuts and von Neumann's construction of the
ordinals:

  \begin{quote}'If $L$ and $R$ are any two sets of (surreal) numbers, and no member of $L$
  is~$\geqslant$ any member of~$R$, then there is a (surreal) number $\{ L|R
  \}$. All (surreal) numbers are constructed in this way.'\end{quote}

{\noindent}The notation~$\{ \nosymbol | \nosymbol \}$ is called
{\tmem{Conway's bracket}}. Conway proposed to consider $\{ L|R \}$ as the
{\tmem{simplest}} number between $L$ and $R$. In general, one may define a
partial ordering $\sqsubset$ on~$\mathbf{No}$ by setting $a \sqsubset b$ (we
say that $a$ is {\tmem{simpler}} than $b$) if there exist $L$ and $R$ with $b
= \{ L|R \}$ and {$a \in L \cup R$}. Conway's bracket is uniquely determined
by this simplicity relation.

The ring operations on~$\mathbf{No}$ are defined in a recursive way that is
both very concise and intuitive: given $x = \{ x_L |x_R \}$ and $y = \{ y_L
|y_R \}$, we define
\begin{eqnarray*}
  0 & \assign & \{ \nosymbol | \nosymbol \}\\
  1 & \assign & \{ 0| \nosymbol \}\\
  - x & \assign & \{ - x_R \divides - x_L \}\\
  x + y & \assign & \{ x_L + y, x + y_L |x_R + y, x + y_R \}\\
  xy & \assign & \left\{ x' y + xy' - x' y', x'' y + xy'' - x'' y'' \right| \\ & & \left.x' y + xy'' -
  x' y'', x'' y + xy' - x'' y' \right\}\\
  &  & \qquad \left( x' \in x_L, \; x'' \in x_R, \; y' \in y_L, \; y'' \in
  y_R \right) .
\end{eqnarray*}
It is quite amazing that these definitions coincide with the traditional
definitions when $x$ and $y$ are real, but that they also work for the ordinal
numbers and beyond. Subsequently, Gonshor also showed how to extend the real
exponential function to $\mathbf{No}$~{\cite{Gon86}} and this extension
preserves all first order properties of $\exp$~{\cite{vdDE01}}. Simpler
accounts and definitions of $\exp$ can be found in {\cite{MM17,Ber20}}.

The theory of Hardy fields focuses on the study of growth properties of germs
of actual real differentiable functions at infinity. An analogue formal theory
arose after the introduction of {\tmem{transseries}} by Dahn and
Göring~{\cite{DG87}} and, independently, by Ecalle~{{\cite{Ec92,Ec16}}}.
Transseries are a natural generalization of the above definition of Hardy's
logarithmico-exponential functions, by also allowing for infinite sums (modulo
suitable precautions to ensure that such sums make sense). One example of a
transseries is
\[ f = \mathe^{\mathe^x + 2 \frac{\mathe^x}{x} + 6 \frac{\mathe^x}{x^2} +
   \cdots} + \mathe^{\mathe^{\mathpi \log \log x} - \sqrt{\log x}} - \sqrt{7}
   + \frac{1}{\log x} + \frac{1}{(\log x)^2} + \cdots . \]
In particular, all transseries can be written as generalized series $f =
\sum_{\mathfrak{m} \in \mathfrak{T}} f_{\mathfrak{m}} \mathfrak{m}$ with real
coefficients $f_{\mathfrak{m}} \in \mathbb{R}$ and whose (trans)monomials
$\mathfrak{m} \in \mathfrak{T}$ are exponentials of other (generally
``simpler'') transseries. The support $\tmop{supp} f \assign \{ \mathfrak{m}
\in \mathfrak{T} \suchthat f_{\mathfrak{m}} \neq 0 \}$ of such a series should
be {\tmem{well based}} in the sense that it should be well ordered for the
opposite ordering of the natural ordering $\preccurlyeq$ on the group of
transmonomials $\mathfrak{T}$. The precise definition of a transseries depends
on further technical requirements on the allowed supports. But for all
reasonable choices, ``the'' resulting field $\mathbb{T}$ of transseries
possesses a lot of closure properties: it is ordered and closed under
derivation, composition, integration, and functional
inversion~{\cite{Ec92,vdH:phd,vdDMM01}}; it also satisfies an intermediate
value property for differential polynomials~{\cite{vdH:ln,vdH:hivp}}.

It turns out that surreal numbers and transseries are similar in many
respects: both~$\mathbf{No}$ and $\mathbb{T}$ are real closed fields that are
closed under exponentiation and taking logarithms of positive elements.
Surreal numbers too can be represented uniquely as Hahn series
$\sum_{\mathfrak{m} \in \mathbf{Mo}} a_{\mathfrak{m}} \mathfrak{m}$ with real
coefficients $a_{\mathfrak{m}} \in \mathbb{R}$ and monomials in a suitable
multiplicative subgroup $\mathbf{Mo}$ of $\mathbf{No}^{>}$. Any transseries $f
\in \mathbb{T}$ actually naturally induces a surreal number $f (\omega) \in
\mathbf{No}$ by substituting $\omega$ for $x$ and the map $f \longmapsto f
(\omega)$ is injective~{\cite{BM19}}.

But there are also differences. Most importantly, elements of $\mathbb{T}$ can
be regarded as functions that can be derived and composed. Conversely, the
surreal numbers $\mathbf{No}$ come equipped with the Conway bracket. In fact,
it would be nice if any surreal number could naturally be regarded as the
value $f (\omega)$ of a unique transseries $f$ at $\omega$. Indeed, this would
allow us to transport the functional structure of $\mathbb{T}$ to the surreal
numbers. Conversely, we might equip the transseries with a Conway bracket and
other exotic operations on the surreal~numbers. The second author conjectured
the existence of such a correspondence between $\mathbf{No}$ and a suitably
generalized field of the transseries~{\cite[page~16]{vdH:ln}}; see
also~{\cite{vdH:icm}} for a more recent account.

Now we already observed that at least {\tmem{some}} surreal numbers $a \in
\mathbf{No}$ can be written uniquely as $a = f (\omega)$ for some transseries
$f \in \mathbb{T}$. Which numbers and what kind of functions do we miss? Since
a perfect correspondence would induce a Conway bracket on $\mathbb{T}$, it is
instructive to consider subsets $L, R \subseteq \mathbb{T}$ with $L < R$ and
examine which natural growth orders might fit between $L$ and $R$.

One obvious problem with ordinary transseries is that there is no
transseries that grows faster than all iterated exponentials $x, \mathe^x,
\mathe^{\mathe^x}, \ldots$. Consequently, there exists no transseries $f \in
\mathbb{T}$ with $f (\omega) = \{ \omega, \mathe^{\omega},
\mathe^{\mathe^{\omega}}, \ldots | \nosymbol \}$. A natural candidate for a
function that grows faster than any iterated exponential is the first
{\tmem{hyperexponential}} $E_{\omega}$, which satisfies the functional
equation
\[ E_{\omega} (x + 1) = \exp E_{\omega} (x) . \]
It was shown by Kneser~{\cite{Kn49}} that this equation actually has a real
analytic solution on~$\mathbb{R}^{>}$. A natural hyperexponential $E_{\omega}$
on $\mathbf{No}^{>, \succ} \assign \{ c \in \mathbf{No} \suchthat c
>\mathbb{R} \}$ was constructed more recently in~{\cite{BvdHM:surhyp}}. In
particular, $E_{\omega} (\omega) = \{ \omega, \mathe^{\omega},
\mathe^{\mathe^{\omega}}, \ldots | \nosymbol \}$.

More generally, one can formally introduce the sequence
$(E_{\alpha})_{\alpha \in \mathbf{On}}$ of {\tmem{hyperexponentials}} of
arbitrary strengths $\alpha$, together with the sequence $(L_{\alpha})_{\alpha
\in \mathbf{On}}$ of their functional inverses, called
{\tmem{hyperlogarithms}}. Each $E_{\omega^n}$ with $n \in \mathbb{N}^{>}$
satisfies the equation
\[ E_{\omega^n} (x + 1) = E_{\omega^{n - 1}} (E_{\omega^n} (x)) \]
and there again exist real analytic solutions to this
equation~{\cite{Schm01}}. The function $E_{\omega^{\omega}}$ does not satisfy
any natural functional equation, but we have the following infinite product
formula for the derivative of every hyperlogarithm~$L_{\alpha}$:
\[ L_{\alpha}' (x) = \prod_{\beta < \alpha} \frac{1}{L_{\beta} (x)} . \]
We showed in~{\cite{vdH:hypno}} how to define $E_{\alpha} (a)$ and $L_{\alpha}
(a)$ for any $\alpha \in \mathbf{On}$ and $a \in \mathbf{No}^{>, \succ}$.

The traditional field $\mathbb{T}$ of transseries is not closed under
hyperexponentials and hyperlogarithms, but it is possible to define
generalized fields of {\tmem{hyperseries}} that do enjoy this additional
closure property. Hyperserial grow rates were studied from a formal point of
view in~{\cite{Ec92,Ec16}}. The first systematic construction of hyperserial
fields of strength $\alpha < \omega^{\omega}$ is due to
Schmeling~{\cite{Schm01}}. In this paper, we will rely on the more recent
constructions from~{\cite{vdH:loghyp,BvdHK:hyp}} that are fully general. In
particular, the surreal numbers $\mathbf{No}$ form a hyperserial field in the
sense of~{\cite{BvdHK:hyp}}, when equipped with the hyperexponentials and
hyperlogarithms from~{\cite{vdH:hypno}}.

A less obvious problematic cut $L < R$ in the field of transseries
$\mathbb{T}$ arises by taking
\begin{eqnarray*}
  L & = & \left\{ \sqrt{x}, \sqrt{x} + \mathe^{\sqrt{\log x}}, \sqrt{x} +
  \mathe^{\sqrt{\log x} + \mathe^{\sqrt{\log \log x}}}, \ldots \right\}\\
  R & = & \left\{ 2 \sqrt{x}, \sqrt{x} + \mathe^{2 \sqrt{\log x}}, \sqrt{x} +
  \mathe^{\sqrt{\log x} + \mathe^{2 \sqrt{\log \log x}}}, \ldots \right\} .
\end{eqnarray*}
Here again, there exists no transseries $f \in \mathbb{T}$ with $L < f < R$.
This cut has actually a natural origin, since any ``tame'' solution of the
functional equation
\begin{equation}
  f (x) = \sqrt{x} + \mathe^{f (\log x)} \label{funceq-f-example}
\end{equation}
lies in this cut. What is missing here is a suitable notion of ``nested
transseries'' that encompasses expressions like
\begin{equation}
  f = \sqrt{x} + \mathe^{\sqrt{\log x} + \mathe^{\sqrt{\log \log x} +
  \mathe^{\udots}}} . \label{nested-f-example}
\end{equation}
This type of cuts were first considered in~{\cite[Section~2.7.1]{vdH:phd}}.
Subsequently, the second author and his former PhD student Schmeling developed
an abstract notion of generalized fields of
transseries~{\cite{vdH:gentr,Schm01}} that may contain nested transseries.
However, it turns out that expressions like~(\ref{nested-f-example}) are
ambiguous: one may construct fields of transseries that contain arbitrarily
large sets of pairwise distinct solutions to~(\ref{funceq-f-example}).

In order to investigate this ambiguity more closely, let us turn to the
surreal numbers. The above cut $L < R$ induces a cut $L (\omega) < R (\omega)$
in $\mathbf{No}$. Nested transseries solutions~$f$ to the functional
equation~(\ref{funceq-f-example}) should then give rise to surreal numbers $f
(\omega)$ with $L (\omega) < f (\omega) < R (\omega)$ and such that $f
(\omega) - \sqrt{\omega}, \log \left( f (\omega) - \sqrt{\omega} \right) -
\mathe^{\sqrt{\log \omega}}, \ldots$ are all monomials in~$\mathbf{Mo}$.
In~{\cite[Section~8]{BvdH19}}, we showed that those numbers $f (\omega)$
actually form a class $\mathbf{Ne}$ that is naturally parameterized by a
surreal number ($\mathbf{Ne}$ forms a so-called {\tmem{surreal
substructure}}). Here we note that analogue results hold when replacing
Gonshor's exponentiation by Conway's $\omega$\mbox{-}map $a \in \mathbf{No}
\longmapsto \omega^a$ (which generalizes Cantor's $\omega$-map when {$a \in
\mathbf{On}$}). This was already noted by Conway himself~{\cite[pages
34--36]{Con76}} and further worked out by
Lemire~{\cite{Lemire96,Lemire97,Lemire99}}. Section~\ref{section-nested-series} of
the present paper will be devoted to generalizing the result from
{\cite[Section~8]{BvdH19}} to nested hyperseries.

Besides the two above types of superexponential and nested cuts, no other
examples of ``cuts that cannot be filled'' come naturally to our mind. This
led the second author to conjecture~{\cite[page 16]{vdH:ln}} that there exists
a~field~$\mathbb{H}$ of suitably generalized hyperseries in $x$ such that each
surreal number can uniquely be represented as the value~$f (\omega)$ of a
hyperseries~{$f \in \mathbb{H}$} at $x = \omega$. In order to prove this
conjecture, quite some machinery has been developed since: a systematic theory
of surreal substructures~{\cite{BvdH19}}, sufficiently general notions of
hyperserial fields~{\cite{vdH:loghyp,BvdHK:hyp}}, and definitions of
$(E_{\alpha})_{\alpha \in \mathbf{No}}$ on the surreals that
give~$\mathbf{No}$ the structure of a hyperserial
field~{{\cite{BvdHM:surhyp,vdH:hypno}}}.

Now one characteristic property of generalized hyperseries in $\mathbb{H}$
should be that they can uniquely be described using suitable expressions that
involve $x$, real numbers, infinite summation, hyperlogarithms,
hyperexponentials, and a way to disambiguate nested expansions. The main goal
of this paper is to show that any surreal number can indeed be described
uniquely by a hyperserial expression of this kind in $\omega$. This
essentially solves the conjecture from {\cite[page 16]{vdH:ln}} by thinking of
hyperseries in $\mathbb{H}$ as surreal numbers in which we replaced $\omega$
by $x$. Of course, it remains desirable to give a formal construction of
$\mathbb{H}$ that does not involve surreal numbers and to specify the precise
kind of properties that our ``suitably generalized'' hyperseries should
possess. We intend to address this issue in a forthcoming paper.

Other work in progress concerns the definition of a derivation and a
composition on~$\mathbb{H}$. Now Berarducci and Mantova showed how to define a
derivation on $\mathbf{No}$ that is compatible with infinite summation and
exponentiation~{\cite{BM18}}. In~{\cite{vdH:bm,vdH:mt}}, it was shown that
there actually exist many such derivations and that they all satisfy the same
first order theory as the ordered differential field $\mathbb{T}$. However, as
pointed out in~{\cite{vdH:icm}}, Berarducci and Mantova's derivation does not
obey the chain rule with respect to $E_{\omega}$. The hyperserial derivation
that we propose to construct should not have this deficiency and therefore be
a better candidate for {\tmem{the}} derivation on $\mathbf{No}$ with respect
to $\omega$.

\subsection{Outline of our results and contributions}

In this paper, we will strongly rely on previous work
from~{\cite{BvdH19,vdH:loghyp,BvdHK:hyp,BvdHM:surhyp,vdH:hypno}}. The main
results from these previous papers will be recalled in
Sections~\ref{subsection-well-based-series}, \ref{subsection-surreal-numbers},
and~\ref{subsection-surreal-substructures}. For the sake of this introduction,
we start with a few brief reminders.

The field of {\tmem{logarithmic hyperseries}} $\mathbb{L}$ was defined and
studied in~{\cite{vdH:loghyp}}. It is a field of Hahn series
$\mathbb{L}=\mathbb{R} [[\mathfrak{L}]]$ in the sense of~{\cite{Hahn1907}}
that is equipped with a logarithm {$\log : \mathbb{L}^{>} \longrightarrow
\mathbb{L}$}, a derivation {$\partial : \mathbb{L} \longrightarrow
\mathbb{L}$}, and a composition $\circ : \mathbb{L} \times \mathbb{L}^{>,
\succ} \longrightarrow \mathbb{L}$. Moreover, for each ordinal $\alpha \in
\mathbf{On}$, it contains an element $\ell_{\alpha}$ such that
\begin{eqnarray*}
  \ell_1 \circ f & = & \log f \text{}\\
  \ell_{\omega^{\mu + 1}} \circ \ell_{\omega^{\mu}} & = & \ell_{\omega^{\mu +
  1}} - 1\\
  \log \ell_{\alpha}' & = & - \sum_{\beta < \alpha} \ell_{\beta + 1} .
\end{eqnarray*}
for all $f \in \mathbb{L}^{>, \succ}$ and all ordinals $\alpha, \mu$.
Moreover, if the Cantor normal form of $\alpha$ is given by $\alpha = \sum_{i
= 1}^p \omega^{\mu_i} n_i$ with $\mu_1 < \cdots < \mu_p$, then we have
\begin{eqnarray*}
  \ell_{\alpha} & = & \ell_{\omega^{\mu_1}}^{\circ n_1} \circ \cdots \circ
  \ell_{\omega^{\mu_p}}^{\circ n_p} .
\end{eqnarray*}
The derivation and composition on $\mathbb{L}$ satisfy the usual rules of
calculus and in particular a formal version of Taylor series expansions.

In~{\cite{BvdHK:hyp}}, Kaplan and the authors defined the concept of a
{\tmem{hyperserial field}} to be a~field $\mathbb{T}=\mathbb{R}
[[\mathfrak{T}]]$ of Hahn series with a logarithm {$\log : \mathbb{T}^{>}
\longrightarrow \mathbb{T}$} and a composition law {$\circ : \mathbb{L} \times
\mathbb{T}^{>, \succ} \longrightarrow \mathbb{T}$}, such that various natural
compatibility requirements are satisfied. For every ordinal $\alpha$, we then
define the {\tmem{hyperlogarithm}} $L_{\alpha}$ of strength $\alpha$ by
$L_{\alpha} : \mathbb{T}^{>, \succ} \longrightarrow \mathbb{T}^{>, \succ} ; {f
\longmapsto \ell_{\alpha} \circ f}$. We showed in {\cite{vdH:hypno}} how to
define bijective hyperlogarithms $L_{\alpha} : \mathbf{No}^{>, \succ}
\longrightarrow \mathbf{No}^{>, \succ}$ for which $\mathbf{No}$ has the
structure of a hyperserial field. For every ordinal $\alpha$, the functional
inverse $E_{\alpha} : \mathbf{No}^{>, \succ} \longrightarrow \mathbf{No}^{>,
\succ}$ of $L_{\alpha}$ is called the {\tmem{hyperexponential}} of strength
$\alpha$.

The main aim of this paper is to show that any surreal number $a \in
\mathbf{No}$ is not just an abstract hyperseries in the sense
of~{\cite{vdH:hypno}}, but that we can regard it as a hyperseries in $\omega$.
We will do this by constructing a suitable unambiguous description of $a$ in
terms of $\omega$, the real numbers, infinite summation, the
hyperexponentials, and the hyperlogarithms.

If $a = f (\omega)$ for some ordinary transseries $f$, then the idea would be
to expand $a$ as a linear combination of monomials, then to rewrite every
monomial as an exponential of a transseries, and finally to recursively expand
these new transseries. This process stops whenever we hit an iterated
logarithm of $\omega$.

In fact, this transserial expansion process works for any surreal number $a
\in \mathbf{No}$. However, besides the iterated logarithms (and exponentials)
of $\omega$, there exist other monomials $\mathfrak{a} \in \mathbf{Mo}^{\succ}
\assign \{ \mathfrak{m} \in \mathbf{Mo} \suchthat \mathfrak{m} \succ 1 \}$
such that $L_n (\mathfrak{a})$ is a monomial for all $n \in \mathbb{N}$. Such
monomials are said to be {\tmem{log\mbox{-}atomic}}. More generally, given
$\mu \in \mathbf{On}$, we say that $\mathfrak{a}$ is {\tmem{$L_{<
\omega^{\mu}}$-atomic}} if $L_{\alpha} (\mathfrak{a}) \in \mathbf{Mo}$ for all
$\alpha < \omega^{\mu}$. We write $\mathbf{Mo}_{\omega^{\mu}}$ for the set of
such numbers. If we wish to further expand an $L_{< \omega^{\mu}}$-atomic
monomial $\mathfrak{a}$ as a hyperseries, then it is natural to pick $\mu$
such that $\mathfrak{a}$ is not $L_{< \omega^{\mu + 1}}$-atomic, to
recursively expand {$b \assign L_{\omega^{\mu}}$}, and then to write
$\mathfrak{a}= E_{\omega^{\mu}} (b)$.

Unfortunately, the above idea is slightly too simple to be useful. In order to
expand monomials as hyperseries, we need something more technical. In
Section~\ref{section-well-nestedness}, we show that every non-trivial monomial
$\mathfrak{m} \in \mathbf{Mo} \setminus \{ 1 \}$ has a unique expansion of
exactly one of the two following forms:
\begin{equation}
  \mathfrak{m}= \mathe^{\psi}  (L_{\beta} (\omega))^{\iota}
  \label{eq-ex-st-expansion-1},
\end{equation}
where $\mathe^{\psi} \in \mathbf{Mo}$, $\iota \in \{ - 1, 1 \}$, and $\beta
\in \mathbf{On}$, with $\tmop{supp} \psi \succ \log (L_{\beta} (\omega))$; or
\begin{equation}
  \mathfrak{m}= \mathe^{\psi}  (L_{\beta} (E_{\alpha} (u)))^{\iota}
  \label{eq-ex-st-expansion-2},
\end{equation}
where $\mathe^{\psi} \in \mathbf{Mo}$, $\iota \in \{ - 1, 1 \}$, $\beta \in
\mathbf{On}, \alpha \in \omega^{\mathbf{On}}$ with $\beta \omega < \alpha$,
$\tmop{supp} \psi \succ \log (L_{\beta} (E_{\alpha} (u)))$, and where
$E_{\alpha} u$ lies in $\mathbf{Mo}_{\alpha} \setminus L_{< \alpha} 
\mathbf{Mo}_{\alpha \omega}$. Moreover, if $\alpha = 1$ then it is imposed
that $\psi = 0$, $\iota = 1$, and that $u$ cannot be written as $u = \varphi +
\varepsilon \mathfrak{b}$ where $\varphi \in \mathbf{No}$, $\varepsilon \in \{
- 1, 1 \}$, $\mathfrak{b} \in \mathbf{Mo}_{\omega}$, and~{$\mathfrak{b} \prec
\tmop{supp} \varphi$}.

After expanding $\mathfrak{m}$ in the above way, we may pursue with the
recursive expansions of $\psi$ and $u$ as hyperseries. Our next objective is
to investigate the shape of the recursive expansions that arise by doing so.
Indeed, already in the case of ordinary transseries, such recursive expansions
may give rise to nested expansions like
\begin{equation}
  \sqrt{\omega} + \mathe^{\sqrt{\log \omega} + \mathe^{\sqrt{\log \log \omega}
  + \mathe^{\udots}}} \label{well-nested-example}
\end{equation}
One may wonder whether it is also possible to obtain expansions like
\begin{equation}
  \sqrt{\omega} + \mathe^{\sqrt{\log \omega} + \mathe^{\sqrt{\log \log \omega}
  + \mathe^{\udots \ddots} + \log \log \log \omega} + \log \log \omega} + \log
  \omega . \label{ill-nested-example}
\end{equation}
Expansions of the forms~(\ref{well-nested-example})
and~(\ref{ill-nested-example}) are said to be {\tmem{well-nested}} and
{\tmem{ill-nested}}, respectively. The axiom {\tmstrong{T4}} for fields of
transseries in~{\cite{Schm01}} prohibits the existence of ill-nested
expansions. It was shown in~{\cite{BM18}} that $\mathbf{No}$ satisfies this
axiom {\tmstrong{T4}}.

The definition of hyperserial fields in~{\cite{vdH:hypno}} does not contain a
counterpart for the axiom~{\tmstrong{T4}}. The main goal of section 4 is to
generalize this property to hyperserial fields and prove the following
theorem:

\begin{theorem}
  \label{th-well-nested}Every surreal number is well-nested.
\end{theorem}

Now there exist surreal numbers for which the above recursive expansion
process leads to a nested expansion of the form~(\ref{well-nested-example}).
In {\cite[Section~8]{BvdH19}}, we proved that the class~$\mathbf{Ne}$ of such
numbers actually forms a {\tmem{surreal substructure}}. This means that
$(\mathbf{No}, \leqslant, \sqsubseteq)$ is isomorphic to $(\mathbf{Ne},
\leqslant, \sqsubseteq_{\mathbf{Ne}})$ for the restriction
$\sqsubseteq_{\mathbf{Ne}}$ of $\sqsubseteq$ to $\mathbf{Ne}$. In particular,
although the nested expansion~(\ref{well-nested-example}) is inherently
ambiguous, elements in $\mathbf{Ne}$ are naturally parameterized by surreal
numbers~in~$\mathbf{No}$.

The main goal of Section~\ref{section-nested-series} is to prove a hyperserial
analogue of the result from~{\cite[Section~8]{BvdH19}}. Now the
expansion~(\ref{well-nested-example}) can be described in terms of the
sequence $\sqrt{\omega}, \sqrt{\log \omega}, \sqrt{\log \log \omega}, \ldots$.
More generally, in Section~\ref{section-nested-series} we the define the
notion of a {\tmem{nested sequence}} in order to describe arbitrary nested
hyperserial expansions. Our main result is the following:

\begin{theorem}
  \label{th-nested-numbers}Any nested sequence $\Sigma$ induces a surreal
  substructure $\mathbf{Ne}$ of nested hyperseries.
\end{theorem}

In Section~\ref{section-numbers-as-hyperseries}, we reach the main goal of
this paper, which is to uniquely describe any surreal number as a generalized
hyperseries in $\omega$. This goal can be split up into two tasks. First of
all, we need to specify the hyperserial expansion process that we informally
described above and show that it indeed leads to a hyperserial expansion in
$\omega$, for any surreal number. This will be done in
Section~\ref{tree-expansion-sec}, where we will use labeled trees in order to
represent hyperserial expansions. Secondly, these trees may contain infinite
branches (also called paths) that correspond to nested numbers in the sense of
Section~\ref{section-nested-series}. By Theorem~\ref{th-nested-numbers}, any
such nested number can uniquely be identified using a surreal parameter. By
associating a surreal number to each infinite branch, this allows us to
construct a unique {\tmem{hyperserial description}} in $\omega$ for any
surreal number and prove our main result:

\begin{theorem}
  \label{th-hyperserial-representation}Every surreal number has a unique
  hyperserial description. Two numbers with the same hyperserial description
  are equal.
\end{theorem}

\section{Ordered fields of well-based
series}\label{subsection-well-based-series}

\subsection{Well-based series}

Let $(\mathfrak{M}, \times, 1, \prec)$ be a totally ordered (and possibly
class-sized) abelian group. We say that $\mathfrak{S} \subseteq \mathfrak{M}$
is {\tmem{well-based}} if it contains no infinite ascending chain
(equivalently, this means that $\mathfrak{S}$ is well-ordered for the opposite
ordering). We denote by $\mathbb{R} [[\mathfrak{M}]]$\label{autolab1} the
class of functions $f : \mathfrak{M} \longrightarrow \mathbb{R}$ whose support
\[ \tmop{supp} f \assign \{ \mathfrak{m} \in \mathfrak{M} \suchthat f
   (\mathfrak{m}) \neq 0 \} \text{\label{autolab2}} \]
is a {\tmem{well-based}}. The elements of $\mathfrak{M}$ are called monomials
and the elements in $\mathbb{R}^{\neq} \mathfrak{M}$ are called
{\tmem{terms}}{\index{term}}. We also define
\[ \tmop{term} f \assign \{ f_{\mathfrak{m}} \mathfrak{m} \suchthat
   \mathfrak{m} \in \tmop{supp} f \}, \text{\label{autolab3}} \]
and elements $\tau \in \tmop{term} f$ are called terms in $f$.

We see elements $f$ of $\mathbb{S}$ as formal {\tmem{well-based series}} $f =
\sum_{\mathfrak{m}} f_{\mathfrak{m}} \mathfrak{m}$ where $f_{\mathfrak{m}}
\assign f (\mathfrak{m}) \in \mathbb{R}$ for all $\mathfrak{m} \in
\mathfrak{M}$. If $\tmop{supp} f \neq \varnothing$, then $\mathfrak{d}_f
\assign \max \tmop{supp} f \in \mathfrak{M}$\label{autolab4} is called the
{\tmem{dominant monomial}}{\index{dominant monomial}} of~$f$. For
$\mathfrak{m} \in \mathfrak{M}$, we define $f_{\succ \mathfrak{m}} \assign
\sum_{\mathfrak{n} \succ \mathfrak{m}} f_{\mathfrak{n}}
\mathfrak{n}$\label{autolab5} and $f_{\succ} \assign f_{\succ
1}$\label{autolab6}. For $f, g \in \mathbb{S}$, we sometimes write $f + g = f
\pplus g$\label{autolab7} if $\tmop{supp} g \prec f$. We say that a series $g
\in \mathbb{S}$ is a {\tmem{truncation}}{\index{truncation}} of $f$ and we
write $g \trianglelefteqslant f$\label{autolab8} if $\tmop{supp} (f - g) \succ
g$. The relation $\trianglelefteqslant$ is a well-founded partial order on
$\mathbb{S}$ with minimum $0$.

By {\cite{Hahn1907}}, the class $\mathbb{S}$ is field for the pointwise sum
\[ (f + g) \assign \sum_{\mathfrak{m}} (f_{\mathfrak{m}} + g_{\mathfrak{m}})
   \mathfrak{m}, \]
and the Cauchy product
\[ fg \assign \sum_{\mathfrak{m}} \left(
   \sum_{\mathfrak{u}\mathfrak{v}=\mathfrak{m}} f_{\mathfrak{u}}
   g_{\mathfrak{v}} \right) \mathfrak{m}, \]
where each sum $\sum_{\mathfrak{u}\mathfrak{v}=\mathfrak{m}} f_{\mathfrak{u}}
g_{\mathfrak{v}}$ is finite. The class $\mathbb{S}$ is actually an ordered
field, whose positive cone $\mathbb{S}^{>} = \{ f \in \mathbb{S} \suchthat f >
0 \}$ is defined by
\[ \text{$\mathbb{S}^{>} \assign \{ f \in \mathbb{S} \suchthat f \neq 0 \wedge
   f_{\mathfrak{d}_f} > 0 \}$} . \]
The ordered group $(\mathfrak{M}, \times, \prec)$ is naturally embedded into
$(\mathbb{S}^{>}, \times, <)$.

The relations $\prec$ and $\preccurlyeq$ on $\mathfrak{M}$ extend to
$\mathbb{S}$ by\label{autolab9} \label{autolab10}
\begin{eqnarray*}
  f \prec g & \Longleftrightarrow & \mathbb{R}^{>}  | f | < | g |\\
  f \preccurlyeq g & \Longleftrightarrow & \nobracket \exists r \in
  \mathbb{R}^{>} \nobracket, \hspace{1.2em} | f | \leqslant r | g | .
\end{eqnarray*}
We also write $f \asymp g$\label{autolab11} whenever $f \preccurlyeq g$ and $g
\preccurlyeq f$. If $f, g$ are non-zero, then $f \prec g$ ({\tmabbr{resp.}} $f
\preccurlyeq g$, {\tmabbr{resp.}} $f \asymp g$) if and only if $\mathfrak{d}_f
\prec \mathfrak{d}_g$ ({\tmabbr{resp.}} $\mathfrak{d}_f \preccurlyeq
\mathfrak{d}_g$, {\tmabbr{resp.}} $\mathfrak{d}_f =\mathfrak{d}_g$).

We finally define\label{autolab12} \label{autolab13} \label{autolab14}
\begin{eqnarray*}
  \mathbb{S}_{\succ} & \assign & \{ f \in \mathbb{S} \suchthat \tmop{supp} f
  \subseteq \mathfrak{M}^{\succ} \}\\
  \mathbb{S}^{\prec} & \assign & \{ f \in \mathbb{S} \suchthat \tmop{supp} f
  \subseteq \mathfrak{M}^{\prec} \} = \{ f \in \mathbb{S} \suchthat f \prec 1
  \}, \text{\quad and}\\
  \mathbb{S}^{>, \succ} & \assign & \{ f \in \mathbb{S} \suchthat f
  >\mathbb{R} \} = \{ f \in \mathbb{S} \suchthat f \geqslant 0 \wedge f \succ
  1 \} .
\end{eqnarray*}
Series in $\mathbb{S}_{\succ}$, $\mathbb{S}^{\prec}$ and $\mathbb{S}^{>,
\succ}$ are respectively called {\tmem{purely large}}, {\tmem{infinitesimal}},
and {\tmem{positive infinite}}.{\index{purely large
series}}{\index{infinitesimal series}}{\index{positive infinite series}}

\subsection{Well-based families}

If $(f_i)_{i \in I}$ is a family in $\mathbb{S}$, then we say that $(f_i)_{i
\in I}$ is {\tmem{well-based}}{\index{well-based family}} if
\begin{enumerateroman}
  \item $\bigcup_{i \in I} \tmop{supp} f_i$ is well-based, and
  
  \item $\{ i \in I \suchthat \mathfrak{m} \in \tmop{supp} f_i \}$ is finite
  for all $\mathfrak{m} \in \mathfrak{M}$.
\end{enumerateroman}
Then we may define the sum $\sum_{i \in I} f_i$ of $(f_i)_{i \in I}$ as the
series
\[ \sum_{i \in I} f_i \assign \sum_{\mathfrak{m}} \left( \sum_{i \in I}
   (f_i)_{\mathfrak{m}} \right) \mathfrak{m}. \]
If $\mathbb{U}=\mathbb{R} [[\mathfrak{N}]]$ is another field of well-based
series and $\Psi : \mathbb{S} \longrightarrow \mathbb{U}$ is
$\mathbb{R}$-linear, then we say that $\Psi$ is {\tmem{strongly linear}} if
for every well-based family $(f_i)_{i \in I}$ in $\mathbb{S}$, the family
$(\Psi (f_i))_{i \in I}$ in~$\mathbb{U}$ is well-based, with
\[ \Psi \left( \sum_{i \in I} f_i \right) = \sum_{i \in I} \Psi (f_i) . \]

\subsection{Logarithmic hyperseries}

The field $\mathbb{L}$\label{autolab15} of {\tmem{logarithmic
hyperseries}}{\index{logarithmic hyperseries}} plays an important role in the
theory of hyperseries. Let us briefly recall its definition and its most
prominent properties from {\cite{vdH:loghyp}}.

Let $\alpha$ be an ordinal. For each~$\gamma < \alpha$, we introduce the
formal hyperlogarithm {$\ell_{\gamma} \assign L_{\gamma} x$} and define
$\mathfrak{L}_{< \alpha}$\label{autolab16} to be the group of formal power
products {$\mathfrak{l}= \prod_{\gamma < \alpha}
\ell_{\gamma}^{\mathfrak{l}_{\gamma}}$} with $\mathfrak{l}_{\gamma} \in
\mathbb{R}$. This group comes with a monomial ordering~$\succ$ that is defined
by
\begin{eqnarray*}
  \mathfrak{l} \succ 1 & \Longleftrightarrow & \mathfrak{l}_{\min \{ \gamma <
  \alpha \suchthat \mathfrak{l}_{\gamma} \neq 0 \}} > 0.
\end{eqnarray*}
By what precedes, $\mathbb{L}_{< \alpha} \assign \mathbb{R} [[\mathfrak{L}_{<
\alpha}]]$\label{autolab17} is an ordered field of well-based series. If
$\alpha, \beta$ are ordinals with $\beta < \alpha$, then we define
$\mathfrak{L}_{[\beta, \alpha)}$ to be the subgroup of $\mathfrak{L}_{<
\alpha}$ of monomials $\mathfrak{l}$ with $\mathfrak{l}_{\gamma} = 0$ whenever
$\gamma < \beta$. As in {\cite{vdH:loghyp}}, we define
\begin{eqnarray*}
  \mathbb{L}_{[\beta, \alpha)} & \assign & \mathbb{R} [[\mathfrak{L}_{[\beta,
  \alpha)}]]\\
  \mathfrak{L} & \assign & \bigcup_{\alpha \in \mathbf{On}} \mathfrak{L}_{<
  \alpha}\\
  \mathbb{L} & \assign & \mathbb{R} [[\mathfrak{L}]] .
\end{eqnarray*}
We have natural inclusions $\mathfrak{L}_{[\beta, \alpha)} \subseteq
\mathfrak{L}_{< \alpha} \subset \mathfrak{L}$, hence natural inclusions
$\mathbb{L}_{[\beta, \alpha)} \subseteq \mathbb{L}_{< \alpha} \subset
\mathbb{L}$.

The field $\mathbb{L}_{< \alpha}$ is equipped with a derivation $\partial :
\mathbb{L}_{< \alpha} \longrightarrow \mathbb{L}_{< \alpha}$ which satisfies
the Leibniz rule and which is strongly linear: for each $\gamma < \alpha$, we
first define the logarithmic derivative of $\ell_{\gamma}$ by
$\ell_{\gamma}^{\dag} \assign \prod_{\iota \leqslant \gamma} \ell_{\iota}^{-
1} \in \mathfrak{L}_{< \alpha}$. The derivative of a logarithmic hypermonomial
$\mathfrak{l} \in \mathfrak{L}_{< \alpha}$ is next defined by
\begin{eqnarray*}
  \partial \mathfrak{l} & \assign & \left( \sum_{\gamma < \alpha}
  \mathfrak{l}_{\gamma} \ell_{\gamma}^{\dag} \right) \mathfrak{l}.
\end{eqnarray*}
Finally, this definition extends to $\mathbb{L}_{< \alpha}$ by strong
linearity. Note that $\partial \ell_{\gamma} = \frac{1}{\prod_{\iota < \gamma}
\ell_{\iota}}$ for all $\gamma < \alpha$. For $f \in \mathbb{L}_{< \alpha}$
and $k \in \mathbb{N}$, we will sometimes write $f^{(k)} \assign \partial^k
f$.

Assume that $\alpha = \omega^{\nu}$ for a certain ordinal $\nu$. Then the
field $\mathbb{L}_{< \alpha}$ is also equipped with a~composition $\circ :
\mathbb{L}_{< \alpha} \times \mathbb{L}_{< \alpha}^{>, \succ} \longrightarrow
\mathbb{L}_{< \alpha}$ that satisfies:
\begin{itemizedot}
  \item For $g \in \mathbb{L}_{< \alpha}^{>, \succ}$, the map $\mathbb{L}_{<
  \alpha} \longrightarrow \mathbb{L}_{< \alpha} ; f \longmapsto f \circ g$ is
  a strongly linear embedding {\cite[Lemma~6.6]{vdH:loghyp}}.
  
  \item For $f \in \mathbb{L}_{< \alpha}$ and $g, h \in \mathbb{L}_{<
  \alpha}^{>, \succ}$, we have $g \circ h \in \mathbb{L}_{< \alpha}^{>,
  \succ}$ and $f \circ (g \circ h) = (f \circ g) \circ h$
  {\cite[Proposition~7.14]{vdH:loghyp}}.
  
  \item For $g \in \mathbb{L}_{< \alpha}^{>, \succ}$ and successor ordinals
  $\mu < \nu$, we have $\ell_{\omega^{\mu}} \circ \ell_{\omega^{\mu_-}} =
  \ell_{\omega^{\mu}} - 1$ {\cite[Lemma~5.6]{vdH:loghyp}}.
\end{itemizedot}
The same properties hold for the composition $\circ : \mathbb{L} \times
\mathbb{L}^{>, \succ} \longrightarrow \mathbb{L}$, when $\alpha$ is replaced
by $\mathbf{On}$. For $\gamma < \alpha$, the map $\mathbb{L}_{< \alpha}
\longrightarrow \mathbb{L}_{< \alpha} ; f \longmapsto f \circ \ell_{\gamma}$
is injective, with range $\mathbb{L}_{[\gamma, \alpha)}$
{\cite[Lemma~5.11]{vdH:loghyp}}. For $g \in \mathbb{L}_{[\gamma, \alpha)}$, we
define $g^{\mathord{\uparrow} \gamma}$\label{autolab18} to be the unique
series in $\mathbb{L}_{< \alpha}$ with~$g^{\mathord{\uparrow} \gamma} \circ
\ell_{\gamma} = g$.

\section{Surreal numbers as a hyperserial
field}\label{subsection-surreal-numbers}

\subsection{Surreal numbers}

Following {\cite{Gon86}}, we define $\mathbf{No}$ as the class of sequences
\[ a : \alpha \longmapsto \{ - 1, 1 \} \]
of ``signs'' $- 1, + 1$ indexed by arbitrary ordinals $\alpha \in
\mathbf{On}$\label{autolab19}. We will write $\tmop{dom} a \in \mathbf{On}$
for the domain of such a sequence and $a [\beta] \in \{ - 1, 1 \}$ for its
value at $\beta \in \tmop{dom} a$. Given sign sequences $a$ and $b$, we
define\label{autolab20}
\begin{eqnarray*}
  a \sqsubseteq b & \Longleftrightarrow & \tmop{dom} a \subseteq \tmop{dom} b
  \wedge \left( \forall \beta \in \tmop{dom} a, \hspace{1.2em} a [\beta] = b
  [\beta] \right)
\end{eqnarray*}
Conway showed how to define an ordering, an addition, and a multiplication on
$\mathbf{No}$ that give $\mathbf{No}$ the structure of a real closed
field~{\cite{Con76}}. See {\cite[Section~2]{BvdH19}} for more details about
the interaction between $\sqsubseteq$ and the ordered field structure of
$\mathbf{No}$. By {\cite[Theorem~21]{Con76}}, there is a natural isomorphism
between~$\mathbf{No}$ and the ordered field of well-based series $\mathbb{R}
[[\mathbf{Mo}]]$, where $\mathbf{Mo}$ is a certain subgroup
of~$(\mathbf{No}^{>}, \times, <)$. We will identify those two fields and thus
regard $\mathbf{No}$ as a field of well-based series with monomials in
$\mathbf{Mo}$.

The partial order $(\mathbf{No}, \sqsubseteq)$ contains an isomorphic copy of
$(\mathbf{On}, \in)$ obtained by identifying each ordinal $\alpha$ with the
constant sequence $(1)_{\beta < \alpha}$ of length $\alpha$. We will write
$\tmmathbf{\nu} \leqslant \mathbf{On}$ to specify that $\tmmathbf{\nu}$ is
either an ordinal or the class of ordinals. The ordinal $\omega$, seen as
a~surreal number, is the simplest element, or $\sqsubseteq$-minimum, of the
class $\mathbf{No}^{>, \succ}$.

For $\alpha, \beta \in \mathbf{On}$, we write $\alpha \dotplus
\beta$\label{autolab21} and $\alpha \dottimes \beta$\label{autolab22} for the
non-commutative ordinal sum and product of $\alpha$ and $\beta$, as defined by
Cantor. The surreal sum and product $\alpha + \beta$ and $\alpha \beta$
coincide with the commutative Hessenberg sum and product of ordinals. In
general, we therefore have $\alpha + \beta \neq \alpha \dotplus \beta$ and
$\alpha \beta \neq \alpha \dottimes \beta$.

For $\gamma \in \mathbf{On}$, we write $\omega^{\gamma}$\label{autolab23} for
the ordinal exponentiation of base $\omega$ at $\gamma$. Gonshor also defined
an exponential function on $\mathbf{No}$ with range $\mathbf{No}^{>}$. One
should not confuse $\omega^{\gamma}$ with $\exp (\gamma \log \omega)$, which
yields a different number, in general. We define
\[ \omega^{\mathbf{On}} \assign \{ \omega^{\gamma} \suchthat \gamma \in
   \mathbf{On} \}, \]
Recall that every ordinal $\gamma$ has a unique Cantor normal
form{\index{Cantor normal form}}
\[ \gamma = \omega^{\eta_1} n_1 + \cdots + \omega^{\eta_r} n_r, \]
where $r \in \mathbb{N}$, $n_1, \ldots, n_r \in \mathbb{N}^{> 0}$ and $\eta_1,
\ldots, \eta_r \in \mathbf{On}$ with $\eta_1 > \cdots > \eta_r$. The ordinals
$\eta_i$ are called the {\tmem{exponents}}{\index{exponent of an ordinal}} of
$\gamma$ and the integers $n_i$ its {\tmem{coefficients}}{\index{coefficient
of an ordinal}}. We write $\rho \ll \sigma$\label{autolab24} ({\tmabbr{resp.}}
$\rho \lleq \sigma$\label{autolab25}) if each exponent $\eta_i$ of the Cantor
normal form of $\sigma$ satisfies $\rho \prec \omega^{\eta_i}$
({\tmabbr{resp.}} $\rho \preccurlyeq \omega^{\eta_i}$).

If $\gamma, \beta$ are ordinals, then we write $\gamma \prec
\beta$\label{autolab26} if $\gamma \mathbb{N}< \beta$, we write $\gamma
\preccurlyeq \beta$\label{autolab27} if there exists an~$n \in \mathbb{N}$
with $\gamma \leqslant \beta n$, and we write $\gamma \asymp
\beta$\label{autolab28} if both $\gamma \preccurlyeq \beta$ and $\gamma
\preccurlyeq \beta$ hold. The relation $\preccurlyeq$ is a quasi-order
on~$\textbf{$\mathbf{On}$}$. For $\eta, \beta, \gamma \in \mathbf{On}$ with
$\beta \ggeq \omega^{\eta}$ and $\gamma \preccurlyeq \omega^{\eta}$, we have
$\beta + \gamma = \beta \dotplus \gamma$. In particular, we have $\gamma + 1 =
\gamma \dotplus 1$ for all $\gamma \in \mathbf{On}$.

If $\mu \in \mathbf{On}$ is a successor, then we define $\mu_-$ to be the
unique ordinal with $\mu = \mu_- + 1$.\label{autolab29} We also define $\mu_-
\assign \mu$ if $\mu$ is a limit. Similarly, if $\alpha = \omega^{\mu}$, then
we write $\alpha_{/ \omega} \assign \omega^{\mu_-}$\label{autolab30}.

\subsection{Hyperserial structure on $\mathbf{No}$}

We already noted that Gonshor constructed an exponential and a logarithm on
$\mathbf{No}$ and~$\mathbf{No}^{>}$, respectively. We defined hyperexponential
and hyperlogarithmic functions of all strengths on $\mathbf{No}^{>, \succ}$
in~{\cite{vdH:hypno}}. In fact, we showed~{\cite[Theorem~1.1]{vdH:hypno}} how
to construct a~composition law $\circ : \mathbb{L} \times \mathbf{No}^{>,
\succ} \longrightarrow \mathbf{No}$\label{autolab31} with the following
properties:

\begin{descriptioncompact}
  \item[$\textbf{C1}$] \label{c1}For $f \in \mathbb{L}$, $g \in \mathbb{L}^{>,
  \succ}$ and $a \in \mathbf{No}^{>, \succ}$, we have $g \circ a \in
  \mathbf{No}^{>, \succ}$ and
  \[ f \circ (g \circ a) = (f \circ g) \circ a. \]
  \item[$\textbf{C2}$] \label{c2}For $a \in \mathbf{No}^{>, \succ}$, the
  function $\mathbb{L} \longrightarrow \mathbf{No} ; f \longmapsto f \circ a$
  is a strongly linear field morphism.
  
  \item[$\textbf{C3}$] \label{c3}For $f \in \mathbb{L}$, $a \in
  \mathbf{No}^{>, \succ}$ and $\delta \in \mathbf{No}$ with $\delta \prec a$,
  we have
  \[ f \circ (a + \delta) = \sum_{k \in \mathbb{N}} \frac{f^{(k)} \circ a}{k!}
     \delta^k . \]
  \item[$\textbf{C4}$] \label{c4}For $\gamma \in \mathbf{On}$ and $a, b \in
  \mathbf{No}^{>, \succ}$ with $a < b$, we have $\ell_{\gamma} \circ a <
  \ell_{\gamma} \circ b$.
  
  \item[$\textbf{C5}$] \label{c5}For $\gamma \in \mathbf{On}$ and $a \in
  \mathbf{No}^{>, \succ}$, there is $b \in \mathbf{No}^{>, \succ}$ with $a =
  \ell_{\gamma} \circ b$.
\end{descriptioncompact}

Note that the composition law on $\mathbb{L}$ also satisfies \hyperref[c1]{$\mathbf{C1}$} to
\hyperref[c4]{$\mathbf{C4}$} (but not \hyperref[c5]{$\mathbf{C5}$}), with each occurrence of $\mathbf{No}$ being
replaced by $\mathbb{L}$.

\subsection{Hyperlogarithms}

For $\gamma \in \mathbf{On}$, we write $L_{\gamma}$ for the function
$\mathbf{No}^{>, \succ} \longrightarrow \mathbf{No}^{>, \succ} ; a \longmapsto
\ell_{\gamma} \circ a$, called the {\tmem{hyperlogarithm of strength}}
$\gamma${\index{hyperlogarithm}}. By \hyperref[c4]{$\mathbf{C4}$} and \hyperref[c5]{$\mathbf{C5}$}, this is a
strictly increasing bijection. We sometimes write $L_{\gamma} a \assign
L_{\gamma} (a)$ for $a \in \mathbf{No}^{>, \succ}$. We write $E_{\gamma}$ for
the functional inverse of $L_{\gamma}$, called the {\tmem{hyperexponential of
strength}} $\gamma${\index{hyperexponential}}.

For $\gamma, \rho$ with $\rho \lleq \gamma$, the relation $\ell_{\gamma +
\rho} = \ell_{\rho} \circ \ell_{\gamma}$ in $\mathbb{L}$, combined with
\hyperref[c3]{$\mathbf{C3}$}, yields
\begin{equation}
  \nobracket \forall a \in \mathbf{No}^{>, \succ} \nobracket, \hspace{1.2em}
  L_{\gamma + \rho} a = L_{\gamma} L_{\rho} a \label{hyperlog-def},
\end{equation}
For $\eta \in \mathbf{On}$, the relation $\ell_{\omega^{\eta + 1}} \circ
\ell_{\omega^{\eta}} = \ell_{\omega^{\eta + 1}} - 1$ in $\mathbb{L}$, combined
with \hyperref[c3]{$\mathbf{C3}$}, yields
\begin{equation}
  \nobracket \forall a \in \mathbf{No}^{>, \succ} \nobracket, \hspace{1.2em}
  L_{\omega^{\eta + 1}} (L_{\omega^{\eta}} (a)) = L_{\omega^{\eta + 1}} (a) -
  1 \label{functional-equation},
\end{equation}
and we call this relation the {\tmem{functional equation}}{\index{functional
equation}} for $L_{\omega^{\eta + 1}}$.

Let $a \in \mathbf{No}^{>}$ and write $r_a \assign a_{\mathfrak{d}_a}$ for the
coefficient in $\mathfrak{d}_a$ in the Hahn series representation of $a$.
There is a unique infinitesimal number $\varepsilon_a$ with $a = r_a
\mathfrak{d}_a  (1 + \varepsilon_a)$. We write $\log_{\mathbb{R}}$ for the
natural logarithm on $\mathbb{R}^{>} \subset \mathbf{No}$. The function
defined by
\begin{equation}
  \log a \assign L_1 (\mathfrak{d}_a) + \log_{\mathbb{R}} r_a + \sum_{k \in
  \mathbb{N}} \frac{(-1)^k}{k + 1} \varepsilon_a^{k + 1} \label{eq-log}
\end{equation}
is an ordered group embedding
{$(\mathbf{No}^{>}, +,<) \longrightarrow
(\mathbf{No}, +,<)$} called the {\tmem{logarithm}}{\index{logarithm}} on $\mathbf{No}^{>}$. It extends the partial function $L_1$. It also coincides with the logarithm
on $\mathbf{No}^{>}$ that was defined by Gonshor.

\subsection{Atomicity}

Given $\tmmathbf{\mu} \leqslant \mathbf{On}$, we write
$\mathbf{Mo}_{\omega^{\tmmathbf{\mu}}}$ for the class of numbers $a \in
\mathbf{No}^{>, \succ}$ with $L_{\gamma} a \in \mathbf{Mo}^{\succ}$ for all
$\gamma < \omega^{\tmmathbf{\mu}}$. Those numbers are said to be $L_{<
\omega^{\tmmathbf{\mu}}}${\tmem{-atomic}}{\index{$L_{<
\omega^{\tmmathbf{\mu}}}$-atomic number}} and they play an important role in
this paper. Note that $\mathbf{Mo}_1 = \mathbf{Mo}^{\succ}$ and
\[ L_{\omega^{\eta}}  \mathbf{Mo}_{\smash{\omega^{\eta + 1}}} =
   \mathbf{Mo}_{\smash{\omega^{\eta + 1}}} \]
for all $\eta \in \mathbf{On}$, in view of (\ref{hyperlog-def}). There is a
unique $L_{< \mathbf{On}}$-atomic number {\cite[Proposition~6.20]{vdH:hypno}},
which is the simplest positive infinite number $\omega$.

Each hyperlogarithmic function $L_{\omega^{\eta}}$ with $\eta > 0$ is
essentially determined by its restriction to $\mathbf{Mo}_{\omega^{\eta}}$,
through a generalization of~\Cref{eq-log}. More precisely, for $a \in
\mathbf{No}^{>, \succ}$, there exist~$\gamma < \omega^{\eta}$ and
$\mathfrak{a} \in \mathbf{Mo}_{\omega^{\eta}}$ with $\delta \assign L_{\gamma}
(a) - L_{\gamma} (\mathfrak{a}) \prec L_{\gamma} (a)$. Moreover, the family
{$\left( \left( \left( \ell_{\omega^{\eta}}^{\mathord{\uparrow \gamma}}
\right)^{(k)} \circ L_{\gamma} (\mathfrak{a}) \right) \delta^k \right)_{k \in
\mathbb{N}^{>}}$} is well-based, and the hyperlogarithm $L_{\omega^{\eta}}
(a)$ is given by
\begin{equation}
  L_{\omega^{\eta}} (a) = L_{\omega^{\eta}} (\mathfrak{a}) + \sum_{k \in
  \mathbb{N}^{>}} \frac{(\ell_{\omega^{\eta}}^{\uparrow \gamma})^{(k)} \circ
  L_{\gamma} (\mathfrak{a})}{k!} \delta^k . \label{eq-hyperlog-Taylor}
\end{equation}
\subsection{Hyperexponentiation}\label{subsubsection-hyperexponentiation}

\begin{definition}
  {\tmem{{\cite[Definition~6.10]{BvdHK:hyp}}}} We say that $\varphi \in
  \mathbf{No}^{>, \succ}$ is {\tmem{$\textbf{1}${\tmstrong{-truncated}}}} if
  {$\tmop{supp} \varphi \succ 1$}, {\tmabbr{i.e.}} if $\varphi$ is positive
  and purely large. For $0 < \eta \in \mathbf{On}$, we say that $\varphi \in
  \mathbf{No}^{>, \succ}$ is
  {{\tmstrong{{\tmem{$\omega^{\eta}$\mbox{-}truncated}}}}}{\index{$\beta$-truncated
  series}} if
  \[ \nobracket \forall \mathfrak{m} \in \tmop{supp} \varphi_{\prec}
     \nobracket, \nobracket \forall \gamma < \omega^{\eta}
     \nobracket, \varphi < \ell_{\omega^{\eta}}^{\mathord{\uparrow}
     \gamma} \circ \mathfrak{m}^{- 1} . \]
\end{definition}

One sees that when $E_{\omega^{\eta}} (\varphi)$ is defined, the number $\varphi$ is
  $\omega^{\eta}$-truncated if and only if $\tmop{supp} \varphi \succ 1 /
  L_{\gamma} (E_{\omega^{\eta}} (\varphi))$, for all $\gamma < \omega^{\eta}$.

Given $\beta = \omega^{\eta}$ with $\eta \in \mathbf{On}$, we write
$\mathbf{No}_{\succ, \beta}$\label{autolab32} for the class of
$\beta$-truncated numbers. Note that $\mathbf{No}_{\succ, 1} =
\mathbf{No}_{\succ} \cap \mathbf{No}^{>, \succ}$. We will sometimes write
$E_{\beta} (\varphi) \backassign E_{\beta}^{\varphi}$ when $\varphi \in
\mathbf{No}_{\succ, \beta}$. For $a \in \mathbf{No}^{>, \succ}$, there is a
unique $\trianglelefteqslant$-maximal truncation $\sharp_{\beta} (a)$ of $a$
which is $\beta$-truncated. By {\cite[Proposition~7.17]{BvdHK:hyp}}, the
classes
\begin{equation}
  \left\{ b \in a + \mathbf{No}^{\prec} \suchthat b = a \vee \left( \exists
  \gamma < \beta, b < \ell_{\beta}^{\mathord{\uparrow} \gamma} \circ | a - b
  |^{- 1} \right) \right\} \label{eq-L-class}
\end{equation}
with $a \in \mathbf{No}^{>, \succ}$ form a partition of $\mathbf{No}^{>,
\succ}$ into convex subclasses. Moreover, the series $\sharp_{\beta}
(a)$\label{autolab33} is both the unique $\beta$-truncated element and the
$\trianglelefteqslant$-minimum of the convex class containing $a$. We have
\[ \mathbf{No}_{\succ, \beta} = L_{\beta}  \mathbf{Mo}_{\beta} \]
by {\cite[Proposition~7.6]{vdH:hypno}}. This allows us to define a map
$\mathfrak{d}_{\beta} : \mathbf{No}^{\succ, >} \longrightarrow
\mathbf{Mo}_{\beta}$ by {$\mathfrak{d}_{\beta} (b) \assign
E_{\beta}^{\smash{\sharp_{\beta} (L_{\beta} b)}}$}. In other words,
\[ E_{\beta}^{\smash{\sharp_{\beta} (a)}} =\mathfrak{d}_{\beta} (E_{\beta}
   (a)), \]
for all $a \in \mathbf{No}^{>, \succ}$ (see also
{\cite[Corollary~7.23]{BvdHK:hyp}}).

The formulas~\Cref{eq-log} and~\Cref{eq-hyperlog-Taylor} admit
hyperexponential analogues. For all $a \in \mathbf{No}^{>, \succ}$, there is a
$\gamma < \beta$ with $\varepsilon \assign a - \sharp_{\beta} (a) \prec
\frac{\ell_{\beta}'}{\ell_{\gamma}'} \circ E_{\beta}^{\sharp_{\beta} (a)}$.
For any such $\gamma$, there is a family $(t_{\gamma, k})_{k \in \mathbb{N}}
\in \mathbb{L}_{< \beta}^{\mathbb{N}}$ with $t_0 = \ell_{\gamma}$ such that
$((t_{\gamma, k} \circ E_{\beta}^{\sharp_{\beta} (a)}) \varepsilon^k)_{k \in
\mathbb{N}}$ is well-based and
\begin{equation}
  E_{\beta} a = E_{\gamma} \left( \sum_{k \in \mathbb{N}} \frac{t_{\gamma, k}
  \circ E_{\beta}^{\sharp_{\beta} (a)}}{k!} \varepsilon^k \right) .
  \label{eq-hyperexp-general}
\end{equation}
See {\cite[Section~7.1]{BvdHK:hyp}} for more details on $(t_{\gamma, k})_{k
\in \mathbb{N}}$. The number $L_{\gamma} E_{\beta}^{\sharp_{\beta} (a)}$ is a
monomial with $L_{\gamma} E_{\beta}^{\sharp_{\beta} (a)} \succ (t_{\gamma, 1}
\circ E_{\beta}^{\sharp_{\beta} (a)}) \varepsilon \succ (t_{\gamma, 2} \circ
E_{\beta}^{\sharp_{\beta} (a)}) \varepsilon^2 \succ \cdots$, so
\begin{equation}
  L_{\gamma} E_{\beta}^{\sharp_{\beta} (a)} \trianglelefteqslant L_{\gamma}
  E_{\beta} a \label{eq-trianglehere} .
\end{equation}
\section{Surreal
substructures}\label{subsection-surreal-substructures}\label{subsubsection-surreal-substructures}

In {\cite{BvdH19}}, we introduced the notion of {\tmem{surreal
substructure}}{\index{surreal substructure}}. A surreal substructure is a
subclass $\mathbf{S}$ of $\mathbf{No}$ such that $(\mathbf{No}, \leqslant,
\sqsubseteq)$ and $(\mathbf{S}, \leqslant, \sqsubseteq)$ are isomorphic. The
isomorphism {$\mathbf{No} \longrightarrow \mathbf{S}$} is unique and denoted
by~$\Xi_{\mathbf{S}}$. For the study of $\mathbf{No}$ as a hyperserial field,
many important subclasses of $\mathbf{No}$ turn out to be surreal
substructures. In particular, given {$\alpha = \omega^{\nu} \in \mathbf{On}$},
it is known that the following classes are surreal substructures:
\begin{itemizedot}
  \item The classes $\mathbf{No}^{>}$, $\mathbf{No}^{>, \succ}$ and
  $\mathbf{No}^{\prec}$ of positive, positive infinite and infinitesimal
  numbers.
  
  \item The classes $\mathbf{Mo}$ and $\mathbf{Mo}^{\succ}$ of monomials and
  infinite monomials.
  
  \item The classes $\mathbf{No}_{\succ}$ and $\mathbf{No}_{\succ}^{>}$ of
  purely infinite and positive purely infinite numbers.
  
  \item The class $\mathbf{Mo}_{\alpha}$ of $L_{< \alpha}$-atomic numbers.
  
  \item The class $\mathbf{No}_{\succ, \alpha}$ of $\alpha$-truncated numbers.
\end{itemizedot}
We will prove in Section~\ref{section-nested-series} that certain classes of
nested numbers also form surreal substructures.

\subsection{Cuts}

Given a subclass $\mathbf{X}$ of $\mathbf{No}$ and $a \in \mathbf{X}$, we
define
\[ \begin{array}{rclcccc}
     a_L^{\mathbf{X}} \; & \assign & \{ b \in \mathbf{X} \suchthat b < a
     \wedge b \sqsubseteq a \} & \qquad & a_L & \assign & a_L^{\mathbf{No}}\\
     a_R^{\mathbf{X}} \; & \assign & \{ b \in \mathbf{X} \suchthat b > a
     \wedge b \sqsubseteq a \} &  & a_R & \assign & a_R^{\mathbf{No}}\\
     a_{\sqsubset}^{\mathbf{X}} & \assign & a_L^{\mathbf{X}} \cup
     a_R^{\mathbf{X}} &  & a_{\sqsubset} & \assign &
     a_{\sqsubset}^{\mathbf{No}}
   \end{array} \]
If $\mathbf{X}$ is a subclass of $\mathbf{No}$ and $L, R$ are subsets of
$\mathbf{X}$ with $L < S$, then the class
\begin{eqnarray*}
  (L|R)_{\mathbf{X}} & \assign & \{ a \in \mathbf{X} \suchthat (\forall l \in
  L, l < a) \wedge (\forall r \in R, a < r) \}
\end{eqnarray*}
is called a {\tmem{cut}} in $\mathbf{X}$. If $(L|R)_{\mathbf{X}}$ contains a
unique simplest element, then we denote this element by~$\{ L|R
\}_{\mathbf{X}}$ and say that $(L, R)$ is a {\tmem{cut
representation}}{\index{cut representation}} (of $\{ L|R \}_{\mathbf{X}}$) in
$\mathbf{X}$. These notations naturally extend to the case when $\mathbf{L}$
and $\mathbf{R}$ are subclasses of $\mathbf{X}$ with $\mathbf{L}<\mathbf{R}$.

A surreal substructure $\mathbf{S}$ may be characterized as a subclass of
$\mathbf{No}$ such that for all cut representations $(L, R)$ in $\mathbf{S}$,
the cut $(L|R)_{\mathbf{S}}$ has a unique simplest element
{\cite[Proposition~4.7]{BvdH19}}.

Let $\mathbf{S}$ be a surreal substructure. Note that we have $a = \{
a_L^{\mathbf{S}} |a_R^{\mathbf{S}} \}$ for all $a \in \mathbf{S}$. Let $a \in
\mathbf{S}$ and let~{$(L, R)$} be a cut representation of $a$ in $\mathbf{S}$.
Then $(L, R)$ is {\tmem{cofinal with respect to}}{\index{cofinal with respect
to}} $(a_L^{\mathbf{S}}, a_R^{\mathbf{S}})$ in the sense that $L$ has no
strict upper bound in $a_L^{\mathbf{S}}$ and $R$ has no strict lower bound in
$a_R^{\mathbf{S}}$~{\cite[Proposition~4.11(b)]{BvdH19}}.

\subsection{Cut equations}

Let $\mathbf{X} \subseteq \mathbf{No}$ be a subclass, let $\mathbf{T}$ be a
surreal substructure and $F : \mathbf{X} \longrightarrow \mathbf{T}$ be a
function. Let $\lambda, \rho$ be functions defined for cut representations in
$\mathbf{X}$ such that $\lambda (L, R), \rho (L, R)$ are subsets of
$\mathbf{T}$ whenever $(L, R)$ is a cut representation in $\mathbf{X}$. We say
that $(\lambda, \rho)$ is a {\tmem{cut equation}}{\index{cut equation}} for
$F$ if for all $a \in \mathbf{X}$, we have
\[ \lambda (a_L^{\mathbf{X}}, a_R^{\mathbf{X}}) \; < \; \rho
   (a_L^{\mathbf{X}}, a_R^{\mathbf{X}}), \qquad F (a) \; = \; \{ \lambda
   (a_L^{\mathbf{X}}, a_R^{\mathbf{X}}) | \rho (a_L^{\mathbf{X}},
   a_R^{\mathbf{X}}) \}_{\mathbf{T}} . \]
Elements in $\lambda (a_L^{\mathbf{X}}, a_R^{\mathbf{X}})$ (resp. $\rho
(a_L^{\mathbf{X}}, a_R^{\mathbf{X}})$) are called {\tmem{left}}
({\tmabbr{resp.}} {\tmem{right}}) {\tmem{options}}{\index{left option, right
option}} of this cut equation at~$a$. We say that the cut equation is
{\tmem{uniform}}{\index{uniform equation}} if
\[ \lambda (L, R) \; < \; \rho (L, R), \qquad F (\{ L|R \}_{\mathbf{X}}) \; =
   \; \{ \lambda (L, R) | \rho (L, R) \}_{\mathbf{T}} \]
for all cut representations $(L, R)$ in $\mathbf{X}$. For instance, given $r
\in \mathbb{R}$, consider the translation $T_r : \mathbf{No} \longrightarrow
\mathbf{No} ; a \longmapsto a + r$ on $\mathbf{No}$. By
{\cite[Theorem~3.2]{Gon86}}, we have the following uniform cut equation for
$T_r$ on $\mathbf{No}$:
\begin{equation}
  \nobracket \forall a \in \mathbf{No} \nobracket, \quad a + r = \{ a_L + r, a
  + r_L |a + r_R, a_R + r \} . \label{eq-uniform-sum} \text{}
\end{equation}
Let $\nu \in \mathbf{On}$ with $\nu > 0$ and set $\alpha \assign
\omega^{\nu}$. We have the following uniform cut equations for $L_{\alpha}$ on
$\mathbf{Mo}_{\alpha}$ and $E_{\alpha}$ on $\mathbf{No}_{\succ, \alpha}$
{\cite[Section~8.1]{vdH:hypno}}:
\begin{eqnarray}
  \nobracket \forall \mathfrak{a} \in \mathbf{Mo}_{\alpha} \nobracket, \quad
  L_{\alpha} \mathfrak{a} & = & \{ L_{\alpha}
  \mathfrak{a}_L^{\mathbf{Mo}_{\alpha}} |L_{\alpha}
  \mathfrak{a}_R^{\mathbf{Mo}_{\alpha}}, L_{< \alpha} \mathfrak{a}
  \}_{\mathbf{No}_{\succ, \alpha}} \\
  & = & \{ \mathcal{L}_{\alpha} L_{\alpha}
  \mathfrak{a}_L^{\mathbf{Mo}_{\alpha}} |\mathcal{L}_{\alpha} L_{\alpha}
  \mathfrak{a}_R^{\mathbf{Mo}_{\alpha}}, L_{< \alpha} \mathfrak{a} \} . \\
  \nobracket \forall \varphi \in \mathbf{No}_{\succ, \alpha} \nobracket, \quad
  E_{\alpha}^{\varphi} & = & \left\{ E_{X_{\alpha}} \mathfrak{d}_{\alpha}
  (\varphi), E_{\alpha}^{\varphi_L^{\smash{\mathbf{No}_{\succ, \alpha}}}} |
  E_{\alpha}^{\varphi_R^{\smash{\mathbf{No}_{\succ, \alpha}}}}
  \right\}_{\mathbf{Mo}_{\alpha}} \\
  & = & \left\{ E_{< \alpha} \varphi, \mathcal{E}_{\alpha}
  E_{\alpha}^{\varphi^{\mathbf{No}_{\succ, \alpha}}_L} | \mathcal{E}_{\alpha}
  E_{\alpha}^{\varphi^{\mathbf{No}_{\succ, \alpha}}_R} \right\} . 
  \label{eq-rich-hyperexp}
\end{eqnarray}
where
\begin{eqnarray*}
  X_{\alpha} & = & \left\{\begin{array}{ll}
    \{ 0 \} & \text{\quad if $\nu$ is a limit}\\
    \{ \omega^{\mu} n \suchthat n \in \mathbb{N} \} & \text{\quad if $\nu =
    \mu + 1$ is a successor.}
  \end{array}\right.
\end{eqnarray*}

\subsection{Function groups}

A {\tmem{function group}}{\index{function group}} $\mathcal{G}$ on a surreal
substructure $\mathbf{S}$ is a set-sized group of strictly increasing
bijections $\mathbf{S} \longrightarrow \mathbf{S}$ under functional
composition. We see elements $f, g$ of $\mathcal{G}$ as actions
on~$\mathbf{S}$ and sometimes write $fg$ and $f a$ for $a \in \mathbf{S}$
rather than $f \circ g$ and~$f (a)$. We also write $f^{\tmop{inv}}$ for the
functional inverse of $f \in \mathcal{G}$.

Given such a function group $\mathcal{G}$, the collection of
classes\label{autolab34}
\[ \mathcal{G} [a] \assign \{ b \in \mathbf{S} \suchthat \exists f, g \in
   \mathcal{G}, f a \leqslant b \leqslant g a \} \]
with $a \in \mathbf{S}$ forms a partition of $\mathbf{S}$ into convex
subclasses. For subclasses $\mathbf{X} \subseteq \mathbf{S}$, we write
$\mathcal{G} [\mathbf{X}] \assign \bigcup_{a \in \mathbf{X}} \mathcal{G} [a]$.
An element $a \in \mathbf{S}$ is said to be
$\mathcal{G}${\tmem{-simple}}{\index{$\mathcal{G}$-simple element}} if it is
the simplest element inside $\mathcal{G} [a]$. We write
$\mathbf{Smp}_{\mathcal{G}}$\label{autolab35} for the class of
$\mathcal{G}$-simple elements. Given $a \in \mathbf{S}$, we also define
$\pi_{\mathcal{G}} (a)$\label{autolab36} to be the unique $\mathcal{G}$-simple
element of~$\mathcal{G} [a]$. \ The function $\pi_{\mathcal{G}}$ is a
non-decreasing projection of $(\mathbf{S}, \leqslant)$ onto
$(\mathbf{Smp}_{\mathcal{G}}, \leqslant)$. The main purpose of function groups
is to define surreal substructures:

\begin{proposition}
  \label{prop-group-substructure}{\tmem{{\cite[Theorem~6.7 and
  Proposition~6.8]{BvdH19}}}} The class $\mathbf{Smp}_{\mathcal{G}}$ is a
  surreal substructure. We have the uniform cut equation
  \begin{equation}
    \nobracket \forall z \in \mathbf{No} \nobracket, \quad
    \Xi_{\mathbf{Smp}_{\mathcal{G}}} z = \{ \mathcal{G}
    \Xi_{\mathbf{Smp}_{\mathcal{G}}} z_L | \mathcal{G}
    \Xi_{\mathbf{Smp}_{\mathcal{G}}} z_R \}_{\mathbf{S}} .
    \label{eq-uniform-G}
  \end{equation}
\end{proposition}

Note that for $a, b \in \mathbf{Smp}_{\mathcal{G}}$, we have $a < b$ if and
only if $\mathcal{G}a <\mathcal{G}b$. We have the following criterion to
identify the $\mathcal{G}$-simple elements inside $\mathbf{S}$.

\begin{proposition}
  \label{prop-simplicity-condition}{\tmem{{\cite[Lemma~6.5]{BvdH19}}}} An
  element $a$ of $\mathbf{S}$ is $\mathcal{G}$-simple if and only if there is
  a cut representation $(L, R)$ of $a$ in $\mathbf{S}$ with $\mathcal{G}L < a
  <\mathcal{G}R$. Equivalently, the number $a \in \mathbf{S}$ is
  $\mathcal{G}$-simple if and only if $\mathcal{G}a_L^{\mathbf{S}} < a
  <\mathcal{G}a_R^{\mathbf{S}}$.
\end{proposition}

Given $X, Y$ be sets of strictly increasing bijections $\mathbf{S}
\longrightarrow \mathbf{S}$, we define\label{autolab37} \label{autolab38}
\begin{eqnarray*}
  X \leqangle Y & \xLeftrightarrow{\tmop{def}} & \nobracket \forall a \in
  \mathbf{S} \nobracket, \nobracket \forall f \in X \nobracket, \nobracket
  \exists g \in Y \nobracket, \hspace{1.2em} f a \leqslant g a\\
  X \legeangle Y & \xLeftrightarrow{\tmop{def}} & X \leqangle Y \infixand Y
  \leqangle X\\
  X \leqslant Y & \xLeftrightarrow{\tmop{def}} & \nobracket \forall a \in
  \mathbf{S} \nobracket, \nobracket \forall f \in X \nobracket, \nobracket
  \forall g \in Y \nobracket, \hspace{1.2em} f a \leqslant g a\\
  X < Y & \xLeftrightarrow{\tmop{def}} & \nobracket \forall a \in \mathbf{S}
  \nobracket, \nobracket \forall f \in X \nobracket, \nobracket \forall g \in
  Y \nobracket, \hspace{1.2em} f a < g a.
\end{eqnarray*}
If $X \leqangle Y$, then we say that $Y$ is {\tmem{pointwise cofinal}} with
respect to $X$. For $f, g \in \mathbf{S}$, we also write $f < Y$ or $X < g$
instead of $\{ f \} < Y$ and $X < \{ g \}$.

Given a function group $\mathcal{G}$ on $\mathbf{S}$, we define a partial
order $<$ on $\mathcal{G}$ by $f < g \Longleftrightarrow \{ f \} < \{ g \}$.
We will frequently rely on the elementary fact that this ordering is
compatible with the group structure in the sense that
\[ \nobracket \forall \nobracket f, g, h \nobracket \in \mathcal{G}
   \nobracket, \quad g > \tmop{id}_{\mathbf{S}} \Longleftrightarrow fgh > fh.
\]
Given a set $X$ of strictly increasing bijections $\mathbf{S} \longrightarrow
\mathbf{S}$, we define $\langle X \rangle$ to be the smallest function group
on $\mathbf{S}$ that is generated by $X$, {\tmabbr{i.e.}} $\langle X \rangle
\assign \{ f_1 \circ \cdots \circ f_n \suchthat n \in \mathbb{N}, f_1, \ldots,
f_n \in X \cup X^{\tmop{inv}} \}$\label{autolab39}.

\subsection{Remarkable function groups}

The examples of surreal substructures from the beginning of this section can
all be obtained as classes $\mathbf{Smp}_{\mathcal{G}}$ of
$\mathcal{G}$-simplest elements for suitable function groups $\mathcal{G}$
that act on $\mathbf{No}$, $\mathbf{No}^{>}$, or $\mathbf{No}^{>, \succ}$, as
we will describe now. Given $c \in \mathbb{R}$ and $r \in \mathbb{R}^{>}$, we
first~define\label{autolab40} \label{autolab41} \label{autolab42}
\[ \begin{array}{rcll}
     T_r & \assign & a \longmapsto a + c & \text{\qquad acting on
     $\mathbf{No}$ or $\mathbf{No}^{>, \succ}$}\\
     H_c & \assign & a \longmapsto ra & \text{\qquad acting on
     $\mathbf{No}^{>}$ or $\mathbf{No}^{>, \succ}$}\\
     P_c & \assign & a \longmapsto a^r & \text{\qquad acting on
     $\mathbf{No}^{>}$ or $\mathbf{No}^{>, \succ}$.}
   \end{array} \]
For $\alpha = \omega^{\nu} \in \mathbf{On}$, we then have the following
function groups\label{autolab43} \label{autolab44} \label{autolab45}
\label{autolab46} \label{autolab47}
\begin{eqnarray*}
  \mathcal{T} & \assign & \{ T_c \suchthat c \in \mathbb{R} \}\\
  \mathcal{H} & \assign & \{ H_r \suchthat r \in \mathbb{R}^{>} \}\\
  \mathcal{P} & \assign & \{ P_r \suchthat r \in \mathbb{R}^{>} \}\\
  \mathcal{E}_{\alpha} & \assign & \langle E_{\gamma} H_r L_{\gamma} :
  \matheuler < \alpha, r \in \mathbb{R}^{>} \rangle\\
  \mathcal{L}_{\alpha} & \assign & L_{\alpha} \mathcal{E}_{\alpha} E_{\alpha}
  .
\end{eqnarray*}
Now the action of $\mathcal{T}$ on $\mathbf{No}$ yields the surreal
substructure $\mathbf{No}_{\succ} \assign \mathbf{Smp}_{\mathcal{T}}$ as class
{of~$\mathcal{T}$\mbox{-}simplest} elements. All examples from the beginning
of this section can be obtained in a similar way:
\begin{itemizedot}
  \item The action of $\mathcal{T}$ on $\mathbf{No}$ (resp. $\mathbf{No}^{>,
  \succ}$) yields $\mathbf{No}_{\succ}$ (resp. $\mathbf{No}_{\succ}^{>}$).
  
  \item The action of $\mathcal{H}$ on $\mathbf{No}^{>}$ (resp.
  $\mathbf{No}^{>, \succ}$) yields $\mathbf{Mo}$ (resp.
  $\mathbf{Mo}^{\succ}$).
  
  \item The action of $\mathcal{P}$ on $\mathbf{No}^{>, \succ}$ yields
  $\mathbf{Mo} \lebar \mathbf{Mo} = E_1^{\mathbf{Mo}^{\succ}}$.
  
  \item The action of $\mathcal{E}_{\alpha}$ on $\mathbf{No}^{>, \succ}$
  yields $\mathbf{Mo}_{\alpha}$.
  
  \item The action of $\mathfrak{L}_{\alpha}$ on $\mathbf{No}^{>, \succ}$
  yields $\mathbf{No}_{\succ, \alpha}$.
\end{itemizedot}
We have
\begin{eqnarray*}
  \pi_{\mathbf{Smp}_{\mathcal{E}_{\alpha}}} & = & \mathfrak{d}_{\alpha}\\
  \pi_{\mathbf{Smp}_{\mathcal{L}_{\alpha}}} & = & \sharp_{\alpha} .
\end{eqnarray*}
Let $E_{< \alpha} = \{ E_{\gamma} \suchthat \gamma < \alpha \}$ and $L_{<
\alpha} = \{ L_{\gamma} \suchthat \gamma < \alpha \}$. We will need a few
inequalities from {\cite{vdH:hypno}}. The first one is immediate by definition
and the fact that $\mathcal{H}< E_{\alpha}$. The others are {\cite[Lemma~6.9,
Lemma~6.11, and {Proposition~6.17}]{vdH:hypno}}, in that order:
\begin{eqnarray}
  \mathcal{E}_{\alpha \omega} & < & E_{\alpha}  \label{ineq1}\\
  E_{< \alpha} & < & E_{\alpha} H_2 L_{\alpha} \hspace*{\fill} \text{} 
  \label{ineq2}\\
  \langle E_{\gamma} \suchthat \gamma < \alpha \rangle & \legeangle &
  \mathcal{E}_{\alpha} \text{\quad if $\nu$ is a limit}  \label{ineq3}\\
  \nobracket \forall \gamma < \rho < \alpha \nobracket, \nobracket \forall r,
  s > 1 \nobracket, \quad E_{\gamma} H_r L_{\gamma} & < & E_{\rho} H_s
  L_{\rho} .  \label{ineq4}
\end{eqnarray}
From (\ref{ineq4}), we also deduce that
\begin{equation}
  \{ E_{\gamma} H_r L_{\gamma} \suchthat \gamma < \alpha, r \in \mathbb{R} \}
  \legeangle \mathcal{E}_{\alpha} . \label{ineq5}
\end{equation}
\section{Well-nestedness}\label{section-well-nestedness}

In this section, we prove Theorem~\ref{th-well-nested}, i.e. that each number
is well-nested. In Section~\ref{subsection-standard} we start with the
definition and study of hyperserial expansions. We pursue with the study of
paths and well-nestedness in Section~\ref{subsection-paths}.

The general idea behind our proof of Theorem~\ref{th-well-nested} is as
follows. Assume for contradiction that there exists a number $a$ that is not
well-nested and choose a simplest (i.e. $\sqsubseteq$\mbox{-}minimal) such
number. By definition, $a$ contains a so-called ``bad path''. For the
ill-nested number $a$ from~(\ref{ill-nested-example}), that would be the
sequence
\[ \mathe^{\sqrt{\log \omega} + \mathe^{\sqrt{\log \log \omega} +
   \mathe^{\udots \ddots} + \log \log \log \omega} + \log \log \omega},
   \mathe^{\sqrt{\log \log \omega} + \mathe^{\udots \ddots} + \log \log \log
   \omega}, \mathe^{\udots \ddots}, \ldots \]
From this sequence, we next construct a ``simpler'' number like
\[ a' \assign \sqrt{\omega} + \mathe^{\sqrt{\log \omega} + \mathe^{\sqrt{\log
   \log \omega} + \mathe^{\udots \ddots} + \log \log \log \omega}} \]
that still contains a bad path
\[ \mathe^{\sqrt{\log \omega} + \mathe^{\sqrt{\log \log \omega} +
   \mathe^{\udots \ddots} + \log \log \log \omega}}, \mathe^{\sqrt{\log \log
   \omega} + \mathe^{\udots \ddots} + \log \log \log \omega}, \mathe^{\udots
   \ddots}, \ldots, \]
thereby contradicting the minimality assumption on $a$. In order to make this
idea work, we first need a series of ``deconstruction lemmas'' that allow us
to affirm that $a'$ is indeed simpler than $a$; these lemmas will be listed in
Section~\ref{subsection-decomposition}. We will also need a generalization
$\lesssim$ of the relation $\nonconverted{blacktrianglelefteqslant}$ that was
used by Berarducci and Mantova to prove the well-nestedness of $\mathbf{No}$
as a field of transseries; this will be the subject of
Section~\ref{subsection-nested-truncation-relations}. We prove
Theorem~\ref{th-well-nested} in Section~\ref{subsection-well-nestedness}.
Unfortunately, the relation $\lesssim$ does not have all the nice properties
of $\nonconverted{blacktrianglelefteqslant}$. For this reason,
Sections~\ref{subsection-nested-truncation-relations}
and~\ref{subsection-well-nestedness} are quite technical.

\subsection{Hyperserial expansions}\label{subsection-standard}

Recall that any number can be written as a well-based series. In order to
represent numbers as hyperseries, it therefore suffices to devise a means to
represent the infinitely large monomials~$\mathfrak{m}$
in~$\mathbf{Mo}^{\succ}$. We do this by taking a hyperlogarithm $L_{\alpha}
\mathfrak{m}$ of the monomial and then recursively applying the same procedure
for the monomials in this new series. This procedure stops when we encounter a
monomial in $L_{\mathbf{On}} \omega$.

Technically speaking, instead of directly applying a hyperlogarithm
$L_{\alpha}$ to the monomial, it turns out to be necessary to first decompose
$\mathfrak{m}$ as a product $\mathfrak{m}= \mathe^{\psi} \mathfrak{n}$ and
write $\mathfrak{n}$ as a hyperexponential (or more generally as the
hyperlogarithm of a hyperexponential). This naturally leads to the
introduction of {\tmem{hyperserial expansions}} of monomials $\mathfrak{m} \in
\mathbf{Mo}^{\neq 1}$, as~we will detail now.

\begin{definition}
  We say that a purely infinite number $\varphi \in \mathbf{No}_{\succ}$ is
  {\tmem{{\tmstrong{tail-atomic}}}}{\index{tail-atomic number}} if $\varphi =
  \psi \pplus \iota \mathfrak{a}$, for certain $\psi \in \mathbf{No}_{\succ}$,
  $\iota \in \{ - 1, 1 \}$, and $\mathfrak{a} \in \mathbf{Mo}_{\omega}$.
\end{definition}

\begin{definition}
  Let $\mathfrak{m} \in \mathbf{Mo}^{\neq 1}$. Assume that there are $\psi \in
  \mathbf{No}_{\succ}$, $\iota \in \{ - 1, 1 \}$, $\alpha \in \{ 0 \} \cup
  \omega^{\mathbf{On}}$, $\beta \in \mathbf{On}$ and $u \in \mathbf{No}$ such
  that
  \begin{equation}
    \mathfrak{m}= \mathe^{\psi}  (L_{\beta} E_{\alpha}^u)^{\iota},
    \label{eq-standard-expansion-gen}
  \end{equation}
  with $\tmop{supp} \psi \succ L_{\beta + 1} E_{\alpha}^u$. Then we say that
  \tmtextup{(\ref{eq-standard-expansion-gen})} is a
  {\tmstrong{{\tmem{hyperserial expansion of type I}}}} if
  \begin{itemize}
    \item $\beta \omega < \alpha$;
    
    \item $E_{\alpha}^u \in \mathbf{Mo}_{\alpha} \setminus L_{< \alpha} 
    \mathbf{Mo}_{\alpha \omega}$;
    
    \item $\alpha = 1 \Longrightarrow \left( \psi = 0 \text{ and $u$ is not
    tail-atomic} \right)$.
  \end{itemize}
  We say that \tmtextup{(\ref{eq-standard-expansion-gen})} is a
  {\tmstrong{{\tmem{hyperserial expansion of type II}}}} if $\alpha = 0$ and
  $u = \omega$, so that $E_{\alpha}^u = \omega$ and
  \begin{equation}
    \mathfrak{m}= \mathe^{\psi}  (L_{\beta} \omega)^{\iota} .
    \label{eq-st-expansion-1}
  \end{equation}
\end{definition}

Hyperserial expansions can be represented formally by tuples~${(\psi, \iota, \alpha, \beta, u)}$. By convention, we also consider
\[ 1 = \mathe^0  (L_0 E_0 0)^0, \]
to be a hyperserial expansion of the monomial $\mathfrak{m}= 1$; this
expansion is represented by the tuple~{$(0, 0, 0, 0, 0)$}.

\begin{example}
  \label{example-st}We will give a hyperserial expansion for the monomial
  \[ \mathfrak{m}= \exp \left( 2 E_{\omega} \omega - \sqrt{\omega} + L_{\omega
     + 1} \omega \right), \]
  and show how it can be expressed as a hyperseries. Note that
  \[ u \assign \log \mathfrak{m}= 2 E_{\omega} \omega - \sqrt{\omega} +
     L_{\omega + 1} \omega \]
  is tail-atomic since $L_{\omega} \omega$ is log-atomic. Now $L_{\omega}
  \omega = L_{\omega} \omega$ is a hyperserial expansion of type II and we
  have $L_{\omega + 1} \omega \prec E_{\omega} \omega, \sqrt{\omega}$. Hence
  $\mathfrak{m}= \mathe^{2 E_{\omega} \omega - \sqrt{\omega}} 
  ({\color[HTML]{000000}L_{\omega} \omega})$ is a hyperserial expansion.
  
  Let $\psi \assign {\color[HTML]{000000}2 E_{\omega} \omega} -
  \sqrt{\omega}$, so $\mathfrak{m}= \mathe^{\psi}  (L_{\omega} \omega)$. We
  may further expand each monomial in $\tmop{supp} \psi$. We clearly have
  $E_{\omega} \omega \in \mathbf{Mo}_{\omega^2}$. We claim that $E_{\omega}
  \omega \in \mathbf{Mo}_{\omega^2} \setminus L_{< \omega^2} 
  \mathbf{Mo}_{\omega^3}$. Indeed, if we could write $E_{\omega} \omega = L_n
  L_{\omega m} \mathfrak{a}$ for some $\mathfrak{a} \in
  \mathbf{Mo}_{\omega^3}$ and $n, m \in \mathbb{N}^{>}$, then $\omega =
  L_{\omega} (L_n L_{\omega m} \mathfrak{a}) = L_{\omega (m + 1)}
  \mathfrak{a}- n$ and $L_{\omega (m + 1)} \mathfrak{a}$ would both be
  monomials, which cannot be. Note that $E_{\omega} \omega = E_{\omega^2}
  (L_{\omega^2} E_{\omega} \omega) = E_{\omega^2}^{L_{\omega^2} \omega + 1}$,
  so {$E_{\omega} \omega = E_{\omega^2}^{L_{\omega^2} \omega + 1}$} is a
  hyperserial expansion of type I. We also have $\sqrt{\omega} = \exp \left(
  \frac{1}{2} \log \omega \right)$ where $\frac{1}{2} \log \omega$ is
  tail~expanded. Thus $\sqrt{\omega} = E_1^{\frac{1}{2} \log \omega}$ is a
  hyperserial expansion. Note finally that $\log \omega = L_1 \omega$ is a
  hyperserial expansion. We thus have the following ``recursive'' expansion of
  $\mathfrak{m}$:
  \begin{equation}
    \mathfrak{m}= \mathe^{2 E_{\omega^{\small{2}}}^{L_{\omega^{\small{2}}}
    \omega + 1} - E_1^{\frac{1}{2} L_1 \omega}}  (L_{\omega} \omega) .
    \label{ex-expanded-monomial}
  \end{equation}
\end{example}

\begin{lemma}
  \label{lem-st-expansion-existence}Any $\mathfrak{m} \in \mathbf{Mo}$ has a
  hyperserial expansion.
\end{lemma}

\begin{proof}
  We first prove the result for $\mathfrak{m} \in \mathbf{Mo}_{\omega}$, by
  induction with respect to the simplicity relation $\sqsubseteq$. The
  $\sqsubseteq$-minimal element of $\mathbf{Mo}_{\omega}$ is $\omega$, which
  satisfies~(\ref{eq-st-expansion-1}) for $\psi = \beta = 0$ and~{$\iota =
  1$}. Consider $\mathfrak{m} \in \mathbf{Mo}_{\omega} \setminus \{ \omega \}$
  such that the result holds on
  $\mathfrak{m}_{\sqsubset}^{\mathbf{Mo}_{\omega}}$. By
  {\cite[Proposition~6.20]{vdH:hypno}}, the monomial $\mathfrak{m}$ is not
  $L_{< \mathbf{On}}$-atomic. So there is a maximal $\lambda \in
  \omega^{\mathbf{On}}$ with $\mathfrak{m} \in \mathbf{Mo}_{\lambda}$, and we
  have $\lambda \geqslant \omega$ by our hypothesis.
  
  If there is no ordinal $\gamma < \lambda$ such that
  $E_{\gamma}^{\mathfrak{m}} \in \mathbf{Mo}_{\lambda \omega}$, then we have
  $\mathfrak{m} \in \mathbf{Mo}_{\lambda} \setminus L_{< \lambda} 
  \mathbf{Mo}_{\lambda \omega}$. So setting $\alpha = \lambda$, $\beta = 0$
  and $u = L_{\lambda} \mathfrak{m}$, we are done. Otherwise, let $\gamma <
  \lambda$ be such that $\mathfrak{a} \assign E_{\gamma}^{\mathfrak{m}} \in
  \mathbf{Mo}_{\lambda \omega}$. We cannot have $\gamma = 0$ by definition of
  $\lambda$. So there is a~unique ordinal $\eta$ and a~unique natural number
  $n \in \mathbb{N}^{>}$ such that $\gamma = \gamma' + \omega^{\eta} n$
  and~{$\gamma' \gg \omega^{\eta}$}. Note that~{$\lambda \geqslant
  \omega^{\eta + 1}$}. We must have $\lambda = \omega^{\eta + 1}$: otherwise,
  $L_{\omega^{\eta + 1}} \mathfrak{m}= L_{\gamma' + \omega^{\eta + 1}}
  (\mathfrak{a}) + n$ where $L_{\omega^{\eta + 1}} \mathfrak{m}$ and
  $L_{\gamma' + \omega^{\eta + 1}} \mathfrak{a}$ are monomials. We deduce that
  $\gamma' = 0$ and $\gamma = \omega^{\eta} n$. Note that {$L_{\lambda}
  \mathfrak{a} \asymp L_{\lambda} \mathfrak{m}$}, $\lambda < \lambda \omega$,
  and $\mathfrak{a} \in \mathbf{Mo}_{\lambda \omega}$, so
  $\mathfrak{a}=\mathfrak{d}_{\lambda \omega} (\mathfrak{m})$. We deduce that
  $\mathfrak{a} \sqsubset \mathfrak{m}$. The induction hypothesis yields a
  hyperserial expansion $\mathfrak{a}= \mathe^{\psi}  (L_{\beta}
  E_{\alpha}^u)^{\iota}$. Since $\mathfrak{a}$ is log-atomic, we must have
  $\psi = 0$ and $\iota = 1$. If $\mathfrak{a}= L_{\beta} \omega$, then $\beta
  \ggeq \lambda_{/ \omega} = \omega^{\eta}$, since $\mathfrak{a} \in
  \mathbf{Mo}_{\lambda \omega}$. Thus $\mathfrak{m}= L_{\gamma} \mathfrak{a}=
  L_{\beta + \gamma} \omega$ is a hyperexponential expansion of type II. If
  $\mathfrak{a}= L_{\beta} E_{\alpha}^u$, then likewise $\beta \ggeq
  \omega^{\eta}$ and thus $\mathfrak{m}= L_{\beta + \gamma} E_{\alpha}^u$ is a
  hyperexponential expansion of type I. This completes the inductive proof.
  
  Now let $\mathfrak{m} \in \mathbf{Mo}^{\neq} \setminus \mathbf{Mo}_{\omega}$
  and set $\varphi \assign \log \mathfrak{m}$. If $\varphi$ is tail-atomic,
  then there are $\psi \in \mathbf{No}_{\succ}$, $\iota \in \{ - 1, 1 \}$ and
  $\mathfrak{a} \in \mathbf{Mo}_{\omega}$ with $\varphi = \psi \pplus \iota
  \mathfrak{a}$. Applying the previous arguments to $\mathfrak{a}$, we obtain
  elements $\alpha \geqslant \omega, \beta, u$ with $\mathfrak{a}= L_{\beta}
  E_{\alpha}^u$ and $\beta \omega < \alpha$, or an ordinal $\beta$ with
  $\mathfrak{a}= L_{\beta} \omega$. Then $\mathfrak{m}= \mathe^{\psi} 
  (L_{\beta} E_{\alpha}^u)^{\iota}$ or $\mathfrak{m}= \mathe^{\psi} 
  (L_{\beta} \omega)^{\iota}$ is a hyperserial expansion. If $\varphi$ is not
  tail-atomic, then we have $\mathfrak{m}= E_1^{\varphi}$ is a hyperserial
  expansion of type I.
\end{proof}

\begin{lemma}
  \label{lem-dominant-atomic-standard-expansion}Consider a hyperserial
  expansion $\mathfrak{a}= L_{\beta} E_{\alpha}^u$. Let $\mu > 0$ and define
  $\mu_- \assign \mu - 1$ if $\mu$ is a successor ordinal and $\mu_- \assign
  \mu$ if $\mu$ is a limit ordinal. Let
  \begin{eqnarray*}
    \beta & \backassign & \beta' + \beta'' \text{\quad where}\\
    \beta' & \assign & \beta_{\succcurlyeq \omega^{\mu_{\small{-}}}} \ggeq
    \omega^{\mu_{\small{-}}} \text{\quad and}\\
    \beta'' & \assign & \beta_{\prec \omega^{\mu_{\small{-}}}} <
    \omega^{\mu_{\small{-}}} .
  \end{eqnarray*}
  \begin{enumeratealpha}
    \item \label{lem-dominant-atomic-standard-expansion-a}Then $\mathfrak{a}$
    is $L_{< \omega^{\mu}}$-atomic if and only if $\beta'' = 0$ and either
    $\alpha \geqslant \omega^{\mu}$ or $\alpha = 0$.
    
    \item \label{lem-dominant-atomic-standard-expansion-b}If $\alpha \geqslant
    \omega^{\mu}$, then $\mathfrak{d}_{\omega^{\mu}} (\mathfrak{a}) =
    L_{\beta'} E_{\alpha}^u$.
  \end{enumeratealpha}
\end{lemma}

\begin{proof}
  We first prove $a$). Assume that $\mathfrak{a}$ is $L_{<
  \omega^{\mu}}$-atomic. Assume for contradiction that $\beta'' \neq 0$ and
  let $\omega^{\eta} m$ denote the least non-zero term in the Cantor normal
  form of $\beta''$. Since $\beta'' < \omega^{\mu_{\small{-}}}$, we have
  $\omega^{\eta + 1} < \omega^{\mu}$ so $L_{\omega^{\eta + 1}} \mathfrak{a}$
  is a monomial. But $L_{\omega^{\eta + 1}} \mathfrak{a}= L_{\beta''_{\succ
  \omega^{\eta}}} E_{\alpha}^u - m$ where $L_{\beta''_{\succ \omega^{\eta}}}
  E_{\alpha}^u$ is a monomial: a contradiction. So $\beta'' = 0$. If $\alpha =
  0$ then we are done. Otherwise $E_{\alpha}^u \nin \mathbf{Mo}_{\alpha
  \omega}$, so we must have $\alpha \omega > \omega^{\mu}$, whence $\alpha
  \geqslant \omega^{\mu}$. Conversely, assume that $\alpha \geqslant
  \omega^{\mu}$ or $\alpha = 0$, and that $\beta'' = 0$. If $\alpha \neq 0$,
  then then for all $\gamma < \omega^{\mu}$, we have $L_{\gamma} \mathfrak{a}=
  L_{\beta' + \gamma} E_{\alpha}^u$ where $\beta' + \gamma < \alpha$, so
  $L_{\gamma} \mathfrak{a}$ is a monomial, whence $\mathfrak{a} \in
  \mathbf{Mo}_{\omega^{\mu}}$. If $\alpha = 0$, then for all $\gamma <
  \omega^{\mu}$, we have $L_{\gamma} \mathfrak{a}= L_{\beta' + \gamma} \omega
  \in \mathbf{Mo}$, whence $\mathfrak{a} \in \mathbf{Mo}_{\omega^{\mu}}$. This
  proves $a$).
  
  Now assume that $\alpha \geqslant \omega^{\mu}$. So $L_{\beta'}
  E_{\alpha}^u$ is $L_{< \omega^{\mu}}$-atomic by $a$). If $\beta'' = 0$ then
  we conclude that $\mathfrak{a}= L_{\beta'} E_{\alpha}^u
  =\mathfrak{d}_{\omega^{\mu}} (\mathfrak{a})$. If $\beta'' \neq 0$, then let
  $\omega^{\eta} m$ denote the least non-zero term in its Cantor normal form.
  We have $\omega^{\eta + 1} < \omega^{\mu}$ and $L_{\omega^{\eta + 1}}
  \mathfrak{a}= L_{\omega^{\eta + 1}} L_{\beta'} E_{\alpha}^u - m \asymp
  L_{\omega^{\eta + 1}} L_{\beta'} E_{\alpha}^u$, so $L_{\beta'} E_{\alpha}^u
  =\mathfrak{d}_{\omega^{\mu}} (\mathfrak{a})$.
\end{proof}

\begin{corollary}
  \label{cor-alpha-alphaomega}Let $\mu \in \mathbf{On}^{>}$, $\alpha \assign
  \omega^{\mu}$, $\gamma < \alpha$, and $\mathfrak{b} \in \mathbf{Mo}_{\alpha
  \omega}$. If $L_{\gamma} \mathfrak{b} \in \mathbf{Mo}_{\alpha} \setminus
  \mathbf{Mo}_{\alpha \omega}$, then $\mu$ is a successor ordinal and $\gamma
  = \alpha_{/ \omega} n$ for some $n \in \mathbb{N}^{>}$.
\end{corollary}

\begin{proof}
  Since $L_{\gamma} \mathfrak{b} \in \mathbf{Mo}_{\alpha} \setminus
  \mathbf{Mo}_{\alpha \omega}$, we must have $\gamma \neq 0$. By
  Lemma~\ref{lem-st-expansion-existence}, we have a hyperserial expansion
  $\mathfrak{b}= \mathe^{\psi}  (L_{\beta} E_{\rho}^u)^{\iota}$. Since
  $\mathfrak{b}$ is log-atomic, we have $\log \mathfrak{b}= \psi \pplus \iota
  L_{\beta + 1} E_{\alpha}^u \in \mathbf{Mo}$, whence $\psi = 0$ and $\iota =
  1$. So $\mathfrak{b}= L_{\beta} E_{\rho}^u$. We have $\mathfrak{b} \in
  \mathbf{Mo}_{\alpha \omega}$ so by
  \Cref{lem-dominant-atomic-standard-expansion}(\ref{lem-dominant-atomic-standard-expansion-a}),
  we have {$\beta \ggeq \alpha$}. It follows that $L_{\gamma} \mathfrak{b}=
  L_{\beta + \gamma} E_{\rho}^u$ is a hyperserial expansion. But then
  {$L_{\beta + \gamma} E_{\rho}^u \in \mathbf{Mo}_{\alpha}$} and
  \Cref{lem-dominant-atomic-standard-expansion}({\tmem{\ref{lem-dominant-atomic-standard-expansion-a}}})
  imply that $\gamma \ggeq \omega^{\mu_-}$. The condition that $\gamma <
  \alpha$ now gives $\mu_- < \mu$, whence $\mu$ is a successor and $\gamma =
  \omega^{\mu_-} n$ for a certain $n \in \mathbb{N}^{>}$, as claimed.
\end{proof}

\begin{lemma}
  \label{lem-st-expansion-unicity}Any $\mathfrak{m} \in \mathbf{Mo}$ has a
  unique hyperserial expansion (that we will call \tmtextup{the} hyperserial
  expansion, henceforth).
\end{lemma}

\begin{proof}
  Consider a monomial $\mathfrak{m} \neq 1$ with
  \[ \mathfrak{m}= \mathe^{\psi}  (L_{\beta} \mathfrak{a})^{\iota}, \]
  where $\psi \in \mathbf{No}_{\succ}$, $\iota \in \{ - 1, 1 \}$, $\beta,
  \alpha \in \omega^{\mathbf{On}}$, $\mathfrak{a} \in \mathbf{Mo}_{\alpha}$,
  $\beta \omega < \alpha$, and $\tmop{supp} \psi \succ L_{\beta + 1}
  \mathfrak{a}$. Assume for contradiction that we can write $\mathfrak{m}=
  \mathe^{\psi'}  (L_{\beta'} E_{\alpha'} u')^{\iota'}$ as a hyperserial
  expansion of type I with~$\alpha' < \alpha$. Note in particular that $\alpha
  > 1$, so $L_{\beta + 1} \mathfrak{a}$ is log-atomic. We have
  \[ \log \mathfrak{m}= \psi \pplus \iota L_{\beta + 1} \mathfrak{a}= \psi'
     \pplus \iota' L_{\beta' + 1} E_{\alpha'}^{u'} . \]
  If $\alpha' = 1$, then $\beta' = 0$, $\psi' = 0$, $\iota' = 1$, and $u'$ is
  not tail-atomic. But $\psi \pplus \iota L_{\beta + 1} \mathfrak{a}= u'$,
  where $L_{\beta + 1} \mathfrak{a} \in \mathbf{Mo}_{\omega}$, so $u'$ is
  tail-atomic: a contradiction. Hence $\alpha' > 1$. Note that $\iota L_{\beta
  + 1} \mathfrak{a}$ and $\iota' L_{\beta' + 1} E_{\alpha'}^{u'}$ are both the
  least term of $\log \mathfrak{m}$. It follows that $\psi = \psi'$, $\iota =
  \iota'$, and
  \begin{equation}
    L_{\beta} \mathfrak{a}= L_{\beta'} E_{\alpha'}^{u'} .
    \label{eq-a-beta-prime-beta}
  \end{equation}
  Since $\beta' \omega < \alpha'$, we have
  \[ E_{\alpha'}^{u'} =\mathfrak{d}_{\alpha'} (L_{\beta'} E_{\alpha'}^{u'})
     =\mathfrak{d}_{\alpha'} (L_{\beta} \mathfrak{a}) . \]
  Now $E_{\alpha'}^{u'} \nin \mathbf{Mo}_{\alpha' \omega}$, so
  $\mathfrak{d}_{\alpha'} (L_{\beta} \mathfrak{a}) \neq \mathfrak{a}$ and thus
  $\beta \omega \geqslant \alpha'$. In particular $\beta > \beta'$. Taking
  hyperexponentials on both sides of \Cref{eq-a-beta-prime-beta}, we may
  assume without loss of generality that~{$\beta' = 0$} or that the least
  exponents $\eta$ and $\eta'$ in the Cantor normal forms of $\beta$
  {\tmabbr{resp.}} $\beta'$ differ. If $\beta' = 0$, then we decompose $b =
  \gamma \dotplus \omega^{\eta} n$ where $n \in \mathbb{N}^{>}$ and $\gamma
  \gg \omega^{\eta}$. Since $L_{\beta} \mathfrak{a}= E_{\alpha'}^{u'} \in
  \mathbf{Mo}_{\alpha'} \setminus \mathbf{Mo}_{\alpha' \omega}$, applying
  \Cref{lem-dominant-atomic-standard-expansion}({\tmem{\ref{lem-dominant-atomic-standard-expansion-a}}})
  twice (for $\omega^{\mu} = \alpha'$ and $\omega^{\mu} = \alpha' \omega$)
  gives $\omega^{\eta + 1} \geqslant \alpha'$ and {$\omega^{\eta + 1}
  \ngeqslant \alpha' \omega$}, whence $\alpha' = \omega^{\eta + 1}$. But then
  $E_{\alpha'}^{u'} = L_{\omega^{\eta} n} \mathfrak{b}$, where $\mathfrak{b}
  \assign L_{\gamma} \mathfrak{a} \in \mathbf{Mo}_{\alpha'} \omega$ by
  \Cref{lem-dominant-atomic-standard-expansion}({\tmem{\ref{lem-dominant-atomic-standard-expansion-a}}}).
  So $E_{\alpha'}^{u'} \in L_{< \alpha'}  \mathbf{Mo}_{\alpha' \omega}$: a
  contradiction. Assume now that $\beta' \neq 0$.
  \Cref{lem-dominant-atomic-standard-expansion}({\tmem{\ref{lem-dominant-atomic-standard-expansion-a}}})
  yields both $L_{\beta} \mathfrak{a} \in {\mathbf{Mo}_{\omega^{\eta + 1}}
  \setminus \mathbf{Mo}_{\omega^{\eta + 2}}}$ and $L_{\beta'} E_{\alpha'}^{u'}
  \in \mathbf{Mo}_{\omega^{\eta' + 1}} \setminus \mathbf{Mo}_{\omega^{\eta' +
  2}}$, which contradicts \Cref{eq-a-beta-prime-beta}.
  
  Taking $\mathfrak{a}= \omega \in \mathbf{No}$ and $\alpha \assign \max
  (\alpha' \omega, \beta \omega^2)$, this proves that no two hyperserial
  expansions of distinct types I and II can be equal. Taking $\mathfrak{a}=
  E_{\alpha}^u$ with $\alpha > \alpha'$, this proves that no two hyperserial
  expansions $\mathe^{\psi}  (L_{\beta} E_{\alpha}^u)^{\iota}, \mathe^{\psi'} 
  (L_{\beta'} E_{\alpha'}^{u'})^{\iota'}$ of type I with $\alpha \neq \alpha'$
  can be equal.
  
  The two remaining cases are hyperserial expansions of type II and
  hyperserial expansions $\mathe^{\psi}  (L_{\beta} E_{\alpha}^u)^{\iota}$ and
  $\mathe^{\psi'}  (L_{\beta'} E_{\alpha'}^{u'})^{\iota'}$ of type I with
  $\alpha = \alpha'$. Consider a monomial $\mathfrak{m} \in
  \mathbf{Mo}^{\neq}$ with the hyperserial expansions $\mathfrak{m}=
  \mathe^{\psi}  (L_{\gamma} \omega)^{\iota} = \mathe^{\psi'}  (L_{\gamma'}
  \omega)^{\iota'}$ of type II. As above we have $\psi = \psi'$, $\iota =
  \iota'$, and $L_{\gamma} \omega = L_{\gamma'} \omega$. We deduce that
  $\gamma = \gamma'$, so the expansions coincide.
  
  Finally, consider a monomial $\mathfrak{m} \neq 1$ with two hyperserial
  expansions of type I

  \begin{equation}
    \mathfrak{m}= \mathe^{\psi}  (L_{\beta} E_{\alpha}^u)^{\iota} =
    \mathe^{\psi'}  (L_{\beta'} E_{\alpha}^{u'})^{\iota'} .
    \label{eq-conf-uni}
  \end{equation}
  If $\alpha = 1$, then we have $\psi = \psi' = 0$ and $\beta = \beta' = 0$
  and $\iota = \iota' = 1$, whence $u = u'$, so we are done.
  
  Assume now that $\alpha > 1$. Taking logarithms in (\ref{eq-conf-uni}), we
  see that $\psi = \psi'$, $\iota = \iota'$, and
  \begin{equation}
    L_{\beta} E_{\alpha}^u = L_{\beta'} E_{\alpha}^{u'} .
    \label{eq-beta-prime-beta}
  \end{equation}
  We may assume without loss of generality that~$\beta \geqslant \beta'$.
  Assume for contradiction that $\beta > \beta'$. Taking hyperexponentials on
  both sides of \Cref{eq-beta-prime-beta}, we may assume without loss of
  generality that~{$\beta' = 0$} or that the least exponents $\eta$ and
  $\eta'$ in the Cantor normal forms of~$\beta$~{\tmabbr{resp.}}~$\beta'$
  differ. On the one hand,
  \Cref{lem-dominant-atomic-standard-expansion}({\tmem{\ref{lem-dominant-atomic-standard-expansion-a}}})
  yields $L_{\beta} E_{\alpha}^u \in \mathbf{Mo}_{\omega^{\eta + 1}} \setminus
  \mathbf{Mo}_{\omega^{\eta + 2}}$. Note in particular that $L_{\beta}
  E_{\alpha}^u \nin \mathbf{Mo}_{\alpha}$, since $\beta \omega < \alpha$. On
  the other hand, if $\beta \neq 0$, then
  \Cref{lem-dominant-atomic-standard-expansion}({\tmem{\ref{lem-dominant-atomic-standard-expansion-a}}})
  yields $L_{\beta'} E_{\alpha}^{u'} \in \mathbf{Mo}_{\omega^{\eta'}}
  \setminus \mathbf{Mo}_{\omega^{\eta' + 1}}$; if $\beta' = 0$, then
  $L_{\beta'} E_{\alpha}^{u'} \in \mathbf{Mo}_{\alpha}$. Thus
  \Cref{eq-beta-prime-beta} is absurd: a contradiction. We conclude that
  $\beta = \beta'$. Finally $E_{\alpha}^u = E_{\alpha}^{u'}$ yields $u = u'$,
  so the expansions are identical.
\end{proof}

\begin{lemma}
  \label{lem-st-expansion-no-sign}If $\mathfrak{m}= \mathe^{\psi}  (L_{\beta}
  E_{\alpha}^u)^{\iota}$ is a hyperserial expansion of type I, then we have
  \[ \tmop{supp} \psi \cap \tmop{supp} u = \varnothing . \]
\end{lemma}

\begin{proof}
  Assume for contradiction that $\mathfrak{n} \in \tmop{supp} \psi \cap
  \tmop{supp} u$. In particular $\mathfrak{n} \succ L_{\beta + 1}
  E_{\alpha}^u$. Since $u > 0$, there is $r \in \mathbb{R}^{>}$ with $u
  \geqslant r\mathfrak{n}$, so $L_{\beta + 1} E_{\alpha}^u \succ
  \mathfrak{n}$: a contradiction.
\end{proof}

\subsection{Paths and subpaths}\label{subsection-paths}

Let $\lambda$ be an ordinal with $0 < \lambda \leqslant \omega$ and note that
$i < 1 \dotplus \lambda \Longleftrightarrow (i \leqslant \lambda < \omega \vee
i < \omega = \lambda)$ for all~$i \in \mathbb{N}$. Consider a
sequence\label{autolab48} \label{autolab49} \label{autolab50}
\[ P = (P (i))_{i < \lambda} = (\tau_{P, i})_{i < \lambda} = (r_{P, i}
   \mathfrak{m}_{P, i})_{i < \lambda} \text{\quad in\quad$(\mathbb{R}^{\neq} 
   \mathbf{Mo})^{\lambda}$} . \]
We say that $P$ is a {\tmem{path}}{\index{path}} if there exist sequences
$(u_{P, i})_{i < 1 \dotplus \lambda}$, $(\psi_{P, i})_{i < 1 \dotplus
\lambda}$, $(\iota_{P, i})_{i < \lambda}$, $(\alpha_{P, i})_{i < \lambda}$,
and $(\beta_{P, i})_{i < 1 \dotplus \lambda}$ with
\begin{itemizedot}
  \item $u_{P, 0} = \tau_{P, 0}$ and $\psi_{P, 0} = 0$;
  
  \item $\tau_{P, i} \in \tmop{term} \psi_{P, i}$ or $\tau_{P, i} \in
  \tmop{term} u_{P, i}$ for all $i < \lambda$;
  
  \item $\tau_{P, i} \in \mathbb{R}^{\neq} \cup \{ \omega \} \Longrightarrow
  \lambda = i + 1$ for all $i < \lambda$;
  
  \item For $i < \lambda$, the hyperserial expansion of $\mathfrak{m}_{P, i}$
  is
  \begin{eqnarray*}
    \mathfrak{m}_{P, i} & = & \mathe^{\psi_{P, i + 1}}  (L_{\beta_{P, i}}
    E_{\alpha_{P, i}}^{u_{P, i + 1}})^{\iota_{P, i}} .
  \end{eqnarray*}
\end{itemizedot}
We call $\lambda$ the {\tmem{length}} of $P$ and we write $| P | \assign
\lambda$\label{autolab51}. We say that $P$ is {\tmem{infinite}} if $| P | =
\omega$ and {\tmem{finite}} otherwise. We set $a_{P, 0} \assign a$. For $0 < i
< | P |$, we define
\begin{eqnarray*}
  (s_{P, i}, a_{P, i}) & \assign & \left\{\begin{array}{ll}
    (- 1, \psi_{P, i}) & \text{{\hspace*{\fill}}if $\mathfrak{m}_{P, i} \in
    \tmop{supp} \psi_{P, i}$}\\
    (1, u_{P, i}) & \text{{\hspace*{\fill}}if $\mathfrak{m}_{P, i} \in
    \tmop{supp} u_{P, i}$} .
  \end{array}\right.
\end{eqnarray*}
By Lemma~\ref{lem-st-expansion-no-sign}, those cases are mutually exclusive so
$(s_{P, i}, a_{P, i})$ is well-defined. For~{$a \in \mathbf{No}$}, we say that
$P$ is a path in $a$ if $P (0) \in \tmop{term} a$.

For $k \leqslant | P |$, we let $P_{\nearrow k}$ denote the path of length $|
P | - k$ in $a_{P, k}$ with
\[ \forall i < | P | - k, \tau_{P_{\nearrow k}, i} \assign \tau_{P, k + i} .
\]
\begin{example}
  Let us find all the paths in the monomial $\mathfrak{m}$ of
  Example~\ref{example-st}. We have a representation
  (\ref{ex-expanded-monomial}) of $\mathfrak{m}$ as a hyperseries
  \[ \mathfrak{m}= \mathe^{2 E_{\omega^2}^{L_{\omega^{\small{2}}} \omega + 1}
     - E_1^{\frac{1}{2} L_1 \omega}}  (L_{\omega} \omega) \]
  which by Lemma~\ref{lem-st-expansion-unicity} is unique. There are nine
  paths in $\mathfrak{m}$, namely
  \begin{itemizedot}
    \item 1 path $(\mathfrak{m})$ of length $1$;
    
    \item 3 paths $\left( \mathfrak{m}, 2
    E_{\omega^2}^{L_{\omega^{\small{2}}} \omega + 1} \right)$, $\left(
    \mathfrak{m}, - E_1^{\frac{1}{2} L_1 \omega} \right)$, and $(\mathfrak{m},
    \omega)$ of length $2$;
    
    \item 3 paths $\left( \mathfrak{m}, 2
    E_{\omega^2}^{L_{\omega^{\small{2}}} \omega + 1}, L_{\omega^2} \omega
    \right)$, $\left( \mathfrak{m}, 2 E_{\omega^2}^{L_{\omega^{\small{2}}}
    \omega + 1}, 1 \right)$, $\left( \! \mathfrak{m}, - E_1^{\frac{1}{2}
    L_1 \omega}, \frac{1}{2} L_1 \omega \! \right)$ of length~$3$;
    
    \item 2 paths $\left( \mathfrak{m}, 2
    E_{\omega^2}^{L_{\omega^{\small{2}}} \omega + 1}, L_{\omega^2} \omega,
    \omega \right)$ and $\left( \mathfrak{m}, - E_1^{\frac{1}{2} L_1 \omega},
    \frac{1}{2} L_1 \omega, \omega \right)$ of length $4$.
  \end{itemizedot}
  Note that the paths which cannot be extended into strictly longer paths are
  those whose last value is a real number or $\omega$.
\end{example}

Infinite paths occur in so-called nested numbers that will be studied in more
detail in~Section~\ref{section-nested-series}.

\begin{definition}
  \label{def-good-path}Let $a \in \mathbf{No}$ and let $P$ be a path in a. We
  say that an index $i < | P |$ is {\tmem{{\tmstrong{bad}}}}{\index{bad
  index}} for~$(P, a)$ if one of the following conditions is satisfied
  \begin{enumeratenumeric}
    \item \label{def-good-path-1}$\mathfrak{m}_{P, i}$ is not the
    $\preccurlyeq$-minimum of $\tmop{supp} u_{P, i}$;
    
    \item \label{def-good-path-2}$\mathfrak{m}_{P, i} = \min \tmop{supp} u_{P,
    i}$ and $\beta_{P, i} \neq 0$;
    
    \item \label{def-good-path-3}$\mathfrak{m}_{P, i} = \min \tmop{supp} u_{P,
    i}$ and $\beta_{P, i} = 0$ and $r_{P, i} \nin \{ - 1, 1 \}$;
    
    \item \label{def-good-path-4}$\mathfrak{m}_{P, i} = \min \tmop{supp} u_{P,
    i}$ and $\beta_{P, i} = 0$ and $r_{P, i} \in \{ - 1, 1 \}$ and
    $\mathfrak{m}_{P, i} \in \tmop{supp} \psi_{P, i}$.
  \end{enumeratenumeric}
  The index $i$ is {\tmem{{\tmstrong{good}}}}{\index{good index}} for $(P, a)$
  if it is not bad for $(P, a)$.
  
  If $P$ is infinite, then we say that it is
  {\tmstrong{{\tmem{good}}}}{\index{good path, bad path}} if~$(P, \tau_{P,
  0})$ is good for all but a finite number of indices. In the opposite case,
  we say that $P$ is a {\tmem{{\tmstrong{bad}}}} path. An element $a \in
  \mathbf{No}$ is said to be
  {\tmem{{\tmstrong{well-nested}}}}{\index{well-nested number}} every path in
  $a$ is good.
\end{definition}

\begin{remark}
  The above definition extends the former definitions of paths
  in~{\cite{vdH:phd,Schm01,BM18}}. More precisely, a path $P$ with with
  $\alpha_{P, i} = 1$ (whence $\psi_{P, i} = 0$) for all $i < | P |$,
  corresponds to a path for these former definitions. The validity of the
  axiom {\tmstrong{T4}} for $\mathbf{No}$ means that those paths are good.
  With Theorem~\ref{th-well-nested}, we will extend this result to all paths.
\end{remark}

\begin{lemma}
  \label{lem-no-path-in-hell}For $\mathfrak{m} \in (L_{\mathbf{On}}
  \omega)^{\pm 1}$ and for any path $P$ in $\mathfrak{m}$, we have $| P |
  \leqslant 2$. For $a \in \mathbb{L} \circ \omega$ and for any path $P$ in
  $a$, we have $| P | \leqslant 3$.
\end{lemma}

\begin{proof}
  Let $\mathfrak{l} \in \mathfrak{L} \setminus \{ - 1 \}$ and let $P$ be
  a~path in $\mathfrak{l} \circ \omega$. If there is an ordinal $\gamma$ with
  $\mathfrak{l}= \ell_{\gamma}$, then the hyperserial expansion of
  $\mathfrak{l} \circ \omega$ is $L_{\gamma} \omega$, so $| P | = 1$ if
  $\gamma = 0$ and $| P | = 2$ otherwise. If there is an ordinal $\gamma$ with
  $\mathfrak{l}= \ell_{\gamma}^{- 1}$, then the hyperserial expansion of
  $\mathfrak{l} \circ \omega$ is $(L_{\gamma} \omega)^{- 1}$ and $| P | = 2$.
  
  Assume now that $\mathfrak{l} \nin \ell_{\mathbf{On}}^{\pm 1}$. If $\log
  \mathfrak{l} \circ \omega$ is not tail-atomic, then hyperserial expansion of
  $\mathfrak{l} \circ \omega$ is $\mathfrak{l} \circ \omega = \mathe^{\log
  \mathfrak{l} \circ \omega}$. If $\log \mathfrak{l} \circ \omega$ is
  tail-atomic, then the hyperserial expansion of $\mathfrak{l} \circ \omega$
  is $\mathfrak{l} \circ \omega = \mathe^{\psi \circ \omega}  (\mathfrak{a}
  \circ \omega)^{\iota}$ for a certain log-atomic $\mathfrak{a} \in
  \mathbb{L}$. In both cases, $P_{\nearrow 1}$ is a path in some monomial in
  $L_{\mathbf{On}} \omega$, whence $| P_{\nearrow 1} | \leqslant 2$ and $| P |
  \leqslant 3$, by the previous argument.
\end{proof}

\begin{definition}
  Let $P, Q$ be paths. We say that $Q$ is a {\tmem{{\tmstrong{subpath
  of}}}}{\index{subpath}} $P$, or equivalently that $P$
  {\tmem{{\tmstrong{extends}}}} $Q$, if there exists a $k < | P |$ with $Q =
  P_{\nearrow k}$. For $a \in \mathbf{No}$, we say that $Q$ is a
  {\tmem{{\tmstrong{subpath in}}}} $a$ if there is a path $P$ in $a$ such that
  $Q$ is a subpath of $P$. We say that $P$ {\tmstrong{{\tmem{shares a subpath
  with}}}} $a$ if there is a subpath of $P$ which is a subpath in $a$.
\end{definition}

Let $P$ be a finite path and let $Q$ be a path with $Q (0) \in \tmop{supp}
u_{P, | P |} \cup \tmop{supp} \psi_{P, | P |}$. Then we define $P \ast
Q$\label{autolab52} to be the path $(P (0), \ldots, P (| P |), Q (0), \ldots)$
of length $| P | + | Q |$.

\begin{lemma}
  \label{lem-hyperlog-subpath}Let $\lambda \in \omega^{\mathbf{On}}$ and
  $\mathfrak{m} \in \mathbf{Mo}_{\lambda}$. Let $P$ be a path in
  $\mathfrak{m}$ with $| P | > 1$. Then $P_{\nearrow 1}$ is a subpath in
  $L_{\lambda} \mathfrak{m}$.
\end{lemma}

\begin{proof}
  By Lemma~\ref{lem-no-path-in-hell}, we have $\mathfrak{m} \nin
  (L_{\mathbf{On}} \omega)^{\pm 1}$. If $\mathfrak{m}$ has a hyperserial
  expansion of the form $\mathfrak{m}= \mathe^{\psi}  (L_{\gamma}
  \omega)^{\iota}$, then $P_{\nearrow 1}$ must be a path in $\psi$. So $\psi$
  is non-zero and thus $\lambda = 1$. It follows that $P_{\nearrow 1}$ is a
  path in $\log \mathfrak{m}= \psi \pplus \iota (L_{\gamma} \omega)^{\iota}$.
  Otherwise, let $\mathfrak{m}= \mathe^{\psi}  (L_{\beta}
  E_{\alpha}^u)^{\iota}$ be the hyperserial expansion of $\mathfrak{m}$. If
  $P_{\nearrow 1}$ is a path in $\psi$, then it is a path in $\log
  \mathfrak{m}$ as above. Otherwise, it is a path in $u$. Assume that $\lambda
  = 1$. If $\alpha = 1$, then we have $\psi = 0$ and $\log \mathfrak{m}= \iota
  u$ so $P_{\nearrow 1}$ is a path in $\log \mathfrak{m}$. If $\alpha > 1$,
  then $\log \mathfrak{m}= \psi \pplus \iota L_{\beta + 1} E_{\alpha}^u$ where
  $L_{\beta + 1} E_{\alpha}^u$ is a hyperserial expansion, so $P_{\nearrow 1}$
  is a path in $\log \mathfrak{m}$. Assume now that $\lambda > 1$, so $\psi =
  0$, $\iota = 1$, and $\alpha \geqslant \omega$. We must have $\beta \ggeq
  \lambda_{/ \omega}$ so there are $\beta' \in \mathbf{On}$ and $n \in
  \mathbb{N}$ with $\beta' \gg \lambda_{/ \omega}$ and $\beta = \beta' +
  \lambda_{/ \omega} n$. We have $L_{\lambda} \mathfrak{m}= L_{\beta' +
  \lambda} E_{\alpha}^u - n$ where $L_{\beta' + \lambda} E_{\alpha}^u$ is a
  hyperserial expansion, so $P_{\nearrow 1}$ is a path in~$L_{\lambda}
  \mathfrak{m}$.
\end{proof}

\begin{lemma}
  \label{lem-hyperexp-subpath}Let $a \in \mathbf{No}^{>, \succ}$, $\alpha \in
  \omega^{\mathbf{On}}$ and $k \in \mathbb{N}^{>}$. If $P$ is a path in
  $\sharp_{\alpha} (a)$ with $| P | > 2$, then $P_{\nearrow 1}$ is a subpath
  in~$\mathfrak{d}_{E_{\alpha k} a}$.
\end{lemma}

\begin{proof}
  We prove this by induction on $\alpha k$, for any number $a \in
  \mathbf{No}^{>, \succ}$. We consider {$a \in \mathbf{No}^{>, \succ}$}, and a
  fixed path $P$ in $\sharp_{\alpha} (a)$ with $| P | > 1$.
  
  Assume that~$\alpha = k = 1$. We have $\sharp_1 (a) = a_{\succ}$ and
  $\mathfrak{d}_{\exp a} = \mathe^{a_{\succ}}$. Assume that $a_{\succ} = \psi
  \pplus \iota \mathfrak{a}$ for certain $\psi \in \mathbf{No}_{\succ}$,
  $\iota \in \{ - 1, 1 \}$, and $\mathfrak{a} \in \mathbf{Mo}_{\omega}$. Let
  $\mathfrak{a}= L_{\gamma} E_{\lambda}^u$ be the hyperserial expansion of
  $\mathfrak{a}$. If $\lambda = \omega$, then $\gamma = 0$ and the hyperserial
  expansion of $\mathe^{\mathfrak{a}}$ is $\mathe^{\mathfrak{a}} =
  E_{\omega}^{u + 1}$. Therefore~$P_{\nearrow 1}$ is a~subpath in
  $\mathfrak{d}_{\exp a} = \mathe^{\psi}  (E_{\omega}^{u + 1})^{\iota}$. If
  $\lambda > \omega$, then the hyperserial expansion of
  $\mathe^{\mathfrak{a}}$ is {$\mathe^{\mathfrak{a}} = L_{\gamma + 1}
  E_{\lambda}^u$}. Therefore $P_{\nearrow 1}$ is a subpath in
  $\mathfrak{d}_{\exp a} = \mathe^{\psi}  (L_{\gamma + 1}
  E_{\lambda}^u)^{\iota}$. Finally, if $\mathe^{a_{\succ}}$ is not
  tail-atomic, then~$P_{\nearrow 1}$ is a subpath in $\mathfrak{d}_{\exp a} =
  \left( E_1^{\smash{\epsilon a_{\succ}}} \right)^{\epsilon}$, where $\epsilon
  \in \{ - 1, 1 \}$ is the sign of $a_{\succ}$.
  
  Now assume that $\alpha = 1$, $k > 1$, and that the result holds strictly
  below $k$. We have $E_k a = E_{k - 1} (\exp a)$ where $P_{\nearrow 1}$ is a
  subpath in $\mathfrak{d}_{\exp a}$ by the previous argument. We have
  $r\mathfrak{d}_{\exp a} \trianglelefteqslant \sharp_1 (\exp a)$ for a
  certain $r \in \mathbb{R}^{\neq}$, so $Q \assign (r\mathfrak{d}_{\exp a})
  \ast P_{\nearrow 1}$ is a path in $\sharp_1 (\exp a)$. The induction
  hypothesis on $k - 1$ implies that $Q_{\nearrow 1} = P_{\nearrow 1}$ is a
  subpath in $\mathfrak{d}_{E_k a}$.
  
  Assume now that $\alpha \geqslant \omega$ and that the result holds strictly
  below $\alpha$. Write $v \assign \sharp_{\alpha} (a)$.
  By~(\ref{eq-trianglehere}), there exist $\eta \in \mathbf{On}$, $n <
  \omega$, and $\delta \in \mathbf{No}$ with $\beta \assign \omega^{\eta} <
  \alpha$ and
  \[ E_{\alpha} a = E_{\beta n} (L_{\beta n} E_{\alpha}^v \pplus \delta) . \]
  Assume for contradiction that there is a $\gamma \in \mathbf{On}$ with
  $E_{\alpha}^v = L_{\gamma} \omega$. We must have $\gamma \ggeq \alpha_{/
  \omega}$, so there are a number $n \in \mathbb{N}$ and an ordinal $\gamma'
  \ggeq \alpha$ with $\gamma = \gamma' + \alpha_{/ \omega} n$. We have $v =
  {L_{\gamma' + \alpha} \omega - n}$. By Lemma~\ref{lem-no-path-in-hell}, this
  contradicts the fact that $| P | > 2$. So by
  Lemma~\ref{lem-st-expansion-existence}, there exist $\beta \in
  \omega^{\mathbf{On}}$ and $\gamma \in \mathbf{On}$ with {$\beta \geqslant
  \alpha$}, $\gamma \omega < \beta$, $E_{\alpha}^v = L_{\gamma} E_{\beta}^u$,
  and $E_{\beta}^u \in \mathbf{Mo}_{\beta} \setminus L_{< \beta} 
  \mathbf{Mo}_{\beta \omega}$. Since $E_{\alpha}^v \in \mathbf{Mo}_{\alpha}$,
  we must have $\gamma \ggeq \alpha_{/ \omega}$ so there are a number $n \in
  \mathbb{N}$ and an ordinal $\gamma' \ggeq \alpha$ with $\gamma = \gamma' +
  \alpha_{/ \omega} n$ (note that $n = 0$ whenever $\alpha_{/ \omega} =
  \alpha$). Thus $v + n = L_{\alpha} L_{\gamma' + \alpha_{/ \omega} n}
  E_{\beta}^u + n = L_{\gamma' + \alpha} E_{\beta}^u$ is a monomial with
  hyperserial expansion $v + n = L_{\gamma' + \alpha} E_{\beta}^u$. There is
  no path in $n$ of length $> 1$, so $P$ must be a path in $L_{\gamma' +
  \alpha} E_{\beta}^u$. We deduce that $P_{\nearrow 1}$ is a path in $u$.
  Consequently, $Q = (L_{\gamma} E_{\beta}^u) \ast P_{\nearrow 1}$ is a path
  in $E_{\alpha}^v$ with $| Q | = | P | > 1$. Applying $n$~times
  Lemma~\ref{lem-hyperlog-subpath}, we deduce that $Q_{\nearrow 1} =
  P_{\nearrow 1}$ is a subpath in $L_{\beta n} E_{\alpha}^v$, hence in
  $\sharp_{\beta} (L_{\beta n} E_{\alpha} a)$. Consider a path~$R$ in
  $\sharp_{\beta} (L_{\beta n} E_{\alpha} a)$ with $P_{\nearrow 1} =
  R_{\nearrow i}$ for a certain $i > 0$. Applying the induction hypothesis
  for~$L_{\beta n} E_{\alpha} a$ and~$\beta n$ in the roles of $a$ and $\alpha
  k$, the path $R_{\nearrow 1}$ is a subpath in~$\mathfrak{d}_{E_{\beta n}
  (L_{\beta n} E_{\alpha} a)} =\mathfrak{d}_{E_{\alpha} a}$. Therefore
  $P_{\nearrow 1}$ is a subpath in~$\mathfrak{d}_{E_{\alpha} a}$. We deduce as
  in the case $\alpha = 1$ that $P_{\nearrow 1}$ is a subpath
  in~$\mathfrak{d}_{E_{\alpha k} a}$.
\end{proof}

\begin{lemma}
  \label{lem-subpath-formula}Let $\psi \in \mathbf{No}_{\succ}$, and
  $\mathfrak{m} \in \mathbf{Mo}^{\neq}$ with $\tmop{supp} \psi \succ \log
  \mathfrak{m}$. Let $P$ be a path in~$\mathfrak{m}$ with $| P | > 1$. Then
  $P_{\nearrow 1}$ is a subpath in $\mathe^{\psi} \mathfrak{m}$.
\end{lemma}

\begin{proof}
  Let $\mathfrak{m}= \mathe^{\varphi}  (L_{\beta} E_{\alpha}^u)^{\iota}$ be a
  hyperserial expansion. The condition $\tmop{supp} \psi \succ \log
  \mathfrak{m}$ implies $\varphi + \psi = \varphi \pplus \psi$, whence
  $\mathe^{\psi} \mathfrak{m}= \mathe^{\psi \pplus \varphi}  (L_{\beta}
  E_{\alpha}^u)^{\iota}$ is also a hyperserial expansion. In particular
  $P_{\nearrow 1}$ is a subpath in $\mathe^{\psi} \mathfrak{m}$.
\end{proof}

\begin{corollary}
  \label{cor-subpath-hypexplog}Let $\alpha = \omega^{\nu} \in \mathbf{On}$,
  $\beta \in \mathbf{On}$ with $\beta < \alpha$, and $\varphi \in
  \mathbf{No}_{\succ, \alpha}$. If $P$ is an infinite path, then $P$ shares a
  subpath with $\varphi$ if and only if it shares a subpath with $L_{\beta}
  E_{\alpha}^{\varphi}$.
\end{corollary}

\begin{proof}
  Write \ $\beta = \omega^{\eta_1} m_1 + \cdots + \omega^{\eta_k} m_k$ in
  Cantor normal form, with $\eta_1 > \cdots > \eta_k$ and {$m_1, \ldots, m_k
  \in \mathbb{N}^{>}$} and let
  \[ \mathfrak{a}_i \assign L_{\omega^{\eta_1} m_1 + \cdots + \omega^{\eta_{i
     - 1}} m_{i - 1}} E_{\alpha}^{\varphi} \]
  for all $i = 1, \ldots, k$.
  
  Assume that $P$ shares a subpath with $\varphi$. In other words, there is a
  path $R$ in $\varphi$ which has a common subpath with $P$. The path $R$ must
  be infinite, so by \Cref{lem-hyperexp-subpath}, it shares a subpath with
  $E_{\alpha}^{\varphi} =\mathfrak{a}_1$. Let us prove by induction on $i = 1,
  \ldots, k$ that $R$ shares a subpath with $E_{\alpha}^{\varphi}
  =\mathfrak{a}_i$. Assuming that this holds for $i < k$, we note that
  $\mathfrak{a}_i$ is $L_{< \omega^{\eta_{i - 1}} \omega}$-atomic, hence $L_{<
  \omega^{\eta_i}}$\mbox{-}atomic. So $P$ shares a subpath with
  $\mathfrak{a}_{i + 1}$ by \Cref{lem-hyperlog-subpath} and the induction
  hypothesis. We conclude by induction that $P$ shares a subpath with
  $\mathfrak{a}_k = L_{\beta} E_{\alpha}^{\varphi}$.
  
  Suppose conversely that $P$ shares a subpath with $L_{\beta}
  E_{\alpha}^{\varphi} =\mathfrak{a}_k$. By induction on {$i = k - 1, \ldots,
  1$}, it follows from \Cref{lem-hyperexp-subpath} that $P$ shares a subpath
  with $\mathfrak{a}_i$. Applying \Cref{lem-hyperlog-subpath} to
  $\mathfrak{a}_1 = E_{\alpha}^{\varphi}$, we conclude that $P$ shares a
  subpath with $\varphi$.
\end{proof}

\subsection{Deconstruction lemmas}\label{subsection-decomposition}

In this subsection, we list several results on the interaction between the
simplicity relation $\sqsubseteq$ and various operations in $(\mathbf{No}, +,
\times, (L_{\alpha})_{\alpha \in \mathbf{On}})$.

\begin{lemma}
  \label{neg-dec-lem}{\tmem{{\cite[Theorem~3.3]{Gon86}}}} For $a, b \in
  \mathbf{No}$, we have
  \[ a \sqsubseteq b \Longleftrightarrow \mathord{-} a \sqsubseteq \mathord{-}
     b. \]
\end{lemma}

\begin{lemma}
  \label{sign-dec-lem}{\tmem{{\cite[Theorem~5.12(a)]{Gon86}}}} For
  $\mathfrak{m} \in \mathbf{Mo}$ and $r \in \mathbb{R}^{\neq}$, we have
  \[ \tmop{sign} (r) \mathfrak{m} \sqsubseteq r\mathfrak{m}. \]
\end{lemma}

\begin{lemma}
  \label{imm-dec-lem}{\tmem{{\cite[Proposition~4.20]{BvdH19}}}} Let $\varphi
  \in \mathbf{No}$. For $\delta, \varepsilon$ with $\delta, \varepsilon \prec
  \tmop{supp} \varphi$, we have
  \begin{eqnarray*}
    \varphi \pplus \delta \sqsubseteq \varphi \pplus \varepsilon &
    \Longleftrightarrow & \delta \sqsubseteq \varepsilon .
  \end{eqnarray*}
\end{lemma}

\begin{lemma}
  \label{inv-dec-lem}{\tmem{{\cite[Corollary~4.21]{BM18}}}} For $\mathfrak{m},
  \mathfrak{n} \in \mathbf{Mo}$, we have
  \begin{eqnarray*}
    \mathfrak{m} \sqsubseteq \mathfrak{n} & \Longleftrightarrow &
    \mathfrak{m}^{- 1} \sqsubseteq \mathfrak{n}^{- 1} .
  \end{eqnarray*}
\end{lemma}

\begin{lemma}
  \label{exp-dec-lem}{\tmem{{\cite[Proposition~4.23]{BM18}}}} Given $\varphi,
  a, b$ in $\mathbf{No}_{\succ}$ with $a, b \prec \tmop{supp} \varphi$, we
  have
  \begin{eqnarray*}
    \mathe^a \sqsubseteq \mathe^b & \Longrightarrow & \mathe^{\varphi \pplus
    a} \sqsubseteq \mathe^{\varphi \pplus b} .
  \end{eqnarray*}
\end{lemma}

\begin{lemma}
  \label{exp-mon-lem}{\tmem{{\cite[Proposition~4.24]{BM18}}}} Given
  $\mathfrak{m}, \mathfrak{n} \in \mathbf{Mo}^{\succ}$ with $\log \mathfrak{m}
  \prec \mathfrak{n}$, we have
  \begin{eqnarray*}
    \mathfrak{m} \sqsubseteq \mathfrak{n} & \Longrightarrow &
    \mathe^{\mathfrak{m}} \sqsubseteq \mathe^{\mathfrak{n}} .
  \end{eqnarray*}
\end{lemma}

\begin{lemma}
  \label{prop-full-exp-dec}Let $\varphi \in \mathbf{No}_{\succ}$ and $r \in
  \mathbb{R}^{\neq}$, let $\mathfrak{m}, \mathfrak{n} \in \mathbf{Mo}^{\succ}
  \cap \mathbf{No}^{\prec \tmop{supp} \varphi}$ with $\mathfrak{m} \in
  \mathcal{E}_{\omega} [\mathfrak{n}]$, and let $\delta \in
  \mathbf{No}_{\succ}$ with $\delta \prec \tmop{supp} \mathfrak{n}$. Then
  \[ \mathfrak{m} \sqsubseteq \mathfrak{n} \Longrightarrow \mathe^{\varphi
     \pplus \tmop{sign} (r) \mathfrak{m}} \sqsubseteq \mathe^{\varphi \pplus
     r\mathfrak{n} \pplus \delta} . \]
\end{lemma}

\begin{proof}
  The condition $\mathfrak{m} \in \mathcal{E}_{\omega} [\mathfrak{n}]$ yields
  $\log \mathfrak{m} \prec \mathfrak{n}$. We have $\mathe^{\mathfrak{m}}
  \sqsubseteq \mathe^{\mathfrak{n}}$ by Lemma~\ref{exp-mon-lem}. The identity
  $\mathe^{\mathbf{Mo}^{\succ}} = \mathbf{Smp}_{\mathcal{P}}$ implies that
  $\mathe^{\mathfrak{m}} \sqsubseteq \mathe^{| r | \mathfrak{n}}$, whence
  $\mathe^{\tmop{sign} (r) \mathfrak{m}} \sqsubseteq \mathe^{r\mathfrak{n}}$
  by Lemma~\ref{inv-dec-lem}. Lemma~\ref{exp-dec-lem} yields $\mathe^{\varphi \pplus
  \tmop{sign} (r) \mathfrak{m}} \sqsubseteq \mathe^{\varphi \pplus
  r\mathfrak{n}}$. Since $\mathe^0 = 1 \sqsubseteq
  \mathe^{\delta} \in \mathbf{No}^{>}$, we may apply Lemma~\ref{exp-dec-lem}
  to $\varphi \pplus r\mathfrak{n}$ and $\varphi \pplus r\mathfrak{n} \pplus
  \delta$ to obtain $\mathe^{\varphi \pplus r\mathfrak{n}} \sqsubseteq
  \mathe^{\varphi \pplus r\mathfrak{n} \pplus \delta}$. We conclude using the
  transitivity of $\sqsubseteq$.
\end{proof}

\begin{lemma}
  \label{hypexp-mon-lem}Let $\alpha \in \omega^{\mathbf{On}}$ with $\alpha >
  1$. For $\varphi, \psi \in \mathbf{No}_{\succ, \alpha}$ with $L_{\alpha}
  E_{< \alpha} \varphi < \psi$, we have
  \[ \varphi \sqsubseteq \psi \Longrightarrow E_{\alpha}^{\varphi} \sqsubseteq
     E_{\alpha}^{\psi} . \]
\end{lemma}

\begin{proof}
  By~(\ref{eq-rich-hyperexp}), we have
  \[ E_{\alpha} \varphi = \left\{ E_{< \alpha} \varphi, \mathcal{E}_{\alpha}
     E_{\alpha}^{\varphi^{\mathbf{No}_{\succ, \alpha}}_L} |
     \mathcal{E}_{\alpha} E_{\alpha}^{\varphi^{\mathbf{No}_{\succ, \alpha}}_R}
     \right\} . \]
  Since $\varphi \sqsubseteq \psi$, we have $\varphi_L^{\mathbf{No}_{\succ,
  \alpha}} \subseteq \psi_L^{\mathbf{No}_{\succ, \alpha}}$ and
  $\varphi_R^{\mathbf{No}_{\succ, \alpha}} \subseteq
  \psi_R^{\mathbf{No}_{\succ, \alpha}}$, whence
  \[ \mathcal{E}_{\alpha} E_{\alpha}^{\varphi^{\mathbf{No}_{\succ, \alpha}}_L}
     < E_{\alpha} \psi <\mathcal{E}_{\alpha}
     E_{\alpha}^{\varphi^{\mathbf{No}_{\succ, \alpha}}_R} . \]
  Furthermore, we have $L_{\alpha} E_{< \alpha} \varphi < \psi$, so~$E_{<
  \alpha} \varphi < E_{\alpha}^{\psi}$. We conclude that $E_{\alpha}^{\varphi}
  \sqsubseteq E_{\alpha}^{\psi}$.
\end{proof}

\subsection{Nested truncation}\label{subsection-nested-truncation-relations}

In {\cite[Section~8]{BM18}}, the authors prove the well-nestedness axiom
$\textbf{T4}$ for $\mathbf{No}$ by relying on a well-founded partial order
$\nonconverted{blacktrianglelefteqslant}_{\tmop{BM}}$\label{autolab53} that is
defined by induction. This relation has the additional property that
\[ \nobracket \forall \nobracket a, b \nobracket \in \mathbf{No}^{\neq}
   \nobracket, \quad a \nonconverted{blacktrianglelefteqslant}_{\tmop{BM}} b
   \Longrightarrow a \sqsubseteq b. \]
In this subsection, we define a similar relation $\lesssim$ on $\mathbf{No}$
that will be instrumental in deriving results on the structure of
$(\mathbf{No}, (L_{\gamma})_{\gamma \in \omega^{\mathbf{On}}})$. However, this
relation does {\tmem{not}} {satisfy} $a \lesssim b \Longrightarrow a
\sqsubseteq b$ for all $a, b \in \mathbf{No}$.

Given $a, b \in \mathbf{No}$, we define
\begin{eqnarray*}
  a \lesssim b & \xLeftrightarrow{\tmop{def}} & \nobracket \exists n \in
  \mathbb{N} \nobracket, \hspace{1.2em} a \lesssim_n b,
\end{eqnarray*}
where $(\lesssim_n)_{n \in \mathbb{N}}$ is a sequence of relations that are
defined by induction on $n$, as follows. For $n = 0$, we set $a \lesssim_0 b$,
if $a \trianglelefteqslant b$ or if there exist decompositions
\begin{eqnarray*}
  a & = & \varphi \pplus \tmop{sign} (r) \mathfrak{m}\\
  b & = & \varphi \pplus r\mathfrak{m} \pplus \delta,
\end{eqnarray*}
with $r \in \mathbb{R}^{\neq}$ and $\mathfrak{m} \in \mathbf{Mo}$. Assuming
that $\lesssim_n$ has been defined, we set $a \lesssim_{n + 1} b$ if we are in
one of the two following configurations:

\begin{descriptioncompact}
  \item[Configuration I] \label{conf1}We may decompose $a$ and $b$ as
  \begin{eqnarray}
    a & = & \varphi \pplus \tmop{sign} (r) \mathe^{\psi} 
    (E_{\alpha}^u)^{\iota}  \label{x-dec-st}\\
    b & = & \varphi \pplus r \mathe^{\psi}  (L_{\beta} E_{\alpha}^v)^{\iota}
    \pplus \delta,  \label{y-dec-st}
  \end{eqnarray}
  where $r \in \mathbb{R}^{\neq}$, $\psi \in \mathbf{No}_{\succ}$, $\alpha \in
  \omega^{\mathbf{On}}$, $\beta \omega < \alpha$, $\iota \in \{ - 1, 1 \}$,
  $u, v \in \mathbf{No}_{\succ, \alpha}$,
  \[ \tmop{supp} \psi \succ \log E_{\alpha}^u, \hspace{1em} L_{\beta + 1}
     E_{\alpha}^v, \]
  and $u \lesssim_n v$. If $\alpha = 1$, then we also require that $\psi = 0$.
  
  \item[Configuration II] \label{conf2}We may decompose $a$ and $b$ as
  \begin{eqnarray}
    a & = & \varphi \pplus \tmop{sign} (r) \mathe^{\psi}  \label{z-dec-st}\\
    b & = & \varphi \pplus r \mathe^{\psi'} \mathfrak{a}^{\iota} \pplus
    \delta,  \label{t-dec-st}
  \end{eqnarray}
  where $r \in \mathbb{R}^{\neq}$, $\psi, \psi' \in \mathbf{No}_{\succ}$,
  $\iota \in \{ - 1, 1 \}$, $\mathfrak{a} \in \mathbf{Mo}_{\omega}, \delta \in
  \mathbf{No}$, $\tmop{supp} \psi' \succ \log \mathfrak{a}$, and $\psi
  \lesssim_n \psi'$.
\end{descriptioncompact}

\begin{warning}
  Taking $\alpha = 1$ in the first configuration, we see that $\lesssim$
  extends Berarducci and Mantova's nested truncation ordering. However, the
  relation $\lesssim$ is neither transitive nor anti-symmetric. Furthermore,
  as we already noted above, we do {\tmem{not}} have $\forall a, b \in
  \mathbf{No}, a \lesssim b \Longrightarrow a \sqsubseteq b$.
\end{warning}

\begin{lemma}
  \label{lem-truncated-truncation}Let $\alpha \in \omega^{\mathbf{On}}$. Let
  $a, b \in \mathbf{No}^{>, \succ}$ be numbers of the form
  \begin{eqnarray*}
    a & = & \varphi \pplus r\mathfrak{m}\\
    b & = & \varphi \pplus s\mathfrak{n} \pplus \delta
  \end{eqnarray*}
  where $\varphi, \delta \in \mathbf{No}$, $r, s \in \mathbb{R}^{\neq}$ with
  $\tmop{sign} (r) = \tmop{sign} (s)$, and $\mathfrak{m}, \mathfrak{n} \in
  \mathbf{Mo}^{\prec}$. If $\mathfrak{m}^{- 1} < E_{\rho} \mathfrak{n}^{- 1}$
  for sufficiently large $\rho < \alpha$, then
  \[ b \in \mathbf{No}_{\succ, \alpha} \Longrightarrow a \in
     \mathbf{No}_{\succ, \alpha} . \]
\end{lemma}

\begin{proof}
  Let $\nu \in \mathbf{On}$ and $\alpha \assign \omega^{\nu}$. Assume for
  contradiction that $b \in \mathbf{No}_{\succ, \alpha}$ and~$a \nin
  \mathbf{No}_{\succ, \alpha}$. Assume first that $a \vartriangleleft b$, so
  $b = a \pplus \delta$. Then $\tmop{supp} b \succ \frac{1}{L_{< \alpha}
  E_{\alpha} b}$. Let $k \in \mathbb{N}^{>}$ be such that {$a \pplus
  k\mathfrak{d}_{\delta} \geqslant b$}. Since $\tmop{supp} (a \pplus
  k\mathfrak{d}_{\delta}) \subseteq \tmop{supp} b$, we deduce that
  $\tmop{supp} (a + k\mathfrak{d}_{\delta}) \succ \frac{1}{L_{< \alpha}
  E_{\alpha} (a + k\mathfrak{d}_{\delta})}$, whence~${a +
  k\mathfrak{d}_{\delta}} \in \mathbf{No}_{\succ, \alpha}$. Modulo replacing
  $b$ by $a + k\mathfrak{d}_{\delta}$, it follow that we may assume without
  loss of generality that $\delta = k\mathfrak{p}$ for some $k \in
  \mathbb{N}^{>}$ and some monomial $\mathfrak{p}$.
  
  On the one hand, $a$ is not $\alpha$-truncated, so there are $\mathfrak{q}
  \in (\tmop{supp} \varphi)_{\prec}$ and $\gamma$ with $0 < \gamma < \alpha$
  and~$a < L_{\alpha}^{\mathord{\uparrow} \gamma} (\mathfrak{q}^{- 1})$. We
  may choose $\gamma = \omega^{\eta} n$ for certain $\eta < \nu$ and $n \in
  \mathbb{N}^{>}$, so~$a < L_{\alpha}^{\mathord{\uparrow} \omega^{\eta} n}
  (\mathfrak{p}^{- 1})$. On the other hand, $a + k\mathfrak{p}$ is
  $\alpha$-truncated, so we have
  \[ a + k\mathfrak{p}> L_{\alpha}^{\mathord{\uparrow} \omega^{\eta}  (n
     +\mathbb{N}^{>})} (\mathfrak{p}^{- 1}) > L_{\alpha}^{\mathord{\uparrow}
     \omega^{\eta} n} (\mathfrak{p}^{- 1}) > a. \]
  We deduce that~$k\mathfrak{p}> L_{\alpha}^{\mathord{\uparrow} \omega^{\eta} 
  (n +\mathbb{N}^{>})} (\mathfrak{p}^{- 1}) - L_{\alpha}^{\mathord{\uparrow}
  \omega^{\eta} n} (\mathfrak{p}^{- 1})$. If $\nu$ is a successor, then
  choosing ${\eta = \nu_-}$, we obtain {$k\mathfrak{p}>
  L_{\alpha}^{\mathord{\uparrow} \omega^{\eta} n} (\mathfrak{p}^{- 1})
  +\mathbb{N}^{>} - L_{\alpha}^{\mathord{\uparrow} \omega^{\eta} n}
  (\mathfrak{p}^{- 1})$}, so $k\mathfrak{p} \succ 1$: a contradiction.
  Otherwise, \ $k\mathfrak{p}> \ell_{[\omega^{\eta + 1}, \alpha)}^{- 1} \circ
  \mathfrak{p}^{- 1}$ by {\cite[(2.4)]{BvdHK:hyp}}, where $\ell_{[\omega^{\eta
  + 1}, \alpha)} \assign \prod_{\omega^{\eta + 1} \leqslant \gamma < \alpha}
  \ell_{\gamma}$. Thus $k^{- 1} \mathfrak{p}^{- 1} < {\ell_{[\omega^{\eta +
  1}, \alpha)} \circ \mathfrak{p}^{- 1}}$, whence $k^{- 1} \ell_0 <
  \ell_{[\omega^{\eta + 1}, \alpha)}$: a~contradiction.
  
  We now treat the general case. By a similar argument as above, we may assume
  without loss of generality that $b = \varphi \pplus s\mathfrak{n}$. Assume
  that $b \leqslant a$. Since $a$ is not $\alpha$-truncated, there exists a
  $\gamma < \alpha$ with $\mathfrak{m} \prec (L_{\gamma} E_{\alpha} a)^{- 1}
  \leqslant (L_{\gamma} E_{\alpha} b)^{- 1}$, whence $\mathfrak{m}^{- 1} \succ
  L_{\gamma} E_{\alpha} b$. But $b$ is $\alpha$\mbox{-}truncated, so
  $\mathfrak{n}^{- 1} \prec L_{< \alpha} E_{\alpha} b$. In particular
  $\mathfrak{n}^{- 1} \prec L_{\gamma} E_{\alpha} b$, so our hypothesis
  $\mathfrak{m}^{- 1} < L_{\rho} \mathfrak{n}^{- 1}$ implies
  that~$\mathfrak{m}^{- 1} \prec L_{\rho} L_{\gamma} E_{\alpha} b \preccurlyeq
  L_{\gamma} E_{\alpha} b$: a contradiction.
  
  Assume now that~$b > a$. As in the first part of the proof, there are $\eta
  < \nu$ and $n < n' < \omega$ with $\varphi \pplus s\mathfrak{n}>
  L_{\alpha}^{\mathord{\uparrow} \omega^{\eta} n'} (\mathfrak{n}^{- 1})$ and
  $L_{\alpha}^{\mathord{\uparrow} \omega^{\eta} n} (\mathfrak{m}^{- 1}) >
  \varphi \pplus r\mathfrak{m}$. Recall that $\mathfrak{m}^{- 1} < E_{\rho}
  \mathfrak{n}^{- 1}$ for sufficiently large $\rho < \alpha$. Take $\eta <
  \nu$ and $n' < \omega$ such that
  \begin{eqnarray}
    L_{\alpha}^{\mathord{\uparrow} \omega^{\eta} n'} (\mathfrak{n}^{- 1}) & >
    & L_{\alpha}^{\mathord{\uparrow} \omega^{\eta}  (n + 1)} (\mathfrak{m}^{-
    1}) \nonumber\\
    L_{\omega^{\eta}} \mathfrak{m}^{- 1} & \prec & \mathfrak{n}^{- 1} \quad
    \text{if $\nu$ is a limit.}  \label{ineq-trunc-limit}
  \end{eqnarray}
  Then $b - a > L_{\alpha}^{\mathord{\uparrow} \omega^{\eta}  (n + 1)}
  (\mathfrak{m}^{- 1}) - L_{\alpha}^{\mathord{\uparrow} \omega^{\eta} n}
  (\mathfrak{m}^{- 1})$. If $\nu$ is a successor, then choosing $\eta = \nu_-$
  yields {$b - a > 1$}, which contradicts the fact that $\mathfrak{m}$ and
  $\mathfrak{n}$ are infinitesimal. So $\nu$ is a limit. Writing $\mathfrak{q}
  \assign \max (\mathfrak{m}, \mathfrak{n})$, we have $b - a \asymp
  \mathfrak{q}$. As in the first part of the proof, we obtain $\mathfrak{q}
  \succcurlyeq {\ell_{[\omega^{\eta + 1}, \alpha)}^{- 1} \circ \mathfrak{m}^{-
  1}}$, so $\mathfrak{q}^{- 1} \preccurlyeq \ell_{[\omega^{\eta + 1}, \alpha)}
  \circ \mathfrak{m}^{- 1} \prec \mathfrak{m}^{- 1}$. In view of
  (\ref{ineq-trunc-limit}), we also obtain {$\mathfrak{q}^{- 1} \prec
  \mathfrak{n}^{- 1}$}, so $\mathfrak{q}^{- 1} \prec \max (\mathfrak{m},
  \mathfrak{n})^{- 1}$: a~contradiction.
\end{proof}

\begin{lemma}
  \label{lem-transitive-ineq}Let $\alpha, \alpha' \in \omega^{\mathbf{On}}$
  with $\alpha' \geqslant \alpha$. For $u, v \in \mathbf{No}^{>, \succ}$, we
  have
  \begin{eqnarray*}
    L_{\alpha} u <\mathcal{E}_{\alpha} v & \Longrightarrow & L_{\alpha'}
    E_{\alpha} u <\mathcal{E}_{\alpha'} E_{\alpha} v.
  \end{eqnarray*}
\end{lemma}

\begin{proof}
  Assume that $L_{\alpha} u <\mathcal{E}_{\alpha} v$. Let $h \in
  \mathcal{E}_{\alpha'}$ and let $h^{\tmop{inv}}$ be its functional inverse in
  $\mathcal{E}_{\alpha'}$. We have $h^{\tmop{inv}} < E_{\alpha'} H_2
  L_{\alpha'}$ by (\ref{ineq4}, \ref{ineq5}), whence $h > E_{\alpha'} H_{1 /
  2} L_{\alpha'}$. Furthermore, $u < E_{\alpha} \mathcal{E}_{\alpha} v$, so
  \begin{equation}
    E_{\alpha} u < E_{\alpha} E_{\alpha} \mathcal{E}_{\alpha} v.
    \label{ineq-trans}
  \end{equation}
  We want to prove that~$E_{\alpha} u < (E_{\alpha'} h E_{\alpha}) v$. By
  (\ref{ineq-trans}), it is enough to prove that there is a~$g \in
  \mathcal{E}_{\alpha}$ such that the inequality $E_{\alpha} E_{\alpha} g
  \leqslant E_{\alpha'} h E_{\alpha}$ holds on $\mathbf{No}^{>, \succ}$.
  
  Assume that $\alpha = \alpha'$. Setting $g \assign H_{1 / 2} \in
  \mathcal{E}_{\alpha}$, we have $L_{\alpha} h E_{\alpha} > g$, whence
  $E_{\alpha} g \leqslant h E_{\alpha}$, and $E_{\alpha} E_{\alpha} g
  \leqslant E_{\alpha} h E_{\alpha}$.
  
  Assume that $\alpha' > \alpha$. We have $E_{\alpha'} H_{1 / 2} > H_2$ so
  $E_{\alpha' 2} H_{1 / 2} > E_{\alpha'} H_2 > E_{\alpha} E_{\alpha'}$
  by~(\ref{ineq2}). Thus~${E_{\alpha'} h > E_{\alpha' 2} H_{1 / 2} L_{\alpha'}
  > E_{\alpha}}$. Consequently, $E_{\alpha'} h E_{\alpha} > E_{\alpha}
  E_{\alpha}$, as claimed.
\end{proof}

Given $a, b \in \mathbf{No}$, we write $[a \leftrightarrow b]$ for the interval
$[\min (a, b), \max (a, b)]$.

\begin{proposition}
  \label{prop-convexity-subpath}For $a, b, c \in \mathbf{No}$ with $a \lesssim
  c$ and $b \in [a \leftrightarrow c]$, any infinite path in $a$ shares
  a~subpath with $b$.
\end{proposition}

\begin{proof}
  We prove this by induction on $n$ with $a \lesssim_n c$. Let $P$ be an
  infinite path in $a$. Assume that $a \lesssim_0 c$. If $a
  \trianglelefteqslant c$, then we have $a \trianglelefteqslant b$ so $P$ is a
  path in $b$. Otherwise, there are $\varphi, \delta \in \mathbf{No}$, $r \in
  \mathbb{R}^{\neq}$ and $\mathfrak{m} \in \mathbf{Mo}$ with $a = \varphi
  \pplus \tmop{sign} (r) \mathfrak{m}$ and $c = \varphi \pplus r\mathfrak{m}
  \pplus \delta$. Then $b = \varphi \pplus s\mathfrak{n} \pplus t$ for certain
  $t \in \mathbf{No}$, $s \in \mathbb{R}^{\neq}$ and $\mathfrak{n} \in
  \mathbf{Mo}$ with $s\mathfrak{n} \in [\tmop{sign} (r) \mathfrak{m}
  \leftrightarrow r\mathfrak{m}]$. We must have $\mathfrak{n}=\mathfrak{m}$.
  If $P$ is a path in $\varphi$, then it is a path in $b$. Otherwise, it is a
  path in $\tmop{sign} (r) \mathfrak{m}$, so $P_{\nearrow 1}$ is a subpath in
  $s\mathfrak{m}$, hence in $b$.
  
  We now assume that $a \lesssim_n c$ where $n > 0$ and that the result holds
  for all $a', b', c' \in \mathbf{No}$ and $k < n$ with $a' \lesssim_k c'$ and
  $b' \in [a' \leftrightarrow c']$. Assume first that $(a, c)$ is in
  \ref{conf1}, and write
  \[ \begin{array}{rcl}
       a & = & \varphi \pplus \tmop{sign} (r) \mathe^{\psi} 
       (E_{\alpha}^u)^{\iota}\\
       c & = & \varphi \pplus r \mathe^{\psi}  (L_{\beta}
       E_{\alpha}^v)^{\iota} \pplus \delta
     \end{array} \mathpunct{\text{\quad with\quad}} u \lesssim_{n - 1} v. \]
  Then we can write $b = \varphi \pplus s\mathfrak{m} \pplus t$ like in the
  case when $n = 0$. If $P$ is a path in~$\varphi$, then it is a path in $b$.
  So we may assume that $P$ is a path in~$\tmop{sign} (r) \mathe^{\psi} 
  (E_{\alpha}^u)^{\iota}$. Note that we have~$\mathfrak{m} \in [\mathe^{\psi} 
  (E_{\alpha}^u)^{\iota} \leftrightarrow \mathe^{\psi}  (L_{\beta}
  E_{\alpha}^v)^{\iota}]$. Setting $\mathfrak{n} \assign (\mathfrak{m}
  \mathe^{- \psi})^{\iota} \in [E_{\alpha}^u \leftrightarrow L_{\beta}
  E_{\alpha}^v]$, we observe that $\tmop{supp} \log \mathfrak{n} \prec
  \tmop{supp} \psi$, whence $\mathe^{\psi} \mathfrak{n}^{\iota}$ is the
  hyperserial expansion of $\mathfrak{m}$. If $P_{\nearrow 1}$ is a path in
  $\psi$, then it is a path in $\mathfrak{m}$.
  
  Suppose that $P_{\nearrow 1}$ is not a path in $\psi$. Assume first that
  {$\alpha = 1$}, so $\psi = 0$, $\beta = 0$, and $P$ is a~path in
  $(E_1^u)^{\iota}$. Then Lemma~\ref{lem-hyperlog-subpath} implies that
  $P_{\nearrow 1}$ is a~subpath in $\iota u$, so $P_{\nearrow 2}$ is a subpath
  in~$u$. Otherwise, consider the hyperserial expansion $E_{\alpha}^u =
  L_{\beta'} E_{\alpha'}^w$, $E_{\alpha'}^w \in \mathbf{Mo}_{\alpha'}
  \setminus L_{< \alpha'}  \mathbf{Mo}_{\alpha' \omega}$ of $E_{\alpha}^u$.
  Since $P_{\nearrow 1}$ is not a path in $\psi$, it must be a path in $w$.
  The number $L_{\beta'} E_{\alpha'}^w$ is $L_{< \alpha}$\mbox{-}atomic, so we
  must have $\alpha' \geqslant \alpha$ and {$\beta' \ggeq \alpha_{/ \omega}$}.
  There are $n \in \mathbb{N}$ and $\beta'' \gg \alpha_{/ \omega}$ such that
  {$\beta' = \beta'' + \alpha_{/ \omega} n$}. Therefore $u = L_{\beta'' +
  \alpha} E_{\alpha'}^w - n$. It follows by \Cref{cor-subpath-hypexplog} that
  $P_{\nearrow 1}$ shares a subpath with~$u$, whence so does $P$.
  
  Let $z \assign \sharp_{\alpha} (L_{\alpha} \mathfrak{n})$. Recall that
  $\mathfrak{n} \in [E_{\alpha}^u \leftrightarrow L_{\beta} E_{\alpha}^v]$, so
  $L_{\alpha} \mathfrak{n} \in [u \leftrightarrow L_{\alpha} L_{\beta}
  E_{\alpha}^v]$. Now (\ref{ineq2}) implies that $L_{\beta} E_{\alpha}^v \in
  \mathcal{E}_{\alpha} [E_{\alpha}^v]$, so $L_{\alpha} L_{\beta} E_{\alpha}^v
  \in L_{\alpha} \mathcal{E}_{\alpha} [E_{\alpha}^v] =\mathcal{L}_{\alpha}
  [v]$. The function $\sharp_{\alpha} =
  \pi_{\mathbf{Smp}_{\mathcal{L}_{\alpha}}}$ is non-decreasing, so \[z =
  \sharp_{\alpha} (L_{\alpha} \mathfrak{n}) \in [u \leftrightarrow
  \sharp_{\alpha} (L_{\alpha} L_{\beta} E_{\alpha}^v)] = [u \leftrightarrow
  v].\] But $u \lesssim_{n - 1} v$, so the induction hypothesis yields that
  $P_{\nearrow 2}$, and thus $P$, shares a subpath with $z$. We deduce with
  Lemma~\ref{lem-hyperexp-subpath} that $P$ shares a subpath
  with~$\mathfrak{n}$, hence with $b$.
  
  Assume now that $(a, c)$ is in \ref{conf2}, and write
  \[ \begin{array}{rcl}
       a & = & \varphi \pplus \tmop{sign} (r) \mathe^{\psi}\\
       c & = & \varphi \pplus r \mathe^{\psi'} \mathfrak{a}^{\iota} \pplus
       \delta
     \end{array} \mathpunct{\text{\quad with\quad}} \psi \lesssim_{n - 1}
     \psi' . \]
  Note that we also have $\psi \lesssim_{n - 1} \psi' \pplus \iota \log
  \mathfrak{a}$. We may again assume that $P_{\nearrow 1}$ is a path in
  $\psi$. Write $b = \varphi \pplus s' \mathfrak{q} \pplus t'$, where $s' \in
  \mathbb{R}^{\neq}$, $t' \in \mathbf{No}$, and~$\mathfrak{q} \in
  [\mathe^{\psi} \leftrightarrow \mathe^{\psi'} \mathfrak{a}^{\iota}] \cap
  \mathbf{Mo}$. Then $\log \mathfrak{q} \in [\psi \leftrightarrow \psi' \pplus
  \iota \log \mathfrak{a}]$ where $\psi \lesssim_{n - 1} \psi' \pplus \iota
  \log \mathfrak{a}$. We deduce by induction that $P$ shares a subpath
  with~$\log \mathfrak{q}$. By Lemma~\ref{lem-hyperexp-subpath}, it follows
  that $P$ shares a~subpath with $\mathfrak{q}$, hence with $b$. This
  concludes the proof.
\end{proof}

\begin{lemma}
  \label{lem-subpath-expansion}Let $\lambda, \alpha \in \omega^{\mathbf{On}}$
  and $\beta \in \mathbf{On}$ with $\beta \omega < \alpha$. Let $a \in
  \mathbf{No}_{\succ, \lambda}$ be of the form
  \[ a = \varphi \pplus r \mathe^{\psi}  (L_{\beta} E_{\alpha}^b)^{\iota}
     \pplus \delta, \]
  with $\varphi \in \mathbf{No}$, $r \in \mathbb{R}^{\neq}$, $\psi \in
  \mathbf{No}_{\succ}$, $b \in \mathbf{No}_{\succ, \alpha}$, $\iota \in \{ -
  1, 1 \}$, $\delta \in \mathbf{No}$ and $\log L_{\beta} E_{\alpha}^b \prec
  \tmop{supp} \psi$. Consider an infinite path $P$ in $c \in
  \mathbf{No}_{\succ, \alpha}$ with $c \lesssim b$.
  \begin{enumerateroman}
    \item \label{lem-subpath-expansion-1}If $\log E_{\alpha}^c \nprec
    \tmop{supp} \psi$, then $P$ shares a subpath with $\psi$.
    
    \item \label{lem-subpath-expansion-2}If $\log E_{\alpha}^c \prec
    \tmop{supp} \psi$ and $\mathe^{\psi}  (E_{\alpha}^c)^{\iota} \nprec
    \tmop{supp} \varphi$, then $P$ shares a subpath with $\varphi$.
    
    \item \label{lem-subpath-expansion-3}If $\log E_{\alpha}^c \prec
    \tmop{supp} \psi$ and $\mathe^{\psi}  (E_{\alpha}^c)^{\iota} \prec
    \tmop{supp} \varphi$ and $a' \assign \varphi \pplus \tmop{sign} (r)
    \mathe^{\psi}  (E_{\alpha}^c)^{\iota} \nin \mathbf{No}_{\succ, \lambda}$,
    then $P$ shares a subpath with $\varphi$.
  \end{enumerateroman}
\end{lemma}

\begin{proof}
  \ref{lem-subpath-expansion-1}. If $\log E_{\alpha}^c \nprec \tmop{supp}
  \psi$, then we have $\psi \neq 0$, so $\alpha > 1$. Let $\mathfrak{m} \in
  \tmop{supp} \psi$ with $\log E_{\alpha}^c \succcurlyeq \mathfrak{m}$. Since
  $\log E_{\alpha}^c$ and $\mathfrak{m}$ are monomials, we have $\mathfrak{m}
  \leqslant \log E_{\alpha}^c$, whence $\mathe^{\mathfrak{m}} \leqslant
  E_{\alpha}^c$. Our assumption that $\mathfrak{m} \in \tmop{supp} \psi \succ
  \log L_{\beta} E_{\alpha}^b$ also implies $\mathe^{\mathfrak{m}} \leqslant
  L_{\beta} E_{\alpha}^b$. Hence $\mathe^{\mathfrak{m}} \in [E_{\alpha}^c
  \leftrightarrow L_{\beta} E_{\alpha}^b]$. Now~$P$ shares a subpath with
  $E_{\alpha}^c$, by Lemma~\ref{lem-hyperexp-subpath}. Since $E_{\alpha}^c
  \lesssim L_{\beta} E_{\alpha}^b$, Proposition~\ref{prop-convexity-subpath}
  next implies that $P$ shares a subpath with $\mathe^{\mathfrak{m}}$. Using
  Lemma~\ref{lem-hyperlog-subpath}, we conclude that $P$ shares a subpath with
  $\mathfrak{m}$, and hence with $\psi$.
  
  \ref{lem-subpath-expansion-2}. Let $\mathfrak{m} \in \tmop{supp} \varphi$
  with $\mathfrak{m} \preccurlyeq \mathe^{\psi}  (E_{\alpha}^c)^{\iota}$. It
  is enough to prove that $P$ shares a subpath with~$\mathfrak{m}$. Since
  $\mathfrak{m}$, $\mathe^{\psi}  (L_{\beta} E_{\alpha}^b)^{\iota}$, and
  $\mathe^{\psi}  (E_{\alpha}^c)^{\iota}$ are monomials, we have
  $\mathe^{\psi}  (L_{\beta} E_{\alpha}^b)^{\iota} \leqslant \mathfrak{m}
  \leqslant \mathe^{\psi}  (E_{\alpha}^c)^{\iota}$. Let $\mathfrak{n} \assign
  (\mathe^{- \psi} \mathfrak{m})^{\iota}$, so that $\mathfrak{n} \in
  [L_{\beta} E_{\alpha}^b \leftrightarrow E_{\alpha}^c]$. In particular, we
  have $\tmop{supp} \psi \succ \log \mathfrak{n} \succ 1$. Moreover
  $E_{\alpha}^c \lesssim L_{\beta} E_{\alpha}^b$, so using
  Lemma~\ref{lem-hyperexp-subpath} and
  Proposition~\ref{prop-convexity-subpath}, we deduce in the same way as above
  that~$P$ shares a subpath with $\mathfrak{n}$. If $\mathfrak{n} \nin
  \mathbf{Mo}_{\omega}$, then $\mathfrak{m}= \mathe^{\psi \pplus \iota \log
  \mathfrak{n}}$ is the hyperserial expansion of $\mathfrak{m}$, so $P$ shares
  a subpath with $\mathfrak{m}$. If $\mathfrak{n} \in \mathbf{Mo}_{\omega}$,
  then the hyperserial expansion of $\mathfrak{n}$ must be of the form
  $\mathfrak{n}= E_{\beta'} E_{\alpha'}^u$, since otherwise $\log
  \mathfrak{n}$ would have at least two elements in its support. We deduce
  that $P$ shares a subpath with $u$ and that the hyperserial expansion of
  $\mathfrak{m}$ is $\mathe^{\psi}  (E_{\beta'} E_{\alpha'}^u)^{\iota}$.
  Therefore $P$ shares a subpath with~$\mathfrak{m}$.
  
  \ref{lem-subpath-expansion-3}. We assume that $a'$ is not
  $\lambda$-truncated whereas $\log E_{\alpha}^c \prec \tmop{supp} \psi$ and
  $\mathe^{\psi}  (E_{\alpha}^c)^{\iota} \prec \tmop{supp} \varphi$. If
  $\lambda = 1$, then we must have $\mathe^{\psi}  (E_{\alpha}^c)^{\iota}
  \preccurlyeq 1$, which means that $\psi < 0$ or that $\psi = 0$ and {$\iota
  = - 1$}. But then $\mathe^{\psi}  (L_{\beta} E_{\alpha}^b)^{\iota}
  \preccurlyeq 1$: a contradiction.
  
  Assume that $\lambda > 1$. By Lemma~\ref{lem-truncated-truncation}, we may
  assume without loss of generality that $\delta = 0$. The assumption on $a'$
  and the fact that $a \in \mathbf{No}^{>, \succ}$ imply that $\varphi$ is
  non-zero. Write
  \begin{eqnarray*}
    \mathfrak{p} & \assign & \mathe^{\psi}  (E_{\alpha}^c)^{\iota} \text{\quad
    and}\\
    \mathfrak{q} & \assign & \mathe^{\psi}  (L_{\beta} E_{\alpha}^b)^{\iota} .
  \end{eqnarray*}
  So $a = \varphi \pplus r\mathfrak{q}$ and $a' = \varphi \pplus \tmop{sign}
  (r) \mathfrak{p}$. Note that $\mathfrak{p}$ must be infinitesimal since $a'$
  is not $\lambda$\mbox{-}truncated. Thus $\mathfrak{q}$ is also
  infinitesimal. By Lemma~\ref{lem-truncated-truncation}, we deduce that $E_{<
  \lambda} \mathfrak{q}^{- 1} \prec \mathfrak{p}^{- 1}$. We have
  $\sharp_{\lambda} (a') \vartriangleleft a'$, so $\sharp_{\lambda} (a') =
  \varphi$, since $a$ and $\varphi \vartriangleleft a$ are both
  $\lambda$-truncated. Since $a'$ is not $\lambda$\mbox{-}truncated, there is
  an ordinal $\gamma < \lambda$ with $\mathfrak{p} \prec (L_{\gamma}
  E_{\lambda}^{\varphi})^{- 1}$. If $\varphi \geqslant a$, then $\mathfrak{q}
  \succ (L_{< \lambda} E_{\lambda}^a)^{- 1}$, because $a$ is
  $\lambda$-truncated. Thus $\mathfrak{q} \succ (L_{< \lambda}
  E_{\lambda}^{\varphi})^{- 1}$. If $\varphi < a$, then $\varphi + (L_{<
  \lambda} E_{\lambda}^{\varphi})^{- 1} \in \mathcal{L}_{\lambda} [\varphi]
  <\mathcal{L}_{\lambda} [a] \ni a = \varphi \pplus r\mathfrak{q}$, because
  $\varphi$ and $a$ are $\lambda$-truncated. Now $r > 0$, since $\varphi < a$.
  We again deduce that~$\mathfrak{q} \succ (L_{< \lambda}
  E_{\lambda}^{\varphi})^{- 1}$.
  
  In both cases, we have $L_{\gamma} E_{\lambda}^{\varphi} \in
  [\mathfrak{p}^{- 1} \leftrightarrow \mathfrak{q}^{- 1}]$ where
  $\mathfrak{p}^{- 1} \lesssim \mathfrak{q}^{- 1}$, so $P$ shares a subpath
  with~$L_{\gamma} E_{\lambda}^{\varphi}$, by
  Proposition~\ref{prop-convexity-subpath}. It follows by
  \Cref{cor-subpath-hypexplog} that $P$ shares a subpath with~$\varphi$.
\end{proof}

\subsection{Well-nestedness}\label{subsection-well-nestedness}

We now prove that every number is well-nested. Throughout this subsection, $P$
will be an infinite path inside a number $a \in \mathbf{No}$. At the beginning
of Section~\ref{subsection-paths} we have shown how to attach sequences
$(r_{P, i})_{i < \omega}$, $(\mathfrak{m}_{P, i})_{i < \omega}$,
{\tmabbr{etc.}} to this path. In order to alleviate notations, we will
abbreviate $r_i \assign r_{P, i}$, $\mathfrak{m}_i \assign \mathfrak{m}_{P,
i}$, $u_i \assign u_{P, i}$, $\psi_i \assign \psi_{P, i}$, $\iota_i \assign
\iota_{P, i}$, $\alpha_i \assign \alpha_{P, i}$, and $\beta_i \assign
\beta_{P, i}$ for all~$i \in \mathbb{N}$.

We start with a technical lemma that will be used to show that the existence
of a bad path $P$ in $a$ implies the existence of a bad path in a strictly
simpler number than $a$.

\begin{lemma}
  \label{lem-subpath-descent}Let $a \in \mathbf{No}$, let $P$ be an infinite
  path in $a$ and let $i \in \mathbb{N}$ such that every index~$k \leqslant i$
  is good for $(P, a)$. For $k \leqslant i$, let $\varphi_k \assign
  (u_k)_{\succ \mathfrak{m}_k}$, $\varepsilon_k \assign r_k$, and $\rho_k
  \assign (u_k)_{\prec \mathfrak{m}_k}$, so that $\nobracket \varepsilon_0,
  \ldots, \varepsilon_{i - 1} \nobracket \in \{ - 1, 1 \}$ and
  \begin{eqnarray*}
    u_k & = & \varphi_k \pplus \varepsilon_k \mathe^{\psi_{k + 1}} 
    (E_{\alpha_k}^{u_{k + 1}})^{\iota_k} \hspace*{\fill} (k < i)\\
    u_i & = & \varphi_i \pplus r_i \mathe^{\psi_{i + 1}}  (L_{\beta_i}
    E_{\alpha_i}^{u_{i + 1}})^{\iota_i} \pplus \rho_i .
  \end{eqnarray*}
  Let $\chi \in \{ 0, 1 \}$ and let $c_i \in \mathbf{No}_{\succ, \alpha_{i -
  1}}$ be a number with $c_i \lesssim u_i$ and
  \begin{equation}
    c_i = \varphi_i \pplus \chi \tmop{sign} (r_i) \mathe^{\psi_{i + 1}}
    \mathfrak{p}^{\iota_i}, \label{eq-cj}
  \end{equation}
  for a certain $\mathfrak{p} \in \mathbf{Mo}^{\succcurlyeq}$ with $\log
  \mathfrak{p} \prec \tmop{supp} \psi_{i + 1}$, $\mathfrak{p} \sqsubseteq
  E_{\alpha_i}^{u_{i + 1}}$ and $\mathfrak{p} \in \mathcal{E}_{\omega}
  [E_{\alpha_i}^{u_{i + 1}}]$ whenever $\psi_{i + 1} = 0$. For~$k = i - 1,
  \ldots, 0$, we define
  \begin{equation}
    c_k \assign \varphi_k + \varepsilon_k \mathe^{\psi_{k + 1}} 
    (E_{\alpha_k}^{c_{k + 1}})^{\iota_k} \label{eq-config-I}
  \end{equation}
  Assume that $P$ shares a subpath with $c_i$. If $P$ shares no subpath with
  any of the numbers $\varphi_0, \psi_1, \ldots, \varphi_{i - 1}, \psi_i$,
  then we have~$c_0 \sqsubseteq a$, and $P$ shares a subpath with $c_0$.
\end{lemma}

\begin{proof}
  Using backward induction on $k$, let us prove for $k = i - 1, \ldots, 0$
  that
  \begin{eqnarray}
    L_{\alpha_k} c_{k + 1} & < & \mathcal{E}_{\alpha_k} u_{k + 1}
    \label{induc-3}\\
    \log E_{\alpha_k}^{c_{k + 1}} & \prec & \tmop{supp} \psi_{k + 1} 
    \label{induc-4}\\
    \mathe^{\psi_{k + 1}}  (E_{\alpha_k}^{c_{k + 1}})^{\iota_k} & \prec &
    \tmop{supp} \varphi_k  \label{induc-5}\\
    c_k & \lesssim & u_k \label{induc-7}\\
    \hspace*{\fill} \mathord{\text{$P$ shares a subpath}} & \text{with} & c_{k + 1}
    \label{induc-6} \\
    c_{k + 1} & \in & \mathbf{No}_{\succ, \alpha_k}  \label{induc-2}\\
    c_{k + 1} & \sqsubseteq & u_{k + 1}  \label{induc-1}
  \end{eqnarray}
  and that~(\ref{induc-6}) and~(\ref{induc-1}) also
  hold for $k = - 1$. Given $j\leq i-1$, we use the index notation (\ref{induc-3})\tmrsub{$i$} to refer to the equation (\ref{induc-3}) for $k=i$, and so on for the other conditions.
  
  We first treat the case when $k = i - 1$. Note that $c_i \neq 0$ since it
  contains a subpath, so $\varphi_i \in \mathbf{No}^{>, \succ}$ or $\chi = 1$.
  From our assumption that $c_i = \varphi_i \pplus \chi \tmop{sign} (r_i)
  \mathe^{\psi_{i + 1}} \mathfrak{p}^{\iota_i}$ and the fact that
  $\mathfrak{p} \in \mathcal{E}_{\omega} [E_{\alpha_i}^{u_{i + 1}}]$ if
  $\psi_{i + 1} = 0$, we deduce that $c_i \in \mathcal{E}_{\omega} [u_i]$.
  Hence {$L_{\alpha_{i - 1}} c_i <\mathcal{E}_{\alpha_{i - 1}} u_i$}
  and~(\ref{induc-3})\tmrsub{$i - 1$}. Note that (\ref{induc-6})\tmrsub{$i -
  1$} and (\ref{induc-2})\tmrsub{$i - 1$} follow immediately from the other
  assumptions on~$c_i$. If $\chi = 0$ then $c_i = \varphi_i
  \trianglelefteqslant u_i$. If $\chi = 1$, then $\mathfrak{p} \sqsubseteq
  L_{\beta_i} E_{\alpha_i}^{u_{i + 1}}$, since $L_{\beta_i} E_{\alpha_i}^{u_{i
  + 1}} \in \mathcal{E}_{\alpha_i} [E_{\alpha_i}^{u_{i + 1}}]$ and
  $\mathfrak{p} \sqsubseteq E_{\alpha_i}^{u_{i + 1}} \sqsubseteq
  \mathcal{E}_{\alpha_i} [E_{\alpha_i}^{u_{i + 1}}]$. Hence
  $\mathfrak{p}^{\iota_i} \sqsubseteq (L_{\beta_i} E_{\alpha_i}^{u_{i +
  1}})^{\iota_i}$ by \Cref{inv-dec-lem} and $\tmop{sign} (r_i) \mathe^{\psi_{i
  + 1}} \mathfrak{p}^{\iota_i} \sqsubseteq r_i \mathe^{\psi_{i + 1}} 
  (L_{\beta_i} E_{\alpha_i}^{u_{i + 1}})^{\iota_i}$ by
  Lemmas~\ref{sign-dec-lem} and~\ref{exp-dec-lem}. Finally, $c_i \sqsubseteq
  u_i$ by Lemma~\ref{imm-dec-lem}, so (\ref{induc-1})\tmrsub{$i - 1$} holds in
  general. Recall that $P$ is a subpath in~$c_i$, but that it shares no
  subpath with $\psi_i$ or~$\varphi_{i - 1}$. In view of
  (\ref{induc-2})\tmrsub{$i - 1$}, we deduce (\ref{induc-4})\tmrsub{$i - 1$}
  from Lemma~\ref{lem-subpath-expansion}(\ref{lem-subpath-expansion-1}) and
  (\ref{induc-5})\tmrsub{$i - 1$} from
  Lemma~\ref{lem-subpath-expansion}(\ref{lem-subpath-expansion-2}). Combining
  (\ref{induc-4})\tmrsub{$i - 1$}, (\ref{induc-5})\tmrsub{$i - 1$} and
  (\ref{induc-2})\tmrsub{$i - 1$} with the relation $c_i \lesssim u_i$, we
  finally obtain (\ref{induc-7})\tmrsub{$i - 1$}.
  
  Let $k \in \{ 0, \ldots, i - 1 \}$ and assume that
  (\ref{induc-3}--\ref{induc-1})\tmrsub{$\ell$} hold for all $\ell \geqslant
  k$. We shall prove~(\ref{induc-3}--\ref{induc-1})\tmrsub{$k - 1$} if $k >
  0$, as well as (\ref{induc-6})\tmrsub{$- 1$} and~(\ref{induc-1})\tmrsub{$-
  1$}. Recall that
  \[ c_k = \varphi_k + \varepsilon_k \mathe^{\psi_{k + 1}} 
     (E_{\alpha_k}^{c_{k + 1}})^{\iota_k} . \]
  \begin{descriptioncompact}
    \item[(\ref{induc-3})\tmrsub{$k - 1$}] Recall that $k > 0$. If $\varphi_k
    \neq 0$ or $\psi_{k + 1} \neq 0$, then $c_k \in \mathcal{P} [u_k]$
    and~(\ref{induc-4}--\ref{induc-5})\tmrsub{$k$}
    imply~(\ref{induc-3})\tmrsub{$k - 1$}. Assume now that $\varphi_k =
    \psi_{k + 1} = 0$. It follows since $k > 0$ that {$\iota_k = 1$}, so $c_{k
    - 1} = E_{\alpha_{k - 1}}^{c_k}$ and $u_{k - 1} = E_{\alpha_{k - 1}} u_k$.
    Since $E_{\alpha_{k - 1}}^{u_k}$ is a hyperserial expansion, we must have
    $u_k \nin \mathbf{Mo}_{\alpha_{k - 1} \omega}$, so $\alpha_{k - 1}
    \geqslant \alpha_k$. The result now follows from
    (\ref{induc-3})\tmrsub{$k$} and Lemma~\ref{lem-transitive-ineq}.
    
    \item[(\ref{induc-6})\tmrsub{$k - 1$}] We know by
    (\ref{induc-6})\tmrsub{$k$} that $P$ shares a subpath with $c_{k + 1}$.
    Since $c_{k + 1} \in \mathbf{No}_{\succ, \alpha_k}$, we deduce with
    \Cref{cor-subpath-hypexplog} that $P$ also shares a subpath with
    $E_{\alpha_k}^{c_{k + 1}}$, hence with $(E_{\alpha_k}^{c_{k +
    1}})^{\iota_k}$. In view of (\ref{induc-4})\tmrsub{$k$} and
    \Cref{lem-subpath-formula}, we see that $P$ shares a subpath with
    $\mathe^{\psi_{k + 1}}  (E_{\alpha_k}^{c_{k + 1}})^{\iota_k}$ . Hence
    (\ref{induc-5})\tmrsub{$k$} gives that $P$ shares a subpath with $c_k$.
    
    \item[(\ref{induc-4})\tmrsub{$k - 1$}]  By~(\ref{induc-7})\tmrsub{$k$}, we
    have $c_k \lesssim u_k$. Now $P$ shares a subpath with $c_k$ by
    (\ref{induc-6})\tmrsub{$k$}, but it shares no subpath with~$\psi_k$.
    Lemma~\ref{lem-subpath-expansion}(\ref{lem-subpath-expansion-1}) therefore
    yields the desired result $\log E_{\alpha_{k - 1}}^{c_k} \prec \tmop{supp}
    \psi_k$.
    
    \item[(\ref{induc-5})\tmrsub{$k - 1$}] As above, $P$ shares a subpath with
    $c_k$, but no subpath with~$\varphi_{k - 1}$. We also have $c_k \lesssim
    u_k$ and $\log E_{\alpha_{k - 1}}^{c_k} \prec \tmop{supp} \psi_k$, so
    (\ref{induc-5})\tmrsub{$k - 1$} follows from
    Lemma~\ref{lem-subpath-expansion}(\ref{lem-subpath-expansion-2}).
    
    \item[(\ref{induc-7})\tmrsub{$k - 1$}] We obtain (\ref{induc-7})\tmrsub{$k
    - 1$} by combining (\ref{induc-3}--\ref{induc-7})\tmrsub{$k$} and
    (\ref{induc-2})\tmrsub{$k$}.
    
    \item[(\ref{induc-2})\tmrsub{$k - 1$}] The path $P$ shares a subpath with
    $c_k$, but no subpath with $\varphi_k$. By what precedes, we also have
    $\log E_{\alpha_k}^{c_{k + 1}} \prec \tmop{supp} \psi_k$ and
    $\mathe^{\psi_k}  (E_{\alpha_k}^{c_{k + 1}})^{\iota_k} \prec \tmop{supp}
    \varphi_k$. Note finally that $u_k \in \mathbf{No}_{\succ, \alpha_{k -
    1}}$. Hence {$c_k \in \mathbf{No}_{\succ, \alpha_{k - 1}}$}, by applying
    Lemma~\ref{lem-subpath-expansion}(\ref{lem-subpath-expansion-3}) with
    $\alpha_k$, $\alpha_{k - 1}$, $u_k$, $u_{k + 1}$, and $c_{k + 1}$ in the
    roles of $\alpha$, $\lambda$, $a$, $b$, and $c$.
    
    \item[(\ref{induc-1})\tmrsub{$k - 1$}] It suffices to prove that
    $E_{\alpha_k}^{c_{k + 1}} \sqsubseteq E_{\alpha_k}^{u_{k + 1}}$, since
    \begin{align*}
      &  & E_{\alpha_k}^{c_{k + 1}} \sqsubseteq E_{\alpha_k}^{u_{k + 1}} && \\
      & \Longrightarrow & (E_{\alpha_k}^{c_{k + 1}})^{\iota_k} \sqsubseteq
      (E_{\alpha_k}^{u_{k + 1}})^{\iota_k} && \text{(Lemma~\ref{inv-dec-lem})}\\
      & \Longrightarrow & \mathe^{\psi_{k + 1}}  (E_{\alpha_k}^{c_{k +
      1}})^{\iota_k} \sqsubseteq \mathe^{\psi_{k + 1}}  (E_{\alpha_k}^{u_{k +
      1}})^{\iota_k} && \text{(Lemma~\ref{exp-dec-lem})}\\
      & \Longrightarrow & \varepsilon_k \mathe^{\psi_{k + 1}} 
      (E_{\alpha_k}^{c_{k + 1}})^{\iota_k} \sqsubseteq \varepsilon_k
      \mathe^{\psi_{k + 1}}  (E_{\alpha_k}^{u_{k + 1}})^{\iota_k}\\
      & \Longrightarrow & \varphi_k \pplus \varepsilon_k \mathe^{\psi_{k +
      1}}  (E_{\alpha_k}^{c_{k + 1}})^{\iota_k} \sqsubseteq \varphi_k \pplus
      \varepsilon_k \mathe^{\psi_{k + 1}}  (E_{\alpha_k}^{u_{k +
      1}})^{\iota_k} && \text{(Lemma~\ref{imm-dec-lem})}\\
      & \Longrightarrow & c_k \sqsubseteq u_k. &&
    \end{align*}
    Assume that $\alpha_k > 1$ and recall that
    \begin{eqnarray*}
      c_k & = & \varphi_k \pplus \varepsilon_k \mathe^{\psi_{k + 1}}
      (E_{\alpha_k}^{c_{k + 1}})^{\iota_k}\\
      c_{k + 1} & = & \varphi_{k + 1} \pplus \varepsilon_{k + 1}
      \mathe^{\psi_{k + 2}} (E_{\alpha_{k + 1}}^{c_{k + 2}})^{\iota_{k + 1}} .
    \end{eqnarray*}
    By Lemma~\ref{hypexp-mon-lem}, it suffices to prove that $c_{k + 1}
    \sqsubseteq u_{k + 1}$ and that $E_{\gamma} c_{k + 1} < E_{\alpha_k}^{u_{k
    + 1}}$ for all $\gamma < \alpha_k$. The first relation holds by
    (\ref{induc-1})\tmrsub{$k$}. By (\ref{induc-3})\tmrsub{$k$}, we have
    $L_{\alpha_k} c_{k + 1} <\mathcal{E}_{\alpha_k} u_{k + 1}$. Therefore
    $c_{k + 1} < E_{\alpha_k}  \frac{1}{2} u_{k + 1} < L_{< \alpha_k}
    E_{\alpha_k} u_{k + 1}$ by Lemma~\ref{hypexp-mon-lem}. This yields the
    result.
    
    Assume now that $\alpha_k = 1$. For $d = 0, \ldots, i$, let
    \begin{eqnarray*}
      \mathfrak{c}_d & \assign & \mathfrak{d}_{c_d - \varphi_d}\\
      \mathfrak{u}_d & \assign & \mathfrak{d}_{u_d - \varphi_d} .
    \end{eqnarray*}
    We will prove, by a second descending induction on $d = i, \ldots, k - 1$,
    that the monomials $\mathfrak{c}_d$ and $\mathfrak{u}_d$ satisfy the
    premises of Lemma~\ref{prop-full-exp-dec}, {\tmabbr{i.e.}}
    $\mathfrak{c}_d, \mathfrak{u}_d \succ 1$, $\mathfrak{c}_d \in
    \mathcal{E}_{\omega} [\mathfrak{u}_d]$, and $\mathfrak{c}_d \sqsubseteq
    \mathfrak{u}_d$. It will then follow by Lemma~\ref{prop-full-exp-dec} that
    $\mathe^{c_k} \sqsubseteq \mathe^{u_k}$, thus concluding the proof.
    
    If $d = i$, then $\tmop{supp} c_i, \tmop{supp} u_i \succ 1$, because
    $\alpha_{i - 1} = 1$. In particular {$\mathfrak{c}_i, \mathfrak{u}_i \succ
    1$}. Moreover, $\mathfrak{c}_i \sqsubseteq \mathfrak{u}_i$ follows from
    our assumption that $\mathfrak{p} \sqsubseteq E_{\alpha_i}^{u_{i + 1}}$,
    the fact that $E_{\alpha_i}^{u_{i + 1}} \sqsubseteq \mathcal{E}_{\alpha_i}
    [E_{\alpha_i}^{u_{i + 1}}] \ni L_{\beta_i} E_{\alpha_i}^{u_{i + 1}}$, and
    Lemmas~\ref{exp-dec-lem} and~\ref{inv-dec-lem}. If $\psi_{i + 1} \neq 0$,
    then we have $\mathfrak{c}_i \in \mathcal{E}_{\omega} [\mathfrak{u}_i]$
    because $\tmop{supp} \psi_{i + 1} \succ \log \mathfrak{p}, \log
    E_{\alpha_i}^{u_{i + 1}}$. Otherwise, we have $\mathfrak{c}_i
    =\mathfrak{p} \in \mathcal{E}_{\omega} [E_{\alpha_i}^{u_{i + 1}}]
    =\mathcal{E}_{\omega} [\mathfrak{u}_i]$.
    
    Now assume that $d < i$, that the result holds for $d + 1$, and that
    $\alpha_d = 1$. Again {$\alpha_d = 1$} implies that $\mathfrak{c}_{d + 1},
    \mathfrak{u}_{d + 1} \succ 1$. The relation $c_{d + 1} \sqsubseteq u_{d +
    1}$ and Lemmas~\ref{neg-dec-lem},~\ref{sign-dec-lem},
    and~\ref{imm-dec-lem} imply that $\mathfrak{c}_{d + 1} \sqsubseteq
    \mathfrak{u}_{d + 1}$. If $\psi_{d + 2} \neq 0$, then $\mathfrak{c}_{d +
    1} \in \mathcal{E}_{\omega} [\mathfrak{u}_{d + 1}]$
    by~(\ref{induc-4})\tmrsub{$d + 1$}. Otherwise, we have $\iota_{d + 1} =
    1$, because $c_d \in \mathbf{No}_{\succ, 1}$. Since $\alpha_d = 1$, the
    number $u_{d + 1} = \varphi_{d + 1} \pplus \varepsilon_{d + 1}
    E_{\alpha_{d + 1}}^{u_{d + 2}}$ is not tail-atomic, so we must
    have~$\alpha_{d + 1} = 1$. This entails that $\mathfrak{c}_{d + 1} =
    \mathe^{c_{d + 2}}$ and $\mathfrak{u}_{d + 1} = \mathe^{u_{d + 2}}$. By
    the induction hypothesis at {$d + 1$}, we have $\mathfrak{c}_{d + 2} \in
    \mathcal{E}_{\omega} [\mathfrak{u}_{d + 2}]$. We deduce that~$c_{d + 2}
    \in \mathcal{E}_{\omega} [u_{d + 2}]$, so
    \[ \mathfrak{c}_{d + 1} \in \exp \mathcal{E}_{\omega} [u_{d + 2}]
       =\mathcal{E}_{\omega} [\mathe^{u_{d + 2}}] =\mathcal{E}_{\omega}
       [\mathfrak{u}_{d + 1}] . \]
    It follows by induction that (\ref{induc-1})\tmrsub{$k - 1$} is valid.
  \end{descriptioncompact}
  
  This concludes our inductive proof. The lemma follows from
  (\ref{induc-1})\tmrsub{$- 1$} and (\ref{induc-6})\tmrsub{$- 1$}.
\end{proof}

We are now in a position to prove our first main theorem.

\begin{proof*}{Proof of \Cref{th-well-nested}}
  Assume for contradiction that the theorem is false. Let $a$ be a
  $\sqsubseteq$-minimal ill-nested number and let $P$ be a bad path in $a$.
  Let $i \in \mathbb{N}$ be the smallest bad index in $(P, a)$. As in
  Lemma~\ref{lem-subpath-descent}, we define $\varphi_k \assign (u_k)_{\prec
  \mathfrak{m}_k}$, $\rho_k \assign (u_k)_{\succ \mathfrak{m}_k}$, and
  $\varepsilon_k \assign r_k$ for all $k \leqslant i$. We may assume that $i >
  0$, otherwise the number $c_0 \assign \varphi_0 \pplus \tmop{sign} (r_0)
  \mathe^{\psi_1}  (E_{\alpha_0}^{u_1})^{\iota_0}$ is ill-nested and
  {satisfies}~{$c_0 \sqsubset a$}: a contradiction.
  
  Assume for contradiction that there is a $j < i$ such that $\varphi_j$ or
  $\psi_{j + 1}$ is ill-nested. Set $\chi \assign 0$ if $\varphi_j$ is
  ill-nested and $\chi \assign 1$ otherwise. If $\chi = 1$, then $P$ cannot
  share a~subpath with~$\varphi_j$, so $\tmop{supp} \varphi_j \succ
  \mathe^{\psi_{j + 1}}$ by Lemma~\ref{lem-subpath-expansion}, and $\varphi_j
  \pplus \varepsilon_j \mathe^{\psi_{j + 1}}$ is ill-nested. In general, it
  follows that $c_j \assign \varphi_j \pplus \chi \varepsilon_j
  \mathe^{\psi_{j + 1}}$ is ill-nested. Let $Q$ be a bad path in~$c_j$ and set
  $P' \assign (P (0), \ldots, P (j - 1)) \ast Q$. Then we may apply
  Lemma~\ref{lem-subpath-descent} to $j$, $c_j$, and $P'$ in the roles of $i$,
  $c_i$, and $P$. Since $c_j \neq u_j$, this yields an ill-nested number $c_0
  \sqsubset a$: a contradiction.
  
  Therefore the numbers $\varphi_0, \psi_1, \ldots, \varphi_{i - 1}, \psi_i$
  are well-nested. Since $i$ is bad for $(P, a)$, one of the four cases listed
  in Definition~\ref{def-good-path} must occur. We set
  \[ d_i \assign \left\{\begin{array}{lll}
       \varphi_i \pplus \tmop{sign} (r_i) \mathe^{\psi_{i + 1}} & \quad &
       \text{if Definition~\ref{def-good-path}(\ref{def-good-path-4})
       occurs}\\
       \varphi_i \pplus \tmop{sign} (r_i) \mathe^{\psi_{i + 1}} 
       (E_{\alpha_i}^{u_{i + 1}})^{\iota_i} &  & \text{otherwise} .
     \end{array}\right. \]
  By construction, we have $d_i \lesssim u_i$. Furthermore $P$ shares a
  subpath with $d_i$, so there exists a bad path $Q$ in~$d_i$. We have $d_i
  \in \mathbf{No}_{\succ, \alpha_{j - 1}}$ by
  Lemma~\ref{lem-truncated-truncation}. If
  Definition~\ref{def-good-path}(\ref{def-good-path-4}) occurs, then we must
  have $\psi_{i + 1} \neq 0$ so $d_i$ is written as in (\ref{eq-cj}) with
  $d_i$ in the role of $c_i$ and $\mathfrak{p}= \chi = 1$. Otherwise, $d_i$ is
  as in (\ref{eq-cj}) for $\mathfrak{p}= E_{\alpha_i}^{u_{i + 1}}$. Setting
  $P' \assign (P (0), \ldots, P (i - 1)) \ast Q$, it follows that we may apply
  Lemma~\ref{lem-subpath-descent} to $d_i$ and~$P'$ in the roles of $c_i$ and
  $P$. We conclude that there exists an ill-nested number $d_0 \sqsubset a$: a
  contradiction.
\end{proof*}

\section{Surreal substructures of nested numbers}\label{section-nested-series}

In the previous section, we have examined the nature of infinite paths in
surreal numbers and shown that they are ultimately ``well-behaved''. In this
section, we work in the opposite direction and show how to construct surreal
numbers that contain infinite paths of a specified kind. We follow the same
method as in {\cite[Section~8]{BvdH19}}.

Let us briefly outline the main ideas. Our aim is to construct ``nested
numbers'' that correspond to nested expressions like
\begin{eqnarray}
  a & = & \sqrt{\omega} + \mathe^{\sqrt{\log \omega} - \mathe^{\sqrt{\log \log
  \omega} + \mathe^{\sqrt{\log \log \log \omega} - \mathe^{\udots}}}} 
  \label{a-nested}
\end{eqnarray}
Nested expressions of this kind will be presented through so-called
{\tmem{coding sequences}} $\Sigma$. Once we have fixed such a coding sequence
$\Sigma$, numbers $a$ of the form~(\ref{a-nested}) need to satisfy a sequence
of natural inequalities: for any $r \in \mathbb{R}$ with $r > 1$, we require
that
\[ \begin{array}{rcccl}
     \frac{1}{r}   \sqrt{\omega} & < & a & < & r \sqrt{\omega}\\
     \sqrt{\omega} + \mathe^{\frac{1}{r}  \sqrt{\log \omega}} & < & a & < &
     \sqrt{\omega} + \mathe^{r \sqrt{\log \omega}}\\
     \sqrt{\omega} + \mathe^{\sqrt{\log \omega} - \mathe^{r \sqrt{\log \log
     \omega}}} & < & a & < & \sqrt{\omega} + \mathe^{\sqrt{\log \omega} -
     \mathe^{\frac{1}{r}   \sqrt{\log \log \omega}}}\\
     &  & \vdots &  & 
   \end{array} \]
Numbers that satisfy these constraints are said to be {\tmem{admissible}}.
Under suitable conditions, the class {\tmstrong{Ad}} of admissible numbers
forms a convex surreal substructure. This will be detailed in
Sections~\ref{subsection-coding-sequences}
and~\ref{subsection-admissible-sequences}, where we will also introduce
suitable coordinates
\[ \begin{array}{ccccc}
     a_{; 0} & = & \sqrt{\omega} + \mathe^{\sqrt{\log \omega} -
     \mathe^{\sqrt{\log \log \omega} + \mathe^{\sqrt{\log \log \log \omega} -
     \mathe^{\udots}}}} & = & a\\
     a_{; 1} & = & \sqrt{\log \omega} - \mathe^{\sqrt{\log \log \omega} +
     \mathe^{\sqrt{\log \log \log \omega} - \mathe^{\udots}}} & = & \log
     \left( a_{; 0} - \sqrt{\omega} \right)\\
     a_{; 2} & = & \sqrt{\log \log \omega} + \mathe^{\sqrt{\log \log \log
     \omega} - \mathe^{\udots}} & = & \log \left( \sqrt{\log \omega} - a_{; 1}
     \right)\\
     &  & \vdots &  & 
   \end{array} \]
for working with numbers in $\mathbf{Ad}$.

The notation~(\ref{a-nested}) also suggests that each of the numbers $a_{; 0}
- \sqrt{\omega}$, $\sqrt{\log \omega} - a_{; 1}$, $\ldots$ should be a
monomial. An admissible number $a \in \mathbf{Ad}$ is said to be
{\tmem{nested}} if this is indeed the case. The main result of this section is
\Cref{th-nested-numbers}, {\tmabbr{i.e.}} that the class $\mathbf{Ne}$ of
nested numbers forms a surreal substructure. In other words, the
notation~(\ref{a-nested}) is ambiguous, but can be disambiguated using a
single surreal parameter.

\subsection{Coding sequences}\label{subsection-coding-sequences}

\begin{definition}
  \label{def-coding-sequence}Let $\Sigma \assign (\varphi_i, \varepsilon_i,
  \psi_i, \iota_i, \alpha_i)_{i \in \mathbb{N}} \in (\mathbf{No} \times \{ -
  1, 1 \} \times \mathbf{No} \times \{ - 1, 1 \} \times
  \omega^{\mathbf{On}})^{\mathbb{N}}$. We say that $\Sigma$ is a
  {\tmem{{\tmstrong{coding sequence}}}}{\index{coding sequence}} if for all $i
  \in \mathbb{N}$, we have
  \begin{enumeratealpha}
    \item $\psi_i \in \mathbf{No}_{\succ}$;
    
    \item $\varphi_{i + 1} \in \mathbf{No}_{\succ, \alpha_i} \cup \{ 0 \}$;
    
    \item $(\alpha_i = 1) \Longrightarrow (\psi_i = 0 \wedge (\psi_{i + 1} = 0
    \Longrightarrow \alpha_{i + 1} = 1))$;
    
    \item $(\varphi_{i + 1} = \psi_{i + 1} = 0) \Longrightarrow (\alpha_i
    \geqslant \alpha_{i + 1} \wedge \varepsilon_{i + 1} = \iota_{i + 1} = 1)$;
    
    \item \label{def-coding-sequence-e}$\exists j > i, (\varphi_j \neq 0 \vee
    \psi_j \neq 0)$.
  \end{enumeratealpha}
\end{definition}

{\noindent}Taking $\alpha_i = 1$ for all $i \in \mathbb{N}$, we obtain a
reformulation of the notion of coding sequences
in~{\cite[Section~8.1]{BvdH19}}. If $\Sigma = (\varphi_i, \varepsilon_i,
\psi_i, \iota_i, \alpha_i)_{i \in \mathbb{N}}$ is a coding sequence and $k \in
\mathbb{N}$, then we write
\[ \Sigma_{\nearrow k} \assign (\varphi_{k + i}, \varepsilon_{k + i}, \psi_{k
   + i}, \iota_{k + i}, \alpha_{k + i})_{i \in \mathbb{N}}, \]
which is also a coding sequence.

\begin{lemma}
  \label{lem-path-sequence}Let $P$ be an infinite path in a number $a \in
  \mathbf{No}$ without any bad index for $a$. Let $\varphi_0 \assign a_{\succ
  \mathfrak{m}_{P, 0}}$ and $\varphi_i \assign (a_{P, i})_{\succ
  \mathfrak{m}_{P, i}}$\, for all $i \in \mathbb{N}^{>}${\hspace{-0.8em}}.
  Then $\Sigma_P \assign (\varphi_i, r_{P, i}, \psi_{P, i + 1}, \iota_{P, i},
  \alpha_{P, i})_{i \in \mathbb{N}}$ is a coding sequence.
\end{lemma}

\begin{proof}
  Let $i \in \mathbb{N}$. We have $r_{P, i} \in \{ - 1, 1 \}$ because $i$ is a
  good index for $(P, a)$. We have {$\psi_{P, i + 1} \in \mathbf{No}_{\succ}$}
  and $a_{P, i + 1} \in \mathbf{No}_{\succ, \alpha_i}$ by the definition of
  hyperserial expansions. If $i > 0$ and $\varphi_i \neq 0$, then we have
  $\varphi_i \in \mathbf{No}^{>, \succ}$ because $a_{P, i} \in \mathbf{No}^{>,
  \succ}$ by the definition of paths. Lemma~\ref{lem-truncated-truncation}
  also yields $\varphi_i \in \mathbf{No}_{\succ, \alpha_i}$. This proves the
  conditions $a$) and $b$) for coding sequences. Assume that $\alpha_i = 1$.
  Then by the definition of hyperserial expansions, we have $\psi_{P, i + 1} =
  0$ and $u_{P, i + 1} = a_{P, i + 1}$ is not tail-atomic. Assume that
  $\psi_{P, i + 2} = 0$. Then $\tmop{supp} u_{P, i + 1} \succ 1$ so
  {$\iota_{P, i + 2} = 1$}. We have $u_{P, i + 1} = \varphi_{i + 1} \pplus
  r_{P, i + 1} \mathfrak{a}$ where $\mathfrak{a} \assign E_{\alpha_{P, i +
  1}}^{u_{P, i + 2}}$ and $u_{P, i + 1}$ is not tail-atomic. This implies that
  $\mathfrak{a}$ is not log-atomic, so $\alpha_{P, i + 1} = 1$. Thus $c$) is
  valid.
  
  Assume that $\varphi_{i + 1} = \psi_{P, i + 2} = 0$. Recall that $a_{P, i +
  1} = r_{P, i + 1}  (E_{\alpha_{P, i + 1}}^{u_{P, i + 2}})^{\iota_{P, i + 1}}
  = u_{P, i + 1} \in \mathbf{No}^{>, \succ}$, so $r_{P, i + 1} = \iota_{P, i +
  1} = 1$. Since $E_{\alpha_{P, i}}^{u_{P, i + 1}} \nin
  \mathbf{Mo}_{\alpha_{P, i} \omega}$, we have $u_{P, i + 1} \nin
  \mathbf{Mo}_{\alpha_{P, i} \omega}$, whence $\alpha_{P, i + 1} \leqslant
  \alpha_{P, i}$. This proves $d$).
  
  Assume now for contradiction that there is an $i_0 \in \mathbb{N}$ with
  $\varphi_{P, j} = \psi_{P, j + 1} = 0$ for all~{$j > i_0$}. By $d$), we have
  $r_{P, j} = \iota_{P, j} = 1$ for all $j > i_0$, and the sequence
  $(\alpha_{P, j})_{j > i_0}$ is non-increasing, hence eventually constant.
  Let $i_1 > i_0$ with $\alpha_{P, i_1} = \alpha_{P, j}$ for all $j > i_1$.
  For $k \in \mathbb{N}$, we have $a_{P, i_1} = E_{\alpha_{P, i_1} k} a_{P,
  i_1 + k}$ so $a_{P, i_1} \in \bigcap_{k \in \mathbb{N}} E_{\alpha_{P, i_1}
  k}  \mathbf{Mo}_{\alpha_{P, i_1}} = \mathbf{Mo}_{\alpha_{P, i_1} \omega}$.
  Therefore $E_{\alpha_{P, i_1}}^{a_{P, i_1 + 1}}$ is $L_{< \alpha_{P, i_1 +
  1} \omega}$\mbox{-}atomic: a~contradiction. We deduce that $e$) holds as
  well.
\end{proof}

We next fix some notations. For all $i, j \in \mathbb{N}$ with $i \leqslant
j$, we define partial functions $\Phi_i$,~$\Phi_{i ;}$ and $\Phi_{j ; i}$ on
$\mathbf{No}$ by
\[ \begin{array}{rcl}
     \Phi_i (a) & \assign & \varphi_i + \varepsilon_i \mathe^{\psi_{i - 1}} 
     (E_{\alpha_{i - 1}} a)^{\iota_{i - 1}} \text{},\\
     \Phi_{j ; i} (a) & \assign & (\Phi_i \circ \cdots \circ \Phi_{j - 1})
     (a),\\
     \Phi_{i ;} & \assign & \Phi_{i ; 0} .
   \end{array} \]
The domains of these functions are assumed to be largest for which these
expressions make sense. We also write
\[ \begin{array}{ccccc}
     \sigma_{i ;} & = & \sigma_{; i} & \assign & \prod_{k < i} \varepsilon_k
     \iota_k\\
     \sigma_{j ; i} & = & \sigma_{i ; j} & \assign & \prod_{i \leqslant k < j}
     \varepsilon_k \iota_k
   \end{array} \]
We note that on their respective domains, the functions $\Phi_i$, $\Phi_{i
;}$, and $\Phi_{j ; i}$ are strictly increasing if $\varepsilon_i \iota_i =
1$, $\sigma_{i ;} = 1$, and $\sigma_{j ; i} = 1$, respectively, and strictly
decreasing in the contrary cases. We will write $\Phi_{; i}$ and $\Phi_{i ;
j}$ for the partial inverses of $\Phi_{i ;}$ and $\Phi_{j ; i}$. We will also
use the abbreviations
\[ \begin{array}{ccc}
     a_{i ;} & \assign & \Phi_{i ;} (a)\\
     a_{; i} & \assign & \Phi_{; i} (a)
   \end{array} \nocomma \qquad \begin{array}{ccc}
     a_{j ; i} & \assign & \Phi_{j ; i} (a)\\
     a_{i ; j} & \assign & \Phi_{i ; j} (a)
   \end{array} \]
For all $i \in \mathbb{N}$, we set
\[ \begin{array}{rclccc}
     L_i^{[1]} & \assign & (\varphi_i - \sigma_{; i} \mathbb{R}^{>}
     \tmop{supp} \varphi_i)_{i ;}\\
     L_i^{[2]} & \assign & (\varphi_i + \varepsilon_i \mathe^{\psi_i -
     \varepsilon_i \sigma_{; i} \mathbb{R}^{>} \tmop{supp} \psi_i})_{i ;}\\
     L_i^{[3]} & \assign & \left\{\begin{array}{ll}
       \varnothing & \left. \text{if } \varphi_{i + 1} = 0 \right.\\
       & \text{or } \sigma_{; i + 1} \varepsilon_{i + 1} = - 1 \quad\\
       (\mathcal{L}_{\alpha_i} \varphi_{i + 1})_{i + 1 ;} & \text{otherwise}
     \end{array}\right.\\
     L_i & \assign & L_i^{[1]} \cup L_i^{[2]} \cup L_i^{[3]}.
   \end{array} \]
as well as 

   \[ \begin{array}{rclccc}
      R_i^{[1]} & \assign & (\varphi_i +
     \sigma_{; i} \mathbb{R}^{>} \tmop{supp} \varphi_i)_{i ;}\\
     R_i^{[2]} & \assign & (\varphi_i + \varepsilon_i \mathe^{\psi_i +
     \varepsilon_i \sigma_{; i} \mathbb{R}^{>} \tmop{supp} \psi_i})_{i ;}\\
     R_i^{[3]} & \assign & \left\{\begin{array}{ll}
       \varnothing & \left. \text{if } \varphi_{i + 1} = 0 \right.\\
       & \text{or } \sigma_{; i + 1} \varepsilon_{i + 1} = 1\\
       (\mathcal{L}_{\alpha_i} \varphi_{i + 1})_{i + 1 ;} & \text{otherwise}
     \end{array}\right.\\
     R_i & \assign &
     R_i^{[1]} \cup R_i^{[2]} \cup R_i^{[3]}\\
     R & \assign & \bigcup_{i
     \in \mathbb{N}} R_i .
   \end{array} \]
Note that
\begin{eqnarray*}
  \varphi_i = 0 & \Longleftrightarrow & L_i^{[1]} = R_i^{[1]} = \varnothing
  \text{\quad and}\\
  \psi_i = 0 & \Longleftrightarrow & L_i^{[2]} = R_i^{[2]} = \varnothing .
\end{eqnarray*}
The following lemma generalizes~{\cite[Lemma~8.1]{BvdH19}}.

\begin{lemma}
  If $a \in (L \divides R)$, then $a_{; i}$ is well defined for all $i \in
  \mathbb{N}$.
\end{lemma}

\begin{proof}
  Let us prove the lemma by induction on $i$. The result clearly holds for~$i
  = 0$. Assuming that $a_{; i}$ is well defined, let $j > i$ be minimal such
  that $\varphi_j \neq 0$ or $\psi_j \neq 0$. Note that we have $\alpha_i
  \geqslant \alpha_{i + 1} \geqslant \cdots \geqslant \alpha_j$, so
  $E_{\alpha_i} \circ E_{\alpha_{i + 1}} \circ \cdots \circ E_{\alpha_j} =
  E_{\gamma}$ where $\gamma = \alpha_i + \alpha_{i + 1} + \cdots + {\nobreak}
  \alpha_j$. Applying~$\Phi_{; i}$ to the inequality
  \[ L_j < a < R_j \;, \]
  we obtain
  \[ \sigma_{; i}  (L_j)_{; i} < \sigma_{; i} a_{; i} < \sigma_{; i}  (R_j)_{;
     i} . \]
  Now if $\varphi_j \neq 0$, then
  \begin{eqnarray*}
    (L_j)_{; i} & \supseteq & \varphi_i + \varepsilon_i \mathe^{\psi_i} 
    (E_{\gamma} (\varphi_j - \sigma_{; j} \mathbb{R}^{>} \tmop{supp}
    \varphi_j))^{\iota_i}\\
    (R_j)_{; i} & \supseteq & \varphi_i + \varepsilon_i \mathe^{\psi_i} 
    (E_{\gamma} (\varphi_j + \sigma_{; j} \mathbb{R}^{>} \tmop{supp}
    \varphi_j))^{\iota_i},
  \end{eqnarray*}
  wso the number $z \assign \sigma_{; i} 
     \frac{a_{; i} - \varphi_i}{\varepsilon_i}$ satisfies 
  \[ \sigma_{; i} \mathe^{\psi_i}  (E_{\gamma} (\varphi_j - \sigma_{; j}
     \mathbb{R}^{>} \tmop{supp} \varphi_j))^{\iota_i} < z < \sigma_{; i} \mathe^{\psi_i} 
     (E_{\gamma} (\varphi_j + \sigma_{; j} \mathbb{R}^{>} \tmop{supp}
     \varphi_j))^{\iota_i} . \]
  Both in the cases when $\sigma_{; i} = 1$ and when $\sigma_{; i} = - 1$, it
  follows that the number $((a_{; i} - \varphi_i) / \varepsilon_i
  \mathe^{\psi_i})^{\iota_i}$ is bounded from below by the hyperexponential
  $E_{\gamma}$ of a number. Thus $a_{; j} = L_{\gamma} (((a_{; i} - \varphi_i)
  / (\varepsilon_i \mathe^{\psi_i}))^{\iota_i})$ is well defined and so is
  each $a_{; k}$ for $i \leqslant k < j$. If $\varphi_j = 0$, then we have
  $\psi_j \neq 0$ and

  \begin{eqnarray*}
    (L_j)_{; i} & \supseteq & \varphi_i + \varepsilon_i \mathe^{\psi_i} 
    (E_{\gamma} (\mathe^{\psi_j - \varepsilon_j \sigma_{; j} \mathbb{R}^{>}
    \tmop{supp} \psi_j}))^{\iota_i},\\
    (R_j)_{; i} & \supseteq & \varphi_i + \varepsilon_i \mathe^{\psi_i} 
    (E_{\gamma} (\mathe^{\psi_j + \varepsilon_j \sigma_{; j} \mathbb{R}^{>}
    \tmop{supp} \psi_j}))^{\iota_i} .
  \end{eqnarray*}
  Hence
  \[ \varepsilon_i \mathe^{\psi_i}  (E_{\gamma} (\mathe^{\psi_j -
     \varepsilon_j \sigma_{; j} \mathbb{R}^{>} \tmop{supp} \psi_j}))^{\iota_i}
     < a_{; i} - \varphi_i < \varepsilon_i \mathe^{\psi_i}  (E_{\gamma}
     (\mathe^{\psi_j + \varepsilon_j \sigma_{; j} \mathbb{R}^{>} \tmop{supp}
     \psi_j}))^{\iota_i} \]
  Both in the cases when $\varepsilon_i = 1$ and when $\varepsilon_i = - 1$,
  it follows that $((a_{; i} - \varphi_i) / \varepsilon_i
  \mathe^{\psi_i})^{\iota_i}$ is bounded from below by the hyperexponential
  $E_{\gamma}$ of a number, so $a_{; j}$ is well defined and so is each $a_{;
  k}$ for $i \leqslant k < j$.
\end{proof}

\subsection{Admissible sequences}\label{subsection-admissible-sequences}

\begin{definition}
  \label{def-admissible-sequence}Let $\Sigma \assign (\varphi_i,
  \varepsilon_i, \psi_i, \iota_i, \alpha_i)_{i \in \mathbb{N}}$ be a coding
  sequence and let $a \in \mathbf{No}$. We say that $a$ is
  $\Sigma$-{\tmstrong{{\tmem{admissible}}}}{\index{admissible number}} if
  $a_{; i}$ is well defined for all $i \in \mathbb{N}$ and
  \begin{eqnarray*}
    a_{; i} & = & \varphi_i \pplus \varepsilon_i \mathe^{\psi_i} 
    (E_{\alpha_i} a_{; i + 1})^{\iota_i},\\
    \tmop{supp} \psi_i & \succ & \log E_{\alpha_i} a_{; i + 1}, \text{ and}\\
    \varphi_{i + 1} & \vartriangleleft & \sharp_{\alpha_i} (a_{; i + 1})
    \text{\quad if $\varphi_{i + 1} \neq 0$} .
  \end{eqnarray*}
  We say that $\Sigma$ is {\tmstrong{{\tmem{admissible}}}}{\index{admissible
  sequence}} if there exists a $\Sigma$-admissible number.
\end{definition}

Note that we do not ask that $\mathe^{\psi_i}  (E_{\alpha_i} a_{; i +
1})^{\iota_i}$ be a hyperserial expansion, nor even that $E_{\alpha_i} a_{; i
+ 1}$ be a monomial. For the rest of the section, we fix a coding sequence
$\Sigma = {(\varphi_i, \varepsilon_i, \psi_i, \iota_i, \alpha_i)}_{i \in
\mathbb{N}}$. We write $\mathbf{Ad}$\label{autolab54} for the class of
$\Sigma$-admissible numbers. If $a \in \mathbf{Ad}$, then the definition of
$\mathbf{Ad}$ implicitly assumes that $a_{; i}$ is well defined for all $i \in
\mathbb{N}$. Note that if $\Sigma$ is admissible, then so is $\Sigma_{\nearrow
k}$ for $k \in \mathbb{N}$. We denote by $\mathbf{Ad}_{\nearrow
k}$\label{autolab55} the corresponding class of $\Sigma_{\nearrow
k}$\mbox{-}admissible numbers.

The main result of this subsection is the following generalization
of~{\cite[Proposition~8.2]{BvdH19}}:

\begin{proposition}
  \label{prop-Ad-identification}We have $\mathbf{Ad} = (L|R)$.
\end{proposition}

\begin{proof}
  Let $a \in \mathbf{Ad} \cup (L|R)$ and let $i \in \mathbb{N}$. We have $a_{;
  i} \in \mathbf{No}^{>, \succ}$. If $\sigma_{; i} = 1$, then $\Phi_{i ;}$ is
  strictly increasing so we have
  \begin{eqnarray*}
    L_i^{[1]} < a < R_i^{[1]} & \Longleftrightarrow & (L_i^{[1]})_{\smash{;
    i}} < a_{; i} < (R_i^{[1]})_{\smash{; i}}\\
    & \Longleftrightarrow & \varphi_i -\mathbb{R}^{>} \tmop{supp} \varphi_i <
    a_{; i} < \varphi_i +\mathbb{R}^{>} \tmop{supp} \varphi_i\\
    & \Longleftrightarrow & a_{; i} - \varphi_i \prec \tmop{supp} \varphi_i\\
    & \Longleftrightarrow & \varphi_i \trianglelefteqslant a_{; i} .
  \end{eqnarray*}
  If $\sigma_{; i} = - 1$, then $\Phi_{i ;}$ is strictly decreasing and
  likewise we obtain $L_{i ;} < a < R_{i ;} \Longleftrightarrow \varphi_i
  \trianglelefteqslant a_{; i}$.
  
  We have $\iota_i \log E_{\alpha_i} a_{; i + 1} = \iota_i  \left( \log
  \frac{a_{; i} - \varphi_i}{\varepsilon_i \mathe^{\psi_i}} \right)$. If
  $\sigma_{; i} = 1$, then $\Phi_{i ;}$ is strictly increasing so we have
  \begin{eqnarray*}
    L_i^{[2]} < a < R_i^{[2]} & \Longleftrightarrow & \varphi_i +
    \varepsilon_i \mathe^{\psi_i - \varepsilon_i \mathbb{R}^{>} \tmop{supp}
    \psi_i} < a_{; i} < \varphi_i + \varepsilon_i \mathe^{\psi_i +
    \varepsilon_i \mathbb{R}^{>} \tmop{supp} \psi_i}\\
    & \Longleftrightarrow & -\mathbb{R}^{>} \tmop{supp} \psi_i < \log
    \frac{a_{; i} - \varphi_i}{\varepsilon_i \mathe^{\psi_i}} <\mathbb{R}^{>}
    \tmop{supp} \psi_i\\
    & \Longleftrightarrow & \tmop{supp} \psi_i \succ \log \frac{a_{; i} -
    \varphi_i}{\varepsilon_i \mathe^{\psi_i}}\\
    & \Longleftrightarrow & \log E_{\alpha_i} a_{; i + 1} \prec \tmop{supp}
    \psi_i .
  \end{eqnarray*}
  Likewise, we have $L_i^{[2]} < a < R_i^{[2]} \Longleftrightarrow \log
  E_{\alpha_i} a_{; i + 1} \prec \tmop{supp} \psi_i$ if $\sigma_{; i} = - 1$.
  
  Assume that $\varphi_{i + 1} \neq 0$ and $\sigma_{; i + 1} = 1$. If
  $\varepsilon_{i + 1} = 1$, then we have $a_{; i + 1} > \varphi_{i + 1}$.
  Hence
  \begin{eqnarray*}
    L_i^{[3]} \cup L_{i + 1}^{[1]} < a < R_i^{[3]} \cup R_{i + 1}^{[1]} &
    \Longleftrightarrow & \mathcal{L}_{\alpha_i} \varphi_{i + 1} < a_{; i + 1}
    \wedge \varphi_{i + 1} \trianglelefteqslant a_{; i + 1}\\
    & \Longleftrightarrow & \varphi_{i + 1} < \sharp_{\alpha_i} (a_{; i + 1})
    \wedge \varphi_{i + 1} \trianglelefteqslant a_{; i + 1}\\
    & \Longleftrightarrow & \varphi_{i + 1} \vartriangleleft
    \sharp_{\alpha_i} (a_{; i + 1}) .
  \end{eqnarray*}
  If $\varepsilon_{i + 1} = - 1$, then we have $a_{; i + 1} > \varphi_{i +
  1}$, whence
  \begin{eqnarray*}
    L_i^{[3]} \cup L_{i + 1}^{[1]} < a < R_i^{[3]} \cup R_{i + 1}^{[1]} &
    \Longleftrightarrow & a_{; i + 1} <\mathcal{L}_{\alpha_i} \varphi_{i + 1}
    \wedge \varphi_{i + 1} \trianglelefteqslant a_{; i + 1}\\
    & \Longleftrightarrow & \sharp_{\alpha_i} (a_{; i + 1}) < \varphi_{i + 1}
    \wedge \varphi_{i + 1} \trianglelefteqslant a_{; i + 1}\\
    & \Longleftrightarrow & \varphi_{i + 1} \vartriangleleft
    \sharp_{\alpha_i} (a_{; i + 1}) .
  \end{eqnarray*}
  Symmetric arguments apply when $\varphi_{i + 1} \neq 0$ and $\sigma_{; i +
  1} = - 1$.
  
  We deduce by definition of $\mathbf{Ad}$ that $\mathbf{Ad} = \bigcap_{i \in
  \mathbb{N}} (L_i |R_i) = (L|R)$.
\end{proof}

As a consequence of this last proposition and
{\cite[Proposition~4.29(a)]{BvdH19}}, the class $\mathbf{Ad}$ is a surreal
substructure if and only if $\Sigma$ is admissible.

\begin{example}
  \label{ex-admissible}Consider the coding sequence $\Sigma_0 = (\varphi_i,
  \varepsilon_i, \iota_i, \psi_i, \alpha_i)_{i \in \mathbb{N}}$ where for all
  $i \in \mathbb{N}$, we have
  \begin{eqnarray*}
    \varphi_i & = & L_{\omega^{2 i}} \omega + L_{\omega^{2 i} 2} \omega +
    L_{\omega^{2 i} 3} \omega + \cdots,\\
    \varepsilon_i & = & 1,\\
    \psi_i & = & L_{\omega^{2 i + 1}} \omega + L_{\omega^{2 i + 1} 2} \omega +
    L_{\omega^{2 i + 1} 3} \omega + \cdots,\\
    \iota_i & = & - 1 \text{\quad and}\\
    \alpha_i & = & \omega^{2 i + 1} .
  \end{eqnarray*}
  We use the notations from Section~\ref{subsection-coding-sequences}. We
  claim that $\Sigma_0$ is admissible. Indeed for $i \in \mathbb{N}$, set
  \[ a_i \assign \varphi_0 + \mathe^{\psi_0}  \left( E_{\omega}^{\varphi_1 +
     \mathe^{\psi_1}  \left( E_{\small{\omega^3}}^{\udots^{\varphi_i}}
     \right)^{- 1}} \right)^{- 1} . \]
  Given $j \in \mathbb{N}$ and $i > j$, we have $L_j < a_i$ and $a_i < R_j$.
  We deduce that $L < R$, whence $\Sigma_0$ is admissible.
\end{example}

\begin{lemma}
  \label{lem-little-more}Let $a \in \mathbf{Ad}$ and $b \in \mathbf{No}$ be
  such that $a - \varphi_0$ and $b - \varphi_0$ have the same sign and the
  same dominant monomial. Then $b \in \mathbf{Ad}$.
\end{lemma}

\begin{proof}
  For $x, y \in \mathbf{No}^{\neq}$, we write $x \equiv y$ if $x \asymp y$ and
  $x$ and $y$ have the same sign. Let us prove by induction on $i \in
  \mathbb{N}$ that $b_{; i}$ is defined and that $a_{; i} - \varphi_i \equiv
  b_{; i} - \varphi_i$. Since this implies that $\varphi_i
  \trianglelefteqslant b_{; i}$, that $\log \frac{b_{; i} -
  \varphi_i}{\varepsilon_i \mathe^{\psi_i}} \prec \tmop{supp} \psi_i$, and
  that $\varphi_i \vartriangleleft \sharp_{\alpha_{i - 1}} (b_{; i})$ if $i >
  0$, this will yield~$b \in \mathbf{Ad}$.
  
  The result follows from our hypothesis if $i = 0$. Assume now that $a_{; i}
  - \varphi_i \equiv b_{; i} - \varphi_i$ and let us prove that $a_{; i + 1} -
  \varphi_{i + 1} \equiv b_{; i + 1} - \varphi_{i + 1}$. Let
  \[ c_i \assign \left( \frac{b_{; i} - \varphi_i}{\varepsilon_i
     \mathe^{\psi_i}} \right) . \]
  We have $c_i \equiv \left( \frac{a_{; i} - \varphi_i}{\varepsilon_i
  \mathe^{\psi_i}} \right)^{\iota_i} = E_{\alpha_i} a_{; i + 1} \in
  \mathbf{No}^{>, \succ}$, so $b_{; i + 1} = L_{\alpha_i} (c_i)$ is defined.
  Moreover $c_i \in \mathcal{E}_{\alpha_i} [E_{\alpha_i} a_{; i + 1}]$ so
  $b_{; i + 1} \in \mathcal{L}_{\alpha_i} [a_{; i + 1}]$. Since $\varphi_{i +
  1} \vartriangleleft \sharp_{\alpha_i} (a_{; i + 1}) = \sharp_{\alpha_i}
  (b_{; i + 1})$, we deduce that ${b_{; i + 1} - \varphi_{i + 1}} \sim a_{; i
  + 1} - \varphi_{i + 1}$, whence in particular $b_{; i + 1} - \varphi_{i + 1}
  \equiv a_{; i + 1} - \varphi_{i + 1}$. This concludes the proof.
\end{proof}

\begin{corollary}
  \label{cor-little-more}We have $\mathbf{Ad}_{\nearrow 1}
  =\mathcal{L}_{\alpha_0} [\mathbf{Ad}_{\nearrow 1}]$.
\end{corollary}

\begin{proof}
  For $b \in \mathbf{Ad}_{\nearrow 1}$, and $c \in \mathcal{L}_{\alpha_0}
  [b]$, we have $\varphi_1 \vartriangleleft \sharp_{\alpha_0} (b) =
  \sharp_{\alpha_0} (c)$ so $c - \varphi_1 \sim b - \varphi_1$. We conclude
  with the previous lemma.
\end{proof}

\begin{lemma}
  \label{lem-admissible-ineq}For $a, b \in \mathbf{Ad}$ and $i \in
  \mathbb{N}^{>}$, we have $L_{\alpha_{i - 1}} a_{; i} <\mathcal{E}_{\alpha_{i
  - 1}} b_{; i}$.
\end{lemma}

\begin{proof}
  Let $j > i$ be minimal with $\varphi_j \neq 0$ or $\psi_j \neq 0$. We thus
  have $a_{; j}, b_{; j} \in \mathcal{P} [\varphi_j \pplus \varepsilon_j
  \mathe^{\psi_j}]$ so $\log a_{; j} \prec b_{; j}$. We have $a_{; i} =
  E_{\alpha_i + \cdots + \alpha_{j - 1}} a_{; j}$ and $b_{; i} = E_{\alpha_i +
  \cdots + \alpha_{j - 1}} b_{; j}$ where $\alpha_i \geqslant \cdots \geqslant
  \alpha_j \geqslant 1$. We deduce by induction using
  Lemma~\ref{lem-transitive-ineq} that $L_{\alpha_{i - 1}} a_{; i}
  <\mathcal{E}_{\alpha_{i - 1}} b_{; i}$.
\end{proof}

\subsection{Nested sequences}

In this subsection, we assume that $\Sigma$ is admissible. For $k \in
\mathbb{N}$ we say that a $\Sigma_{\nearrow k}$-admissible number $a$ is
$\Sigma_{\nearrow k}${\tmem{-nested}}{\index{nested number}} if we have
$E_{\alpha_{k + i}} a_{k ; i + 1} \in \mathbf{Mo}_{\alpha_{k + i}} \setminus
L_{< \alpha_{k + i}}  \mathbf{Mo}_{\alpha_{k + i} \omega}$ for all~$i \in
\mathbb{N}$. We write $\mathbf{Ne}_{\nearrow k}$\label{autolab56} for the
class of $\Sigma_{\nearrow k}$-nested numbers. For $k = 0$ we simply say that
$a$ is $\Sigma$-nested and we write $\mathbf{Ne} \assign \mathbf{Ne}_{\nearrow
0}$\label{autolab57}.

\begin{definition}
  \label{nested-seq-def}We say that $\Sigma$ is
  {\tmem{{\tmstrong{nested}}}}{\index{nested sequence}} if for all $k \in
  \mathbb{N}$, we have
  \[ \mathbf{Ad}_{\nearrow k} = \varphi_k + \varepsilon_k \mathe^{\psi_k} 
     (E_{\alpha_k}  \mathbf{Ad}_{\nearrow k + 1})^{\iota_k} . \]
\end{definition}

Note that the inclusion $\mathbf{Ad}_{\nearrow k} \subseteq \varphi_k +
\varepsilon_k \mathe^{\psi_k}  (E_{\alpha_k}  \mathbf{Ad}_{\nearrow k +
1})^{\iota_k}$ always holds. In {\cite[Section~8.4]{BvdH19}}, we gave examples
of nested and admissible non-nested sequences in the case of transseries, i.e.
with $\alpha_i = 1$ for all $i \in \mathbb{N}$. We next give an example in the
hyperserial case.

\begin{example}
  \label{ex-nested}We claim that the sequence $\Sigma_0$ from
  Example~\ref{ex-admissible} is nested. Indeed, let $k \in \mathbb{N}$ and $a
  \in \mathbf{Ad}_{\nearrow k + 1}$. We have $a = \varphi_{k + 1} \pplus
  \mathe^{\psi_{k + 1}}  (E_{\omega^{2 k + 3}} b)^{- 1}$ for a certain $b \in
  \mathbf{No}^{>, \succ}$ with~$b \asymp L_{\smash{\omega^{2 k + 4}}} \omega$.
  Let us check that the conditions of Definition~\ref{def-admissible-sequence}
  are satisfied for $c \assign \varphi_k + \mathe^{\psi_k}  (E_{\omega^{2 k +
  1}} a)^{- 1}$.
  
  First let $\mathfrak{m} \in \tmop{supp} \psi_k$. We want to prove that
  $\mathfrak{m} \succ \log E_{\omega^{2 k + 1}} a$. We have $\mathfrak{m}=
  L_{\smash{\omega^{2 k + 1}} n} \omega$ for a certain $n \in \mathbb{N}^{>}$.
  Now $a < 2 L_{\smash{\omega^{2 k + 2}}} \omega$, so $\log E_{\omega^{2 k +
  1}} a \prec E_{\omega^{2 k + 1} 2}^{L_{\smash{\omega^{\small{2 k + 2}}}}
  \omega} = L_{\smash{\omega^{2 k + 2}}} (\omega + 2) \prec \mathfrak{m}$.
  
  Secondly, let $\mathfrak{n} \in \tmop{supp} \varphi_k$. We want to prove
  that $\mathfrak{n} \succ \mathe^{\psi_k}  (E_{\omega^{2 k + 1}} a)^{- 1}$.
  We have $\mathfrak{n}= L_{\smash{\omega^{2 k}} n} \omega$ for a certain $n
  \in \mathbb{N}^{>}$. Then $\mathe^{\psi_k}  (E_{\omega^{2 k + 1}} a)^{- 1}
  \prec \mathe^{2 \psi_k}$ by the previous paragraph. Now $2 \psi_k
  +\mathbb{N}< 3 L_{\omega^{2 k + 1}} \omega$ so $\mathe^{2 \psi_k} \prec
  \mathe^{3 L_{\omega^{2 k + 1}} \omega} \prec \mathfrak{n}$.
  
  Finally, we claim that $\varphi_{k + 1} \vartriangleleft \sharp_{\omega^{2 k
  + 1}} (a)$. This is immediate since the dominant term $\tau$ of
  $\mathe^{\psi_{k + 1}}  (E_{\omega^{2 k + 3}} b)^{- 1}$ is positive
  infinite, so $\varphi_{k + 1} \vartriangleleft \varphi_{k + 1} \pplus \tau
  \trianglelefteqslant \sharp_{\omega^{2 k + 1}} (a)$. Therefore $\Sigma_0$ is
  nested.
\end{example}

A crucial feature of nested sequences is that they are sufficient to describe
nested expansions. This is the content of Theorem~\ref{th-eventually-nested}
below.

\begin{lemma}
  \label{lem-standard-form-nested}Let $b \in \mathbf{Ad}_{\nearrow 1}$. If
  $\alpha_0 > 1$, or $\alpha_0 = 1$ and $b_{\succ}$ is not tail-atomic, then
  the hyperserial expansion of $E_{\alpha_0} \sharp_{\alpha_0} (b)$ is
  \[ E_{\alpha_0} \sharp_{\alpha_0} (b) =
     E_{\alpha_0}^{\smash{\sharp_{\alpha_0} (b)}} \]
  If $\alpha_0 = 1$, $b_{\succ} = \psi \pplus \iota \mathfrak{b}$ is
  tail-atomic, and $\mathe^{\mathfrak{b}} = L_{\beta} E_{\alpha}^u$ is a
  hyperserial expansion, then $\psi \in \mathbf{Ad}_{\nearrow 1}$ and the
  hyperserial expansion of $\exp b_{\succ}$ is
  \[ \exp b_{\succ} = \mathe^{\psi}  (L_{\beta} E_{\alpha}^u)^{\iota} . \]
\end{lemma}

\begin{proof}
  Recall that $\sharp_1 (b) = b_{\succ}$. By Corollary~\ref{cor-little-more},
  we have $\sharp_{\alpha_0} (b) \in \mathbf{Ad}_{\nearrow 1}$,. So we may
  assume without loss of generality that $b = \sharp_{\alpha_0} (b)$.
  
  We claim that $E_{\alpha_0}^b \in \mathbf{Mo}_{\alpha_0} \setminus L_{<
  \alpha_0}  \mathbf{Mo}_{\alpha_0 \omega}$. Assume for contradiction that
  $E_{\alpha_0}^b \in L_{< \alpha_0}  \mathbf{Mo}_{\alpha_0 \omega}$ and write
  $E_{\alpha_0}^b = L_{\gamma} \mathfrak{a}$ accordingly. Then
  Corollary~\ref{cor-alpha-alphaomega} implies that $\gamma = 0$, in which
  case we define $n \assign 0$, or $\alpha_0 = \omega^{\mu + 1}$ for some
  ordinal $\mu$ and $\gamma = (\alpha_0)_{/ \omega} n$ for some $n \in
  \mathbb{N}^{>}$. Therefore $E_{\alpha_0}^{b + n} \in \mathbf{Mo}_{\alpha_0
  \omega}$, so ${b + n} \in \mathbf{Mo}_{\alpha_0 \omega}$. This implies that
  \[ b = (b + n) \pplus (- n) . \]
  Recall that $\varphi_1 \vartriangleleft b$. Assume that $n = 0$, so
  $\varphi_1 = 0$. Since $b$ is $\log$-atomic, we also have $\psi_1 = 0$. Let
  $j > 1$ be minimal with $\varphi_j \neq 0$ or $\psi_j \neq 0$. We have
  $\alpha_1 \geqslant \cdots \geqslant \alpha_{j - 1}$ and $b_{1 ; j} =
  L_{\alpha_1 + \cdots + \alpha_{j - 1}} b \in \mathbf{Mo}_{\alpha_{j - 1}
  \omega}$. In particular, the number $b_{1 ; j}$ is log-atomic. If $\varphi_j
  \neq 0$, this contradicts the fact that $\varphi_j \vartriangleleft b_{1 ;
  j}$. If $\psi_j \neq 0$, then $\tmop{supp} \psi_j \succ \log ((b_{1 ; j}
  \mathe^{- \psi_j})^{\iota_j})$ implies
  \[ \log b_{1 ; j} = \psi_j \pplus \log ((b_{1 ; j} \mathe^{-
     \psi_j})^{\iota_j}) . \]
  But then $\log b_{1 ; j}$ is not a monomial: a contradiction. Assume now
  that $n > 0$. So $\varphi_1 = b + n$ and $b = \varphi_1 \pplus (- n)$. But
  then $b_{1 ; 2}$ is not defined: a contradiction. We conclude that
  $E_{\alpha_0}^b \nin L_{< \alpha_0}  \mathbf{Mo}_{\alpha_0 \omega}$.
  
  If $\alpha_0 > 1$, or if $\alpha_0 = 1$ and $b$ is not tail-atomic, then our
  claim yields the result. Assume now that $\alpha_0 = 1$ and that $b = \psi
  \pplus \iota \mathfrak{b}$ is tail-atomic where $\iota \in \{ - 1, 1 \},
  \psi \in \mathbf{No}_{\succ}$, and {$\mathe^{\mathfrak{b}} = L_{\beta}
  E_{\alpha}^u \in \mathbf{Mo}_{\omega}$} is a hyperserial expansion. Then the
  hyperserial expansion of $\exp b$ is $\exp b = \mathe^{\psi}  (L_{\beta}
  E_{\alpha}^u)^{\iota}$.
  
  We next show that $\psi \in \mathbf{Ad}_{\nearrow 1}$. If $\mathfrak{b} \nin
  \mathe^{\psi_1}  (E_{\alpha_1}  \mathbf{Ad}_{\nearrow 2})^{\iota_1}$,
  then~$\varphi_1 \vartriangleleft \psi$, and we conclude with
  Lemma~\ref{lem-little-more} that $\psi \in \mathbf{Ad}_{\nearrow 1}$. Assume
  for contradiction that~$\mathfrak{b} \in \mathe^{\psi_1}  (E_{\alpha_1} 
  \mathbf{Ad}_{\nearrow 2})^{\iota_1}$. Since $\mathfrak{b}$ is log-atomic, we
  must have $\psi_1 = 0$. By the definition of coding sequences, this implies
  that $\iota_1 = 1$ and $\alpha_1 = 1$. So $b = \varphi_1 \pplus
  \varepsilon_1 \exp (b_{1 ; 2})$, whence $\psi = \varphi_1$, $\iota =
  \varepsilon_1$, and $\mathfrak{b}= \exp (b_{1 ; 2})$. In particular the
  number $b_{1 ; 2}$ is $\log$\mbox{-}atomic, hence tail-atomic. Since $b_{1 ;
  2} \in \mathbf{Ad}_{\nearrow 2}$, the claim in the second paragraph of the
  proof, applied to $\Sigma_{\nearrow 1}$, gives $E_1^{b_{1 ; 2}} \nin
  \mathbf{Mo}_{\omega}$. But then also $\mathfrak{b} \nin
  \mathbf{Mo}_{\omega}$: a contradiction.
\end{proof}

We pursue with two auxiliary results that will be used order to construct a
infinite path required in the proof of Theorem~\ref{th-eventually-nested}
below.

\begin{lemma}
  \label{lem-required-paths}For $a \in \mathbf{Ad}$, there is a finite path
  $P$ in $a$ with $u_{P, | P |} \in \mathbf{Ad}_{\nearrow 1} -\mathbb{N}$ or
  {$\psi_{P, | P |} \in \mathbf{Ad}_{\nearrow 1} -\mathbb{N}$}.
\end{lemma}

\begin{proof}
  By Lemma~\ref{lem-subpath-formula}, it is enough to find such a path in
  $E_{\alpha_0} a_{; 1}$. Write~$\alpha_0 \backassign \omega^{\mu}$. Assume
  first that $\mu = 0$, so $\alpha_0 = 1$ and $\psi_0 = 0$. If $(a_{;
  1})_{\succ}$ is not tail-atomic, then the hyperserial expansion of $\exp
  (a_{; 1})_{\succ}$ is $\exp (a_{; 1})_{\succ} = E_1^{\smash{(a_{;
  1})_{\succ}}}$ and $rE_1^{\smash{(a_{; 1})_{\succ}}}$ is the dominant term
  of $\exp a_{; 1}$ for some $r \in \mathbb{R}^{\neq}$. Then the path $P$ with
  $| P | = 1$ and $\tau_{P, 0} \assign rE_1^{\smash{(a_{; 1})_{\succ}}}$
  satisfies $u_{P, | P |} = (a_{; 1})_{\succ} \in \mathbf{Ad}_{\nearrow 1}$.
  If $(a_{; 1})_{\succ}$ is tail-atomic, then there exist $\psi \in
  \mathbf{Ad}_{\nearrow 1}$, $\iota \in \{ - 1, 1 \}$ and $\mathfrak{a} \in
  \mathbf{Mo}_{\omega}$ such that the hyperserial expansion of $\exp (a_{;
  1})_{\succ}$ is $\exp (a_{; 1})_{\succ} = \mathe^{\psi}
  \mathfrak{a}^{\iota}$. Let $r \mathe^{\psi} \mathfrak{a}^{\iota}$ be a term
  in~$\exp a_{; 1}$ with $r \in \mathbb{R}^{\neq}$. Then the path $P$ with $|
  P | = 1$ and $P (0) \assign r \mathe^{\psi} \mathfrak{a}^{\iota}$ satisfies
  $\psi_{P, | P |} = \psi \in \mathbf{Ad}_{\nearrow 1} -\mathbb{N}$.
  
  Assume now that $\mu > 0$. In view of~(\ref{eq-hyperexp-general}), we recall
  that there are an ordinal $\lambda < \alpha_0$ and a number $\delta$ with
  \[ E_{\alpha_0} a_{; 1} = E_{\lambda} \left( L_{\lambda}
     E_{\alpha_0}^{\smash{\sharp_{\alpha_0} (a_{; 1})}} \pplus \delta \right)
     . \]
  If $\mu$ is a limit ordinal, then by Lemma~\ref{lem-standard-form-nested},
  we have a hyperserial expansion $\mathfrak{m} \assign L_{\lambda}
  E_{\alpha_0}^{\smash{\sharp_{\alpha_0} (a_{; 1})}}$. Let $\tau \in
  \tmop{term} \sharp_{\alpha_0} (a_{; 1})$ and set $Q (0) =\mathfrak{m}$ and
  $Q (1) \assign \tau$, so that $Q$ is a path in $\mathfrak{m}$. By
  Lemma~\ref{lem-hyperexp-subpath}, there is a subpath in $E_{\alpha_0} a_{;
  1}$, hence also a path $P$ in $E_{\alpha_0} a_{; 1}$, with $\tau_{P, | P | -
  1} =\mathfrak{m}$. So $u_{P, | P |} = \sharp_{\alpha_0} (a_{; 1}) \in
  \mathbf{Ad}_{\nearrow 1}$. If $\mu$ is a successor ordinal, then we may
  choose $\lambda = \omega^{\mu_-} n$ for a certain $n \in \mathbb{N}$. By
  Lemma~\ref{lem-standard-form-nested}, we have a hyperserial expansion
  $\mathfrak{m} \assign E_{\alpha_0}^{\smash{\sharp_{\alpha_0} (a_{; 1}) -
  n}}$. As in the previous case, there is a path $P$ in $E_{\alpha_0} a_{; 1}$
  with $\tau_{P, | P |} =\mathfrak{m}$, whence $u_{P, | P |} =
  \sharp_{\alpha_0} (a_{; 1}) - n \in \mathbf{Ad}_{\nearrow 1} -\mathbb{N}$.
\end{proof}

\begin{corollary}
  \label{cor-required-path}For $a \in \mathbf{Ad}$ and $k \in \mathbb{N}$,
  there is a finite path $P$ in $a$ with $| P | \geqslant k$ and $u_{P, | P |}
  \in \mathbf{Ad}_{\nearrow k} -\mathbb{N}$ or $\psi_{P, | P |} \in
  \mathbf{Ad}_{\nearrow k} -\mathbb{N}$.
\end{corollary}

\begin{proof}
  This is immediate if $k = 0$. Assume that the result holds at $k$ and pick a
  corresponding path $P$ with $u_{P, | P |} \in \mathbf{Ad}_{\nearrow k}
  -\mathbb{N}$ (resp. $\psi_{P, | P |} \in \mathbf{Ad}_{\nearrow k}
  -\mathbb{N}$). Note that the dominant term $\tau$ of $u_{P, | P |} -
  \varphi_k$ (resp. $\psi_{P, | P |} - \varphi_k$) lies in $\varepsilon_k
  \mathe^{\psi_k}  (E_{\alpha_k}  \mathbf{Ad}_{\nearrow k + 1})^{\iota_k}$ by
  Lemma~\ref{lem-little-more}. Moreover $\tau$ is a term of $u_{P, | P
  |}$~(resp. $\psi_{P, | P |}$). By the previous lemma, there is a path $Q$ in
  $\tau$ with $u_{Q, | Q |} \in \mathbf{Ad}_{\nearrow k + 1} -\mathbb{N}$ or
  $\psi_{Q, | Q |} \in \mathbf{Ad}_{\nearrow k + 1} -\mathbb{N}$, so $(P (0),
  \ldots, P (| P | - 1), Q (0)) \ast Q$ satisfies the conditions.
\end{proof}

\begin{theorem}
  \label{th-eventually-nested}There is a $k \in \mathbb{N}$ such that
  $\Sigma_{\nearrow k}$ is nested.
\end{theorem}

\begin{proof}
  Assume for contradiction that this is not the case. This means that the set
  $\Delta$ of indices $d \in \mathbb{N}$ such that
  $\mathbf{Ad}_{\nearrow d} \neq \varphi_d + \varepsilon_d \mathe^{\psi_d} 
  (E_{\alpha_d}  \mathbf{Ad}_{\nearrow d + 1})^{\iota_d}$ is infinite. We
  write $\Delta = \{ d_i \suchthat i \in \mathbb{N} \}$ where $d_0 < d_1 <
  \cdots$. Fix $a \in \mathbf{Ad}$ and let $d \assign d_i \in \Delta$. Let $u
  \in \mathbf{Ad}_{\nearrow d + 1}$ such that
  \begin{equation}
    \varphi_d + \varepsilon_d \mathe^{\psi_d}  (E_{\alpha_d} u)^{\iota_d} \nin
    \mathbf{Ad}_{\nearrow d}, \label{eq-notin-Ad}
  \end{equation}
  let $n \in \mathbb{N}$ and let $P$ be any finite path with
  \[ u_{P, | P |} = \varphi_d + \varepsilon_d \mathe^{\psi_d}  (E_{\alpha_d}
     u)^{\iota_d} - n. \]
  We claim that we can extend $P$ to a path $Q$ with $| Q | > | P |$, $u_{Q, |
  Q |} \in \mathbf{Ad}_{\nearrow d_{i + 3}} -\mathbb{N}$ and such that $| P |$
  is a bad index in $Q$. Indeed, in view of \Cref{def-admissible-sequence} for
  $\mathbf{Ad}_{\nearrow d}$, the relation (\ref{eq-notin-Ad}) translates into
  the following three possibilities:
  \begin{itemizedot}
    \item There is an $\mathfrak{n} \in \tmop{supp} \psi_d$ with $\mathfrak{n}
    \preccurlyeq \log E_{\alpha_d} u$. We then have $\log E_{\alpha_d} a_{; d
    + 1} \prec \mathfrak{n} \preccurlyeq \log E_{\alpha_d} u$. By
    Lemma~\ref{lem-little-more} and the convexity of $\mathbf{Ad}_{\nearrow d
    + 1}$, we deduce that $\iota_d  (\psi_d)_{\mathfrak{n}} \mathfrak{n}$ lies
    in $\iota_d \log E_{\alpha_d}  \mathbf{Ad}_{\nearrow d + 1}$, so
    $\mathe^{(\psi_d)_{\mathfrak{n}} \mathfrak{n}} \in (E_{\alpha_d} 
    \mathbf{Ad}_{\nearrow d + 1})^{\iota_d}$. By
    Corollary~\ref{cor-required-path} for the admissible sequence starting
    with $(0, 1, 0, \iota_d, \alpha_d)$ and followed by $\Sigma_{\nearrow d +
    1}$, there is a finite path $R_0$ in $\mathe^{(\psi_d)_{\mathfrak{n}}
    \mathfrak{n}}$ with $| R_0 | \geqslant d_{i + 3} - d > 2$ and~$u_{R_0, |
    R_0 |} \in \mathbf{Ad}_{\nearrow d_{i + 3}} -\mathbb{N}$. Taking the
    logarithm and using Lemma~\ref{lem-hyperlog-subpath}, we obtain a finite
    path~$R_1$ in $(\psi_d)_{\mathfrak{n}} \mathfrak{n}$, hence in $\psi_d$,
    with $| R_1 | \geqslant 2$ and $u_{R_1, | R_1 |} = u_{R_0, | R_0 |} \in
    \mathbf{Ad}_{\nearrow d_{i + 3}} -\mathbb{N}$. Write $(E_{\alpha_d} a_{; d
    + 1})^{\iota_d} = r\mathfrak{m} \pplus \rho$ where $r \in
    \mathbb{R}^{\neq}$ and $\mathfrak{m} \in \mathbf{Mo}^{\neq}$. Then $\log
    \mathfrak{m} \asymp E_{\alpha_d} a_{; d + 1} \prec \tmop{supp} \psi_d$, so
    the hyperserial expansion of $\mathe^{\psi_d} \mathfrak{m}$ has one of the
    following forms
    \begin{eqnarray*}
      \mathe^{\psi_d} \mathfrak{m} & = & \mathe^{\psi_d \pplus \delta} 
      (L_{\beta} E_{\alpha}^u)^{\iota} \text{\quad or}\\
      \mathe^{\psi_d} \mathfrak{m} & = & (E_1^{\psi_d \pplus \delta})^{\iota}
    \end{eqnarray*}
    where $(L_{\beta} E_{\alpha}^u)^{\iota}$ is a hyperserial expansion and
    $\delta$ is purely large. In both cases, the path $R = (\varepsilon_d r
    \mathe^{\psi_d} \mathfrak{m}) \ast R_1$ is a finite path $R$ in
    $\varepsilon_d \mathe^{\psi_d}  (E_{\alpha_d} a_{; d + 1})^{\iota_d}$ with
    $u_{R, | R |} = u_{R_1, | R_1 |} \in {\mathbf{Ad}_{\nearrow d_{i + 3}}
    -\mathbb{N}}$. Since $R (0)$ is a term in $u_{P, | P |}$, we may consider
    the path $Q \assign P \ast R$. Moreover, since $\tau_{Q, | P |}$ is a term
    in $\psi_d = \psi_{Q, | P |}$, the index $| P |$ is bad for~$Q$.
    
    \item We have $\log E_{\alpha_d} u \prec \tmop{supp} \psi_d$, but there is
    an $\mathfrak{m} \in \tmop{supp} \varphi_d$ with $\mathfrak{m}
    \preccurlyeq \mathe^{\psi_d}  (E_{\alpha_d} u)^{\iota_d}$. We then have
    $\mathe^{\psi_d}  (E_{\alpha_d} a_{; d + 1})^{\iota_d} \prec
    \varphi_{\mathfrak{m}} \mathfrak{m} \preccurlyeq \mathe^{\psi_d} 
    (E_{\alpha_d} u)^{\iota_d}$. By Lemma~\ref{lem-little-more} and the
    convexity of $\mathbf{Ad}_{\nearrow d + 1}$, we deduce that
    $(\varphi_d)_{\mathfrak{m}} \mathfrak{m}$ lies in $\mathe^{\psi_d} 
    (E_{\alpha_d}  \mathbf{Ad}_{\nearrow d + 1})^{\iota_d}$. So $L_{\alpha_d}
    ((\mathe^{- \psi_d}  (\varphi_d)_{\mathfrak{m}} \mathfrak{m})^{\iota_d})$
    lies in $\mathbf{Ad}_{\nearrow d + 1}$. But then also $v \assign
    \sharp_{\alpha_d} (L_{\alpha_d} ((\mathe^{- \psi_d} 
    (\varphi_d)_{\mathfrak{m}} \mathfrak{m})^{\iota_d}))$ lies in
    $\mathbf{Ad}_{\nearrow d + 1}$ by Corollary~\ref{cor-little-more}. By
    Corollary~\ref{cor-required-path}, there is a finite path $R_0$ in $v$
    with $| R_0 | > 2$ and $u_{R_0, | R_0 |} \in {\mathbf{Ad}_{\nearrow d_{i +
    3}} -\mathbb{N}}$. Applying Lemma~\ref{lem-hyperexp-subpath} to this path
    $R_0$ in $v$, we obtain is a finite path $R_1$ in $(\mathe^{- \psi_d} 
    (\varphi_d)_{\mathfrak{m}} \mathfrak{m})^{\iota_d}$ with $u_{R_1, | R_1 |}
    \in \mathbf{Ad}_{\nearrow d_{i + 3}} -\mathbb{N}$. Since
    $(\varphi_d)_{\mathfrak{m}} \mathfrak{m} \in \mathe^{\psi_d} 
    (E_{\alpha_d}  \mathbf{Ad}_{\nearrow d + 1})^{\iota_d}$, we have
    $\tmop{supp} \psi_d \succ \mathe^{- \psi_d}  (\varphi_d)_{\mathfrak{m}}
    \mathfrak{m}$. So Lemma~\ref{lem-subpath-formula} implies that there is a
    finite path $R$ in $(\varphi_d)_{\mathfrak{m}} \mathfrak{m}$, hence in
    $\varphi_d$, with $u_{R, | R |} \in \mathbf{Ad}_{\nearrow d_{i + 3}}
    -\mathbb{N}$. We have $R (0) \in \tmop{term} \varphi_d \setminus
    \mathbb{R} \subseteq \tmop{term} u_{P, | P |}$, so $Q \assign P \ast R$ is
    a path. Write $\tau$ for the dominant term of $\varepsilon_d
    \mathe^{\psi_d}  (E_{\alpha_d} u')^{\iota_d}$. The index $| P |$ is a bad
    in $Q$ because $\tau_{Q, | P |}$ and $\tau$ both lie in $\tmop{term} a_{Q,
    | P |}$, and $\tau_{Q, | P |} \succ \tau$.
    
    \item We have $\log E_{\alpha_d} u \prec \tmop{supp} \psi_d$ and
    $\tmop{supp} \varphi_d \succ \mathe^{\psi_d}  (E_{\alpha_d} u)^{\iota_d}$,
    but $\varphi_{d + 1} = \sharp_{\alpha_d} (\varphi_{d + 1} \pplus
    \varepsilon_{d + 1} \mathe^{\psi_{d + 1}}  (E_{\alpha_{d + 1}}
    u)^{\iota_{d + 1}})$. By the definition of $\alpha_d$-truncated numbers,
    there is a $\beta < \alpha_d$ with
    \[ \mathe^{\psi_{d + 1}}  (E_{\alpha_{d + 1}} u)^{\iota_{d + 1}} \prec
       \frac{1}{L_{\beta} E_{\alpha_d}^{\varphi_{d + 1}}} \prec
       \mathe^{\psi_{d + 1}}  (E_{\alpha_{d + 1}} a_{; d + 2})^{\iota_{d + 1}}
       . \]
    By convexity of $\mathbf{Ad}_{\nearrow d + 2}$, we get
    $L_{\beta} E_{\alpha_d}^{\varphi_{d + 1}} \in \mathe^{- \psi_{d + 1}} 
    (E_{\alpha_{d + 1}}  \mathbf{Ad}_{\nearrow d + 2})^{- \iota_{d + 1}}$. By
    similar arguments as above (using Corollary~\ref{cor-required-path} and
    Lemmas~\ref{lem-hyperexp-subpath} and~\ref{lem-hyperlog-subpath}), we
    deduce that there is a finite path $R$ in $\varphi_{d + 1}$ with $u_{R, |
    R |} \in \mathbf{Ad}_{\nearrow d_{i + 2}} -\mathbb{N}$. As in the previous
    case $Q \assign P \ast R$ is a path and $| P |$ is a bad index in~$Q$.
  \end{itemizedot}
  Consider a $b \in \mathbf{Ad}_{\nearrow d_1 - 1}$ and the path $P_0 \assign
  \left( \tau_{a - \varphi_{d_0}} \right)$ in $b$. So $P$ is a finite path
  with $u_{P_0, | P_0 |} \in \mathbf{Ad}_{\nearrow d_1}$. Thus there exists a
  path $P_1$ which extends $P_0$ with $u_{P_1, | P_1 |} \in
  \mathbf{Ad}_{\nearrow d_3}$, where $| P_0 |$ is a bad index in $P$.
  Repeating this process iteratively for $i = 2, 3, \ldots$, we construct a
  path $P_i$ that extends $P_{i - 1}$ and such that $u_{P_i, | P_i |} \in
  \mathbf{Ad}_{\nearrow d_{2 i + 1}}$ and such that $| P_{i - 1} |$ is a bad
  index in $P_i$. At the limit, this yields an infinite path $Q$ in $a$ that
  extends each of the paths $P_i$. This path~$Q$ has a cofinal set of bad
  indices, which contradicts Theorem~\ref{th-well-nested}. We conclude that
  there is a $k \in \mathbb{N}$ such that $\Sigma_{\nearrow k}$ is nested.
\end{proof}

\begin{lemma}
  \label{lem-preliminary-group}If $\Sigma$ is nested, then
  $\mathbf{Ad} = \varphi_0 + \varepsilon_0 \mathe^{\psi_0} 
  (\mathcal{E}_{\alpha_0} [E_{\alpha_0}  \mathbf{Ad}_{\nearrow
  1}])^{\iota_0}$.
\end{lemma}

\begin{proof}
  Note that $\mathcal{E}_{\alpha_0} [E_{\alpha_0}  \mathbf{Ad}_{\nearrow 1}] =
  E_{\alpha_0} \mathcal{L}_{\alpha_0} [\mathbf{Ad}_{\nearrow 1}]$. The result
  thus follows from Corollary~\ref{cor-little-more} and the assumption that
  $\Sigma$ is nested.
\end{proof}

\begin{lemma}
  \label{lem-simplicity}Assume that $\Sigma$ is nested. Let $k \in
  \mathbb{N}$, $a \in \mathbf{Ad}$ and $c_k \in \mathbf{No}$ with
  \begin{equation}
    c_k = \varphi_k \pplus \varepsilon_k \mathe^{\psi_k}
    \mathfrak{p}^{\iota_k} \label{eq-ck}
  \end{equation}
  for a certain $\mathfrak{p} \in \mathbf{Mo}^{\succcurlyeq}$ with
  $\mathfrak{p} \sqsubseteq E_{\alpha_j} a_{; k + 1}$ and $\mathfrak{p} \in
  \mathcal{E}_{\omega} [E_{\alpha_k} a_{; k + 1}]$ whenever $\psi_k = 0$. If
  $c_k \in \mathbf{Ad}_{\nearrow k}$, then we have
  \[ (c_k)_{k ;} \sqsubseteq a. \]
\end{lemma}

\begin{proof}
  The proof is similar to the proof of Lemma~\ref{lem-subpath-descent}. We
  have $a_{; k} = \varphi_k \pplus \varepsilon_k \mathe^{\psi_k} 
  (E_{\alpha_k} a_{; k + 1})^{\iota_k}$ and we must have $\tmop{supp} \psi_k
  \succ \log \mathfrak{p}$ since $c_k = \varphi_k \pplus \varepsilon_k
  \mathe^{\psi_k} \mathfrak{p}^{\iota_k} \in \mathbf{Ad}_{\nearrow k}$. If
  follows from the deconstruction lemmas in \Cref{subsection-decomposition}
  that $c_k \sqsubseteq a_{; k}$. This proves the result in the case when~$k =
  0$.
  
  Now assume that $k > 0$. Setting $c_{k - p} \assign \Phi_{k - p ; k} (c_k)$,
  let us prove by induction on {$p \leqslant k$}~that
  \begin{eqnarray*}
    c_{k - p} & \in & \mathbf{Ad}_{\nearrow k - p}\\
    c_{k - p} & \in & \mathbf{No}_{\succ, \alpha_{k - p - 1}}\\
    c_{k - p} & \sqsubseteq & a_{; k - p} .
  \end{eqnarray*}
  For $p = k$, the last relation yields the desired result.
  
  If $p = 0$, then we have $c_k \in \mathbf{Ad}_{\nearrow k}$ by assumption
  and we have shown above that {$c_k \sqsubseteq a_{; k}$}. We have $\varphi_k
  \vartriangleleft \sharp_{\alpha_{k - 1}} (c_k)$ and $\mathe^{\psi_k}
  \mathfrak{p}^{\iota_j}$ is a~monomial, so (\ref{eq-ck}) yields {$c_k =
  \sharp_{\alpha_{k - 1}} (c_k) \in \mathbf{No}_{\succ, \alpha_{k - 1}}$}.
  This deals with the case $p = 0$. In addition, we have {$c_k > 0$} because
  $k > 0$ and $c_k \in \mathbf{Ad}_{\nearrow k}$. Let us show that
  \begin{equation}
    \log c_k \prec a_{; k} . \label{eq-logck}
  \end{equation}
  If $\varphi_k \neq 0$, then this follows from the facts that $\varphi_k
  \vartriangleleft a_{; k}$ and $\varphi_k \vartriangleleft c_k$. If
  $\varphi_k = 0$ and $\psi_k \neq 0$, then $\log (c_k / \varepsilon_k) \sim
  \psi_k \sim \log (a_{; k} / \varepsilon_k) \prec a_{; k}$. If $\varphi_k =
  \psi_k = 0$, then $a_{; k} = E_{\alpha_k} a_{; k + 1}$ and $c_k
  =\mathfrak{p} \in \mathcal{E}_{\omega} [a_{; k}]$, so $\log c_k \prec a_{;
  k}$.
  
  Assume now that $0 < p \leqslant k$ and that the induction hypothesis holds
  for all smaller $p$. We have
  \begin{equation}
    \label{ckp} c_{k - p} = \Phi_{k - p} (c_{k - p + 1}) = \varphi_k +
    \varepsilon_k \mathe^{\psi_k}  \left( E_{\alpha_{k - p}}^{\smash{c_{k - p
    + 1}}} \right)^{\iota_k}
  \end{equation}
  Since $\Sigma$ is nested, we immediately obtain $\varphi_{k - p}
  \vartriangleleft \sharp_{\alpha_{k - p - 1}} (c_{k - p})$, whence $c_{k - p}
  \in \mathbf{No}_{\succ, \alpha_{k - p - 1}}$ as above. Since $c_{k - p - 1}
  \in \mathbf{Ad}_{\nearrow (k - p - 1)}$ and $\Sigma$ is nested, we have
  $c_{k - p} \in \mathbf{Ad}_{\nearrow (k - p)}$.
  Using~(\ref{ckp}),~(\ref{eq-logck}), and the decomposition lemmas, we
  observe that the relation $c_{k - p} \sqsubseteq a_{; k - p}$ is equivalent
  to
  \begin{equation}
    E_{\alpha_{k - p}}^{\smash{c_{k - p + 1}}} \sqsubseteq E_{\alpha_{k - p}}
    a_{; k - p + 1} . \label{eq-simpl}
  \end{equation}
  We have $c_{k - p + 1} \sqsubseteq a_{; k - p + 1}$, so $c_{k - p + 1}
  \sqsubseteq \sharp_{\alpha_{k - p}} (a_{; k - p + 1})$. Note that
  \[ E_{\alpha_{k - p}}^{\smash{\sharp_{\alpha_{k - p}} (a_{; k - p + 1})}}
     =\mathfrak{d}_{\alpha_{k - p}} (E_{\alpha_{k - p}} a_{; k - p + 1})
     \sqsubseteq E_{\alpha_{k - p}} a_{; k - p + 1} . \]
  So it is enough, in order to derive \Cref{eq-simpl}, to prove that
  $E_{\alpha_{k - p}}^{\smash{c_{k - p + 1}}} \sqsubseteq E_{\alpha_{k -
  p}}^{\smash{\sharp_{\alpha_{k - p}} (a_{; k - p + 1})}}$. Now
  \[ L_{\alpha_{k - p}} c_{k - p + 1} <\mathcal{E}_{\alpha_{k - p}}
     \sharp_{\alpha_{k - p}} (a_{; k - p + 1}) \]
  by Lemma~\ref{lem-admissible-ineq}, whence $E_{\alpha_{k - p}}^{\smash{c_{k
  - p + 1}}} \sqsubseteq E_{\alpha_{k - p}}^{\smash{\sharp_{\alpha_{k - p}}
  (a_{; k - p + 1})}}$ by Lemma~\ref{hypexp-mon-lem}.
\end{proof}

For $i \in \mathbb{N}$, $g \in \mathcal{E}_{\alpha_i}$ and $a \in
\mathbf{Ad}$, we have $\varphi_i + \varepsilon_i \mathe^{\psi_i} g
(E_{\alpha_i} a_{; i + 1})^{\iota_i} \in \mathbf{Ad}_{\nearrow i}$ by
Lemma~\ref{lem-preliminary-group}. We may thus consider the strictly
increasing bijection
\[ \Psi_{i, g} \assign \mathbf{Ad} \longrightarrow \mathbf{Ad} ; a \longmapsto
   (\varphi_i + \varepsilon_i \mathe^{\psi_i} g (E_{\alpha_i} a_{; i +
   1})^{\iota_i})_{i ;} . \]
We will prove Theorem~\ref{th-nested-numbers} by proving that the function
group $\mathcal{G} \assign \{ \Psi_{i, g} \suchthat i \in \mathbb{N}, g \in
\mathcal{E}_{\alpha_i} \}$ on $\mathbf{Ad}$ generates the class $\mathbf{Ne}$,
i.e. that we have $\mathbf{Ne} = \mathbf{Smp}_{\mathcal{G}}$. We first need
the following inequality:

\begin{lemma}
  \label{lem-reduction}Assume that $\Sigma$ is nested. Let $i, j \in
  \mathbb{N}$ with $i < j$ and let $g \in \mathcal{E}_{\alpha_i}$. On
  $\mathbf{Ad}$, we have $\Psi_{i, g} < \Psi_{j, H_2}$ if $\sigma_{j + 1 ; i +
  1} = 1$ and $\Psi_{i, g} < \Psi_{j, H_{1 / 2}}$ if $\sigma_{j + 1 ; i + 1} =
  - 1$.
\end{lemma}

\begin{proof}
  It is enough to prove the result for $j = i + 1$. Assume that $\sigma_{i + 2
  ; i + 1} = 1$. Let $a \in \mathbf{Ad}$ and set $a' \assign (\Psi_{i + 1,
  H_2} (a))_{i + 1 ;}$, so that
  \begin{eqnarray*}
    a_{; i + 1} & = & \varphi_{i + 1} + \varepsilon_{i + 1} \mathe^{\psi_{i +
    1}}  (E_{\alpha_{i + 1}} a_{; i + 2})^{\iota_{i + 1}}\\
    a' & = & \varphi_{i + 1} + \varepsilon_{i + 1} \mathe^{\psi_{i + 1}}  (2
    E_{\alpha_{i + 1}} a_{; i + 2})^{\iota_{i + 1}} .
  \end{eqnarray*}
  Note that
  \begin{eqnarray*}
    (\Psi_{i, g} (a))_{i + 1 ;} & \in & \mathcal{T}_{\alpha_i} [a_{; i + 1}] .
  \end{eqnarray*}
  If $\sigma_{; i + 1} = 1$, then $\varepsilon_{i + 1} \iota_{i + 1} =
  \sigma_{; i + 2} / \sigma_{; i + 1} = 1$ and $\Psi_{; i + 1}$ is strictly
  increasing. So we only need to prove that $\mathcal{T}_{\alpha_i} [a_{; i +
  1}] < a'$, which reduces to proving that $\sharp_{\alpha_i} (a_{; i + 1}) <
  \sharp_{\alpha_i} (a')$. Let $\tau$ be the dominant term of $E_{\alpha_{i +
  1}} a_{; i + 2}$. Our assumption that $\Sigma$ is nested gives $\varphi_i +
  \varepsilon_i \mathe^{\psi_i}  (E_{\alpha_i} a')^{\iota_i} \in
  \mathbf{Ad}_{\nearrow i}$, whence $\varphi_{i + 1} \vartriangleleft
  \sharp_{\alpha_i} (a')$. We deduce that $\varphi_{i + 1} + \varepsilon_{i +
  1} \mathe^{\psi_{i + 1}}  (2 \tau)^{\iota_{i + 1}} \trianglelefteqslant
  \sharp_{\alpha_i} (a')$. \Cref{lem-truncated-truncation} implies that
  $\varphi_{i + 1} + \varepsilon_{i + 1} \mathe^{\psi_{i + 1}}  (2
  \tau)^{\iota_{i + 1}}$ is $\alpha_i$-truncated.
  \begin{eqnarray*}
    \sharp_{\alpha_i} (a_{; i + 1}) - \varphi_{i + 1} & \sim & \varepsilon_{i
    + 1} \mathe^{\psi_{i + 1}} \tau^{\iota_{i + 1}},\\
    \sharp_{\alpha_i} (a') - \varphi_{i + 1} & \sim & \varepsilon_{i + 1}
    \mathe^{\psi_{i + 1}}  (2 \tau)^{\iota_{i + 1}}
  \end{eqnarray*}
  and $\varepsilon_{i + 1} \iota_{i + 1} = 1$ implies that
  \[ \varepsilon_{i + 1} \mathe^{\psi_{i + 1}}  (2 \tau)^{\iota_{i + 1}} -
     \varepsilon_{i + 1} \mathe^{\psi_{i + 1}} \tau^{\iota_{i + 1}} \]
  is a strictly positive term. We deduce that $\sharp_{\alpha_i} (a_{; i + 1})
  - \varphi_{i + 1} < \sharp_{\alpha_i} (a') - \varphi_{i + 1}$, whence
  $\sharp_{\alpha_i} (a_{; i + 1}) < \sharp_{\alpha_i} (a')$. The other cases
  when $\sigma_{; i + 1} = - 1$ or when {$\sigma_{i + 2 ; i + 1} = - 1$} are
  proved similarly, using symmetric arguments.
\end{proof}

We are now in a position to prove the following refinement of
\Cref{th-nested-numbers}.

\begin{theorem}
  \label{th-3}If $\Sigma$ is nested, then $\mathbf{Ne}$ is a surreal
  substructure with $\mathbf{Ne} = \mathbf{Smp}_{\mathcal{G}}$.
\end{theorem}

\begin{proof}
  By Proposition~\ref{prop-group-substructure}, the class
  $\mathbf{Smp}_{\mathcal{G}}$ is a surreal substructure, so it is enough to
  prove the equality. We first prove that $\mathbf{Smp}_{\mathcal{G}}
  \subseteq \mathbf{Ne}$.
  
  Assume for contradiction that there are an $a \in
  \mathbf{Smp}_{\mathcal{G}}$ and a $k \in \mathbb{N}$, which we choose
  minimal, such that $a_{; k}$ cannot be written as $a_{; k} = \varphi_k
  \pplus \varepsilon_k \mathfrak{m}_k$ where $\mathfrak{m}_k = \mathe^{\psi_k}
  (E_{\alpha_k}^{a_{; k + 1}})^{\iota_k}$ is a hyperserial expansion. Set
  $\mathfrak{m} \assign \mathfrak{d}_{a_{; k} - \varphi_k}$, $r \assign (a_{;
  k})_{\mathfrak{m}}$ and $\delta \assign (a_{; k})_{\succ \mathfrak{m}}$.
  
  Our goal is to prove that there is a number $m \in \{ k, k + 1 \}$ and
  $\mathfrak{p} \in \mathbf{Mo}^{\succ}$ with
  \begin{equation}
    \begin{array}{ccl}
      \mathfrak{p} & \in & \mathcal{E}_{\alpha_m} [E_{\alpha_m} a_{; m + 1}]\\
      \mathfrak{p} & \sqsubseteq & E_{\alpha_m} a_{; m + 1}\\
      \mathfrak{p} & \sqsubset & E_{\alpha_m} a_{; m + 1}, \quad
      \text{whenever $\nobracket \delta = 0 \nobracket$ and $\nobracket r \in
      \{ - 1, 1 \} \nobracket$.}
    \end{array} \label{p-conds}
  \end{equation}
  Assume that this is proved and set $c_m \assign \varphi_m + \varepsilon_m
  \mathe^{\psi_m} \mathfrak{p}^{\iota_m}$. The first condition and
  Lemma~\ref{lem-preliminary-group} yield $c_m \in \mathbf{Ad}_{\nearrow m}$
  and the relations $\log \mathfrak{p} \prec \tmop{supp} \psi_m$ and
  $\mathe^{\psi_m} \mathfrak{p}^{\iota_m} \prec \tmop{supp} \varphi_m$. The
  second and third condition, together with Lemma~\ref{lem-simplicity}, imply
  $c \assign (c_m)_{m ;} \sqsubset a$. The first condition also implies that
  $c \in \mathcal{G} [a]$: a contradiction. Proving the existence of $m$
  and~$\mathfrak{p}$ is therefore sufficient.
  
  If $\mathfrak{m} \neq \min \tmop{supp} a_{; k}$ or $\mathfrak{m}= \min
  \tmop{supp} a_{; k}$ and $r \nin \{ - 1, 1 \}$, then $m \assign k$ and
  $\mathfrak{p} \assign \mathfrak{d}_{E_{\alpha_k} a_{; k + 1}}$
  satisfy~(\ref{p-conds}). Assume now that $\mathfrak{m}= \min \tmop{supp}
  a_{; k}$ and that $r \in \{ - 1, 1 \}$, whence $r = \varepsilon_k$. If $a_{;
  k + 1} \nin \mathbf{No}_{\succ, \alpha_k}$ then $m \assign k$ and
  $\mathfrak{p} \assign E_{\alpha_k}^{\smash{\sharp_{\alpha_k} (a_{; k +
  1})}}$ satisfy~(\ref{p-conds}). Assume therefore that $a_{; k + 1} \in
  \mathbf{No}_{\succ, \alpha_k}$. This implies that there exist $\gamma <
  \alpha_k$ and~$\mathfrak{a} \in \mathbf{Mo}_{\alpha_k \omega}$ with
  $E_{\alpha_k}^{\smash{a_{; k + 1}}} = L_{\gamma} \mathfrak{a}.$ By the
  definition of coding sequences, there is a least index $j > k$ with
  $\varphi_j \neq 0$ or $\psi_j \neq 0$, so
  \[ E_{\alpha_k}^{\smash{a_{; k + 1}}} = E_{\alpha_k + \cdots + \alpha_{j -
     1}} \left( \varphi_j \pplus \varepsilon_j \mathe^{\psi_j}  \left(
     E_{\alpha_j}^{\smash{a_{; j + 1}}} \right)^{\iota_j} \right) \nin
     \mathbf{Mo}_{\alpha_k \omega} . \]
  We have $\mathfrak{a} \in \mathbf{Mo}_{\alpha_k \omega}$ and $L_{\gamma}
  \mathfrak{a} \in \mathbf{Mo}_{\alpha_k} \setminus \mathbf{Mo}_{\alpha_k
  \omega}$. So by \Cref{cor-alpha-alphaomega}, we must have {$\alpha_k =
  \omega^{\mu + 1}$} for a certain $\mu \in \mathbf{On}$ and $\gamma =
  (\alpha_k)_{/ \omega} n$ for a certain $n \in \mathbb{N}^{>}$. Note
  that~$a_{; k + 1} = L_{\alpha_k} \mathfrak{a}- n$. Recall that $\varphi_{k +
  1} \vartriangleleft a_{; k + 1}$ and $L_{\alpha_k} \mathfrak{a} \in
  \mathbf{Mo}^{\succ}$, so $\varphi_{k + 1} \in \{ L_{\alpha_k} \mathfrak{a},
  0 \}$. The case $\varphi_{k + 1} = L_{\alpha_k} \mathfrak{a}$ cannot occur
  for otherwise
  \[ a_{; k + 2} = \left( \frac{a_{; k + 1} - \varphi_{k + 1}}{\varepsilon_{k
     + 1} \mathe^{\psi_{k + 1}}} \right)^{\iota_{k + 1}} = \frac{n^{\iota_{k +
     1}}}{\varepsilon_{k + 1} \mathe^{\psi_{k + 1}}} \]
  would not lie in $\mathbf{No}^{>, \succ}$. So $\varphi_{k + 1} = 0$. Let $m
  \assign k + 1$ and
  \[ \mathfrak{p} \assign \left( \frac{L_{\alpha_k}
     \mathfrak{a}}{\mathe^{\psi_{k + 1}}} \right)^{\iota_{k + 1}} = \left(
     \frac{\mathfrak{d}_{a_{; k + 1}}}{\mathe^{\psi_{k + 1}}}
     \right)^{\iota_{k + 1}} =\mathfrak{d}_{E_{\alpha_{k + 1}} a_{; k + 2}} .
  \]
  We have $\mathfrak{p} \in \mathcal{E}_{\alpha_{k + 1}} [E_{\alpha_{k + 1}}
  a_{; k + 2}]$ and $\mathfrak{p} \sqsubset E_{\alpha_{k + 1}} a_{; k + 2}$,
  so $m$ and $\mathfrak{p}$ satisfy~(\ref{p-conds}). We deduce that
  $\mathbf{Smp}_{\mathcal{G}}$ is a subclass of $\mathbf{Ne}$.
  
  Conversely, consider $b \in \mathbf{Ne}$ and set $c \assign
  \pi_{\mathcal{G}} [b]$. So there are $i_1, i_2 \in \mathbb{N}$ and $(g, h)
  \in \mathcal{E}_{\alpha_{i_1}}' \times \mathcal{E}_{\alpha_{i_2}}'$ with
  $\Psi_{i_1, g_1} (b) < c < \Psi_{i_2, g_2} (b)$. Let $i \assign \max (i_1 +
  1, i_2 + 1)$. By Lemma~\ref{lem-reduction}, there exist $d_1, d_2 \in
  \left\{ 1 / 2, 2 \right\}$ with $\Psi_{i_1, g_1} < \Psi_{i, H_{d_1}}$ and
  $\Psi_{i_2, g_2} < \Psi_{i, H_{d_2}}$, whence $\Psi_{i, H_{d_1^{- 1}}} (b) <
  c < \Psi_{i, H_{d_2}} (b)$. Since $\Phi_{; i}$ is strictly monotonous, we
  get $c_{; i} - \varphi_i \asymp b_{; i} - \varphi_i$. The numbers
  $\varepsilon_i  (c_{; i} - \varphi_i)$ and $\varepsilon_i  (b_{; i} -
  \varphi_i)$ are monomials, so $c_{; i} - \varphi_i = b_{; i} - \varphi_i$.
  Therefore~$b = c \in \mathbf{Smp}_{\mathcal{G}}$.
\end{proof}

In view of Theorem~\ref{th-3}, Lemma~\ref{lem-reduction}, and
Proposition~\ref{prop-group-substructure}, we have the following
parametrization of $\mathbf{Ne}$:
\[ \nobracket \forall z \in \mathbf{No} \nobracket, \quad \Xi_{\mathbf{Ne}} z
   = \{ L, \Psi_{\mathbb{N}, \mathcal{H}} \Xi_{\mathbf{Ne}} z_L |
   \Psi_{\mathbb{N}, \mathcal{H}} \Xi_{\mathbf{Ne}} z_R, R \} . \]
We conclude this section with a few remarkable identities for
$\Xi_{\mathbf{Ne}}$.

\begin{lemma}
  If $\Sigma$ is nested, then for $i \in \mathbb{N}$ and $a, b \in
  \mathbf{Ne}$, we have $a \sqsubseteq b \Longleftrightarrow a_{; i}
  \sqsubseteq b_{; i}$.
\end{lemma}

\begin{proof}
  By {\cite[Lemma~4.5]{BvdH19}} and since the function $\Phi_{i ;}$ is
  strictly monotonous, it is enough to prove that~$\forall a, b \in
  \mathbf{Ne}, a \sqsubseteq b \Longleftarrow a_{; i} \sqsubseteq b_{; i}$. By
  induction, we may also restrict to the case when $i = 1$. So assume that
  $a_{; 1} \sqsubseteq b_{; 1}$. Recall that $L_{\alpha_0} a_{; 1} \prec
  \mathcal{E}_{\alpha_0} b_{; 1}$ by Lemma~\ref{lem-admissible-ineq}. Since
  $a_{; 1}, b_{; 1} \in \mathbf{No}_{\succ, \alpha_0}$, we deduce with
  Lemma~\ref{hypexp-mon-lem} that~$E_{\alpha_0}^{\smash{a_{; 1}}} \sqsubseteq
  E_{\alpha_0}^{\smash{b_{; 1}}}$. It follows using the decomposition lemmas
  that $a \sqsubseteq b$.
\end{proof}

\begin{proposition}
  If $\Sigma$ is nested, then $\mathbf{Ne} = (\mathbf{Ne}_{\nearrow
  1})_{1 ;} = \varphi_0 + \varepsilon_0 \mathe^{\psi_0} 
  (E_{\alpha_0}^{\mathbf{Ne}_{\nearrow 1}})^{\iota_0}$.
\end{proposition}

\begin{proof}
  We have $\mathbf{Ne} \subseteq (\mathbf{Ne}_{\nearrow 1})_{1 ;}$ by
  definition of $\mathbf{Ne}$. So we only need to prove that
  {$(\mathbf{Ne}_{\nearrow 1})_{1 ;} \subseteq \mathbf{Ne}$}. Consider~$b \in
  \mathbf{Ne}_{\nearrow 1}$. Since $\Sigma$ is nested, the number $a \assign
  \varphi_0 + \varepsilon_0 \mathe^{\psi_0}  (E_{\alpha_0} b)$ is
  $\Sigma$-admissible, so we need only justify that $E_{\alpha_0} b \in
  \mathbf{Mo}_{\alpha_0} \setminus L_{< \alpha_0}  \mathbf{Mo}_{\alpha_0
  \omega}$. Since $a$ is $\Sigma$-admissible, we have $\varphi_1
  \vartriangleleft \sharp_{\alpha_0} (b)$. But $b$ is $\Sigma_{\nearrow
  1}$-nested, so $b = \varphi_1 \pplus \tau$ for a certain term $\tau$. We
  deduce that $b = \sharp_{\alpha_0} (b) \in \mathbf{No}_{\succ, \alpha_0}$,
  whence $E_{\alpha_0} b \in \mathbf{Mo}_{\alpha_0}$.
  
  Assume for contradiction that $E_{\alpha_0}^b \in L_{< \alpha_0} 
  \mathbf{Mo}_{\alpha_0 \omega}$ and write $E_{\alpha_0}^b = L_{\gamma}
  \mathfrak{a}$ where $\mathfrak{a} \in \mathbf{Mo}_{\alpha_0 \omega}$ and
  $\gamma < \alpha_0$. Note that $\gamma \neq 0$: otherwise $\varphi_i$ and
  $\psi_i$ would be zero for all $i \geqslant 1$, thereby contradicting
  Definition~\ref{def-coding-sequence}(\ref{def-coding-sequence-e}). By
  Corollary~\ref{cor-alpha-alphaomega}, we must have $\alpha_0 = \omega^{\mu +
  1}$ for a certain ordinal~$\mu$ and $\gamma = \omega^{\mu} n$ for a certain
  $n \in \mathbb{N}^{>}$. Consequently, $b = L_{\alpha_0} \mathfrak{a}- n \in
  \mathbf{Mo} - n$. If {$\varphi_1 \neq 0$}, then the condition {$\varphi_1
  \vartriangleleft \sharp_{\alpha_0} (b)$} implies $\varphi_1 = b$, which
  leads to the contradiction that~$b_{1 ; 2} = 0 \nin \mathbf{No}^{>, \succ}$.
  If {$\varphi_1 = 0$}, then {$\mathbf{Ne}_{\nearrow 1} \subseteq
  \varepsilon_1  \mathbf{Mo}$}, whence {$n = 0$}: a contradiction.
\end{proof}

\begin{corollary}
  \label{cor-nested-translations}If $\Sigma$ is nested, then for $z \in
  \mathbf{No}$, we have
  \[ \Xi_{\mathbf{Ne}} z = \varphi_0 + \varepsilon_0 \mathe^{\psi_0}  \left(
     E_{\alpha_0}^{\smash{\Xi_{\mathbf{Ne}_{\nearrow 1}} \sigma_{; 1} z}}
     \right)^{\iota_0} . \]
\end{corollary}

\begin{corollary}
  \label{cor-nested-k}If $\Sigma$ is nested and $k \in \mathbb{N}$, then
  \[ \Xi_{\mathbf{Ne}} = \Phi_{k ;} \circ \Xi_{\mathbf{Ne}_{\nearrow k}} \circ
     H_{\sigma_{; k}} . \]
\end{corollary}

\begin{proposition}
  Assume that $\Sigma$ is nested with $(\varphi_0, \varepsilon_0, \psi_0,
  \iota_0) = (0, 1, 0, 1)$, assume that {$\alpha_0 \in \omega^{\mathbf{On} +
  1}$} and write $\beta \assign (\alpha_0)_{/ \omega}$. Consider the coding
  sequence $\Sigma'$ with $(\varphi'_i, \varepsilon'_i, \psi'_i, \iota'_i,
  \alpha'_i) = (\varphi_i, \varepsilon_i, \psi_i, \iota_i, \alpha_i)$ for all
  $i \in \mathbb{N}$, with the only exception that
  \[ \varphi_1' = \varphi_1 - n. \]
  If $\psi_1 < 0$, or $\psi_1 = 0$ and $\iota_1 = - 1$, then $\Sigma'$ is
  nested and we have
  \[ \Xi_{\smash{\mathbf{Ne}'}} = L_{\beta n} \circ \Xi_{\mathbf{Ne}}, \]
  where $\mathbf{Ne}'$ is the class of $\Sigma^{[n]}$-nested numbers.
\end{proposition}

\begin{proof}
  Assume that $\psi_1 < 0$, or $\psi_1 = 0$ and $\iota_1 = - 1$. In
  particular, if $a$ is $\Sigma$-admissible, then $a_{; 1} - \varphi_1 \prec
  1$, so $a_{; 1} - \varphi_1 \prec \tmop{supp} \varphi_1'$. For $b \in
  \mathbf{No}^{>, \succ}$, it follows that $E_{\alpha_0} (b - n)$ is
  $\Sigma^{[n]}$-admissible if and only if $E_{\alpha_0} b$ is
  $\Sigma$-admissible. Let $\mathbf{Ad}_{\nearrow i}'$ be the class of
  $\Sigma'_{\nearrow i}$\mbox{-}admissible numbers, for each $i \in
  \mathbb{N}$. We have $L_{\beta n}  \mathbf{Ad} = \mathbf{Ad}'$ by the
  previous remarks, and $\Sigma'$ is admissible. For $i > 1$, we have
  $\Sigma'_{\nearrow i} = \Sigma_{\nearrow i}$, so
  \[ \mathbf{Ad}'_{\nearrow i} = \mathbf{Ad}_{\nearrow i} \supseteq \varphi'_i
     + \varepsilon'_i \mathe^{\psi'_i}  (E_{\alpha'_i}  \mathbf{Ad}'_{\nearrow
     i + 1})^{\iota'_i} . \]
  Moreover, $\mathbf{Ad}'_{\nearrow 1} = \mathbf{Ad}_{\nearrow 1} - n$, so
  \begin{eqnarray*}
    \mathbf{Ad}' & \supseteq & L_{\beta n}  \mathbf{Ad} \supseteq L_{\beta n}
    E_{\alpha_0}  \mathbf{Ad}_{\nearrow 1} = L_{\beta n} E_{\alpha_0} 
    (\mathbf{Ad}'_{\nearrow 1} + n) = E_{\alpha_0^{[n]}} 
    \mathbf{Ad}'_{\nearrow 1}\\
    \mathbf{Ad}'_{\nearrow 1} & \supseteq & \varphi_1 - n + \varepsilon_1
    \mathe^{\psi_1}  (E_{\alpha_1}  \mathbf{Ad}_{\nearrow 2})^{\iota_1} =
    \varphi'_1 + \varepsilon'_1 \mathe^{\psi'_1}  (E_{\alpha_1^{[n]}} 
    \mathbf{Ad}'_{\nearrow 2})^{\iota'_1} .
  \end{eqnarray*}
  So $\Sigma'$ is nested. We deduce that $L_{\beta n}  \mathbf{Ne} =
  \mathbf{Ne}'$, that is, we have a strictly increasing bijection $L_{\beta n}
  : \mathbf{Ne} \longrightarrow \mathbf{Ne}'$. It is enough to prove that for
  $a, b \in \mathbf{Ne}$ with $a \sqsubseteq b$, we have $L_{\beta n} a
  \sqsubseteq L_{\beta n} b$. Proceeding by induction on $n$, we may assume
  without loss of generality that $n = 1$. By~{\cite[identity
  (6.3)]{vdH:hypno}}, the function $L_{\beta}$ has the following equation on
  $\mathbf{Mo}_{\alpha_0}$:
  \[ \nobracket \forall \mathfrak{a} \in \mathbf{Mo}_{\alpha_0} \nobracket,
     \quad L_{\beta} \mathfrak{a}= \left\{ L_{\beta}
     \mathfrak{a}_L^{\mathbf{Mo}_{\alpha_0}} | L_{\beta}
     \mathfrak{a}_R^{\mathbf{Mo}_{\alpha_0}}, \mathfrak{a}
     \right\}_{\mathbf{Mo}_{\alpha_0}} . \]
  So it is enough to prove that $L_{\beta} b < a$. Note that $L_{\beta} b =
  E_{\alpha_0}^{b_{; 1} - 1}$ and $a = E_{\alpha_0}^{a_{; 1}}$ where $b_{; 1}
  - \varphi_1, a_{; 1} - \varphi_1 \prec 1$. So $b_{; 1} - a_{; 1} \prec 1$,
  whence $b_{; 1} - 1 < a_{; 1}$. This concludes the proof.
\end{proof}

\subsection{Pre-nested and nested numbers}

Let $a \in \mathbf{No}$ be a number. We say that~$a$ is {\tmem{pre-nested}} if
there exists an infinite path $P$ in $a$ without any bad index for $a$. In
that case, \Cref{lem-path-sequence} yields a coding sequence~$\Sigma_P$ which
is admissible due to the fact that $a \in (L|R)$ with the notations from
\Cref{section-nested-series}. By Theorem~\ref{th-eventually-nested}, we get a
smallest $k \in \mathbb{N}$ such that $(\Sigma_P)_{\nearrow k}$ is nested. If
$k = 0$, then we say that $a$ is {\tmem{nested}}. In that case,
Theorem~\ref{th-3} ensures that the class $\mathbf{Ne}$ of $\Sigma_P$-nested
numbers forms a surreal substructure, so $a$ can uniquely be written as $a =
\Xi_{\mathbf{Ne}} (c)$ for some surreal parameter $c \in \mathbf{No}$.

One may wonder whether it could happen that $k > 0$. In other words: do there
exist pre-nested numbers that are not nested? For this, let us now describe an
example of an admissible sequence $\Sigma^{\ast}$ such that the class
$\mathbf{Ne}_{\Sigma^{\ast}}$ of $\Sigma^{\ast}$-nested numbers contains
a~smallest element $b$. This number $b$ is pre-nested, but cannot be nested by
Theorem~\ref{th-3}. Note that our example is ``transserial'' in the sense that
it does not involve any hyperexponentials.

\begin{example}
  \label{ex-nested-min}Let $\Sigma = (\varphi_i, \varepsilon_i, 0, 1, 1)_{i
  \in \mathbb{N}}$ be a nested sequence with $\varepsilon_1 = - 1$. Let $a$ be
  the simplest $\Sigma$-nested number. We define a coding sequence
  $\Sigma^{\ast} = (\varphi_i^{\ast}, \varepsilon_i^{\ast}, 0, 1, 1)_{i \in
  \mathbb{N}}$ by
  \begin{eqnarray*}
    \varepsilon_0^{\ast} & \assign & - 1\\
    \varphi_0^{\ast} & \assign & \mathe^{\varphi_1 - \frac{1}{2} \mathe^{a_{;
    2}}}\\
    (\varphi_i^{\ast}, \varepsilon_i^{\ast}) & \assign & (\varphi_i,
    \varepsilon_i) \text{\qquad for all $\nobracket i > 0 \nobracket$.}
  \end{eqnarray*}
  Note that
  \[ a_{; 1} = \varphi_1 - \mathe^{a_{; 2}} = \varphi_1 \pplus \varepsilon_1
     \mathe^{a_{; 2}}, \]
  where $\mathe^{a_{; 2}}$ is an infinite monomial, so $b \assign
  \varphi_0^{\ast} - \mathe^{a_{; 1}}$ is $\Sigma^{\ast}$-nested. In
  particular, the sequence $\Sigma^{\ast}$ is admissible.
  
  Assume for contradiction that there is a $\Sigma^{\ast}$-nested number $c$
  with $c < b$. Since {$\varepsilon_0^{\ast} = \varepsilon_1^{\ast} = - 1$},
  we have $c_{; 2} < b_{; 2}$. Recall that $c_{; 2}$ and $b_{; 2}$ are purely
  large, so $\mathe^{c_{; 2}} \prec \mathe^{b_{; 2}} = \mathe^{a_{; 2}}$. In
  particular
  \[ \mathe^{c_{; 1}} = \mathe^{\varphi_1 - \mathe^{c_{; 2}}} \succcurlyeq
     \mathe^{\varphi_1 - \frac{1}{2} \mathe^{a_{; 2}}} = \varphi_0^{\ast}, \]
  which contradicts the assumption that $c$ is $\Sigma^{\ast}$-nested. We
  deduce that $b$ is the minimum of the class $\mathbf{Ne}_{\Sigma^{\ast}}$ of
  $\Sigma^{\ast}$-nested numbers. In view of Theorem~\ref{th-3}, the sequence
  $\Sigma^{\ast}$ cannot be nested.
\end{example}

The above examples shows that there exist admissible sequences that are not
nested. Let us now construct an admissible sequence $\Sigma^{\varnothing}$
such that the class $\mathbf{Ne}_{\Sigma^{\varnothing}}$ of
$\Sigma^{\varnothing}$-nested numbers is actually empty.

\begin{example}
  We use the same notations as in \Cref{ex-nested-min}. Set
  $(\varphi_0^{\varnothing}, \varepsilon_0^{\varnothing}) \assign {(\mathe^b,
  1)}$ and $(\varphi_i^{\varnothing}, \varepsilon_i^{\varnothing}) \assign
  (\varphi_{i - 1}^{\ast}, \varepsilon_{i - 1}^{\ast})$ for all $i > 0$. We
  claim that the coding sequence $\Sigma^{\varnothing} \assign
  {(\varphi_i^{\varnothing}, \varepsilon_i^{\varnothing}, 0, 1, 1)}_{i \in
  \mathbb{N}}$ is admissible. In order to see this, let {$\psi \assign 1 / 2
  \mathe^{b_{; 1}}$}. Then
  \[ \mathe^{\varphi_1^{\varnothing} \pplus \varepsilon_1 \psi} =
     \mathe^{\varphi_0^{\ast} \pplus \varepsilon_0^{\ast} \psi} \prec
     \mathe^{\varphi_0^{\ast} \pplus \varepsilon_0^{\ast} \mathe^{b_{; 1}}} =
     \mathe^b . \]
  Since $\varphi_1^{\varnothing} \pplus \varepsilon_1 \psi$ is
  $(\Sigma^{\varnothing})_{\nearrow 1}$-admissible (i.e.
  $\Sigma^{\ast}$-admissible), we deduce that $\mathe^b +
  \mathe^{\varphi_1^{\varnothing} \pplus \varepsilon_1 \psi}$ is
  $\Sigma^{\varnothing}$-admissible, whence $\Sigma^{\varnothing}$ is
  admissible. Assume for contradiction that
  $\mathbf{Ne}_{\Sigma^{\varnothing}}$ is non-empty, and let $\mathe^b \pplus
  \mathfrak{m} \in \mathbf{Ne}_{\Sigma^{\varnothing}}$. Then $\log
  \mathfrak{m}$ is $\Sigma^{\ast}$-nested, so $\log \mathfrak{m} \geqslant b$,
  whence $\mathfrak{m} \succcurlyeq \mathe^b$: a contradiction.
\end{example}

\section{Numbers as hyperseries}\label{section-numbers-as-hyperseries}

Traditional transseries in $x$ can be regarded as infinite expressions that
involve $x$, real constants, infinite summation, exponentiation and
logarithms. It is convenient to regard such expressions as infinite labeled
trees. In this section, we show that surreal numbers can be represented
similarly as infinite expressions in $\omega$ that also involve
hyperexponentials and hyperlogarithms. One technical difficulty is that the
most straightforward way to do this leads to ambiguities in the case of nested
numbers. These ambiguities can be resolved by associating a surreal number to
every infinite path in the tree. In view of the results from
\Cref{section-nested-series}, this will enable us to regard any surreal number
as a unique hyperseries in $\omega$.

\begin{remark}
  In the case of ordinary transseries, our notion of tree expansions below is
  slightly different from the notion of tree representations that was used
  in~{\cite{vdH:phd,Schm01}}. Nevertheless, both notions coincide modulo
  straightforward rewritings.
\end{remark}

\subsection{Introductory example}

Let us consider the monomial $\mathfrak{m}= \exp \left( 2 E_{\omega} \omega -
\sqrt{\omega} + L_{\omega + 1} \omega \right)$ from Example~\ref{example-st}.
We may recursively expand $\mathfrak{m}$ as
\[ \mathfrak{m}= \mathe^{2 E_{\omega^{\small{2}}}^{L_{\omega^{\small{2}}}
   \omega + 1} - E_1^{\frac{1}{2} L_1 \omega}}  (L_{\omega} \omega) . \]
In order to formalize the general recursive expansion process, it is more
convenient to work with the unsimplified version of this expression
\[ \mathfrak{m}= \mathe^{2 \cdot \mathe^0 \cdot \left(
   E_{\omega^{\small{2}}}^{1 \cdot \mathe^0 \cdot \left(
   L_{\omega^{\small{2}}} \omega \right)^1 + 1 \cdot 1} \right)^1 + (- 1)
   \cdot \mathe^{_0} \cdot \left( E_1^{1 / 2 \cdot \mathe^0 \cdot (L_1
   \omega)^1} \right)^1}  (L_{\omega} \omega)^1 . \]
Introducing $\wp_c : x \longmapsto x^c$ as a notation for the ``power''
operator, the above expression may naturally be rewritten as a tree:
\[ 
   \raisebox{-0.5\height}{\includegraphics[width=8.74298176570904cm,height=8.83497310770038cm]{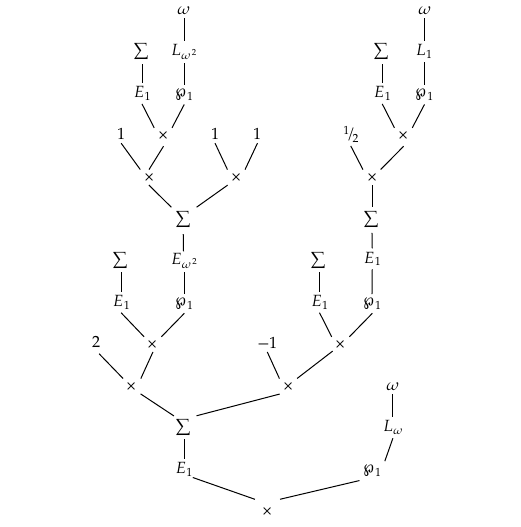}}
\]
In the next subsection, we will describe a general procedure to expand surreal
monomials and numbers as trees.

\subsection{Tree expansions}\label{tree-expansion-sec}

In what follows, a {\tmem{tree}} $T$ is a set of nodes $N_T$ together with a
function that associates to each node $\nu \in N_T$ an {\tmem{arity}}
$\ell_{\nu} \in \mathbf{On}$ and a sequence $(\nu [\alpha])_{\alpha <
\ell_{\nu}} \in N_T^{\ell_{\nu}}$ of {\tmem{children}}; we write $C_{\nu}
\assign \{ \nu [\alpha] \suchthat \alpha < \ell_{\nu} \}$ for the set of
children of $\nu$. Moreover, we assume that $N_T$ contains a~special element
$\rho_T$, called the {\tmem{root}} of $T$, such that for any $\nu \in N_T$
there exist a~unique~$h$ (called the {\tmem{height}} of $\nu$ and also denoted
by $h_{\nu}$) and unique nodes {$\nu_0, \ldots, \nu_h$} with $\nu_0 = \rho_T$,
$\nu_h = \nu$, and $\nu_i \in C_{\nu_{i - 1}}$ for~{$i = 1, \ldots, h$}. The
height $h_T$ of the tree $T$ is the maximum of the heights of all nodes; we
set $h_T \assign \omega$ if there exist nodes of arbitrarily large heights.

Given a class $\tmmathbf{\Sigma}$, an {\tmem{$\tmmathbf{\Sigma}$-labeled
tree}} is a tree together with a map $\lambda : N_T \longrightarrow
\tmmathbf{\Sigma}; \nu \longmapsto \lambda_{\nu}$, called the
{\tmem{labeling}}. Our final objective is to express numbers using
$\tmmathbf{\Sigma}$-labeled trees, where
\begin{eqnarray*}
  \tmmathbf{\Sigma} & \assign & \mathbb{R}^{\neq} \cup \left\{ \omega, \sum,
  \times, \wp_{- 1}, \wp_1 \right\} \cup L_{\omega^{\mathbf{On}}} \cup
  E_{\omega^{\mathbf{On}}} .
\end{eqnarray*}
Instead of computing such expressions in a top-down manner (from the leaves
until the root), we will compute them in a bottom-up fashion (from the root
until the leaves). For this purpose, it is convenient to introduce a separate
formal symbol $\open_c$ for every $c \in \mathbf{On}$, together with the
extended signature
\begin{eqnarray*}
  \tmmathbf{\Sigma}^{\sharp} & \assign & \tmmathbf{\Sigma} \cup \left\{
  \open_c \suchthat c \in \mathbf{No} \right\} .
\end{eqnarray*}
We use $\open_c$ as a placeholder for a tree expression for~$c$ whose
determination is postponed to a later stage.

Consider a $\tmmathbf{\Sigma}^{\sharp}$-labeled tree $T$ and a map $v : N_T
\longrightarrow \mathbf{On}$. We say that $v$ is an {\tmem{evaluation}} of~$T$
if for each node $\nu \in N_T$ one of the following statements holds:
\begin{description}
  \item[E1] \label{E1-rule}$\lambda_{\nu} \in \mathbb{R}^{\neq} \cup \{ \omega
  \}$, $\ell_{\nu} = 0$, and $v (\nu) = \lambda_{\nu}$;
  
  \item[E2] $\lambda_{\nu} = \op{\sum}$, the family $(v (\nu
  [\alpha]))_{\alpha < \ell_{\nu}}$ is well based and $v (\nu) = \sum_{\alpha
  < \ell_v} v (\nu [\alpha])$;
  
  \item[E3] $\lambda_{\nu} = \op{\times}$, $\ell_{\nu} = 2$, and $v (\nu) = v
  (\nu [0]) v (\nu [1])$;
  
  \item[E4] $\lambda_{\nu} = \wp_{\iota}$, $\iota \in \{ - 1, 1 \}$,
  $\ell_{\nu} = 1$, and $v (\nu) = v (\nu [0])^{\iota}$;
  
  \item[E5] $\lambda_{\nu} = L_{\omega^{\mu}}$, $\ell_{\nu} = 1$, and $v (\nu)
  = L_{\omega^{\mu}} v (\nu [0])$;
  
  \item[E6] $\lambda_{\nu} = E_{\omega^{\mu}}$, $\ell_{\nu} = 1$, and $v (\nu)
  = E_{\omega^{\mu}} v (\nu [0])$;
  
  \item[E7] \label{E7-rule}$\lambda_{\nu} = \open_{\alpha}$, $\ell_{\nu} = 0$,
  and $v (\nu) = \alpha$.
\end{description}
We call $v (\rho_T)$ the {\tmem{value}} of $T$ via $v$. We say that $a \in
\mathbf{No}$ is {\tmem{a value}} of $T$ if there exists an evaluation of $T$
with $a = v (\rho_T)$.

\begin{lemma}
  \label{eval-uniqueness-lem}{\tmdummy}
  
  \begin{enumeratealpha}
    \item \label{finite-height-case}If $T$ has finite height, then there
    exists at most one evaluation of $T$.
    
    \item \label{v-vprime-case}Let $v$ and $v'$ be evaluations of $T$ with $v
    (\rho_T) = v' (\rho_T)$. Then $v = v'$.
  \end{enumeratealpha}
\end{lemma}

\begin{proof}
  This is straightforward, by applying the rules
  \hyperref[E1-rule]{$\mathbf{E1}$}--\hyperref[E7-rule]{$\mathbf{E7}$} recursively (from the
  leaves to the root in the case of ({\tmem{\ref{finite-height-case}}}) and
  the other way around for ({\tmem{\ref{v-vprime-case}}})).
\end{proof}

Although evaluations with a given end-value are unique for a fixed tree $T$,
different trees may produce the same value. Our next aim is to describe a
standard way to expand numbers using trees. Let us first consider the case of
a monomial $\mathfrak{m} \in \mathbf{Mo}$. If $\mathfrak{m}= 1$, then the
{\tmem{standard monomial expansion}} of $\mathfrak{m}$ is the
$\tmmathbf{\Sigma}^{\sharp}$-labeled tree $T$ with $N_T = \{ \rho_T \}$ and
$\lambda_{\rho_T} = 1$. Otherwise, we may write $\mathfrak{m}= \mathe^{\psi} 
(L_{\beta} g)^{\iota}$ with $g = \omega$ or $g = E_{\alpha}^u$. Depending on
whether {$g = \omega$} or $g = E_{\alpha}^u$, we respectively take
\[ T \assign
   \raisebox{-0.5\height}{\includegraphics[width=2.99170274170274cm,height=3.37154007608553cm]{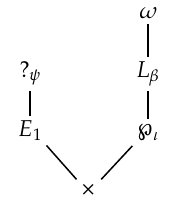}}
   \qquad \infixor \qquad T \assign
   \raisebox{-0.5\height}{\includegraphics[width=2.99170274170274cm,height=4.46225239407058cm]{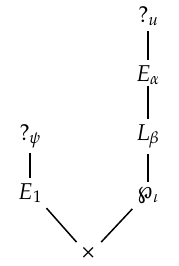}}
\]
and call $T$ the {\tmem{standard monomial expansion}} of $\mathfrak{m}$. Let
us next consider a general number {$a \in \mathbf{No}$} and let $\ell \in
\mathbf{On}$ be the ordinal size of its support. Then we may write $a =
\sum_{\alpha < \ell} c_{\alpha} \mathfrak{m}_{\alpha}$ for a~sequence
$(c_{\alpha})_{\alpha < \ell} \in (\mathbb{R}^{\neq})^{\ell}$ and
a~$\prec$\mbox{-}decreasing sequence $(\mathfrak{m}_{\alpha})_{\alpha < \ell}
\in \mathbf{Mo}^{\ell}$. For each {$\alpha < \ell$}, let~$T_{\alpha}$ be the
standard monomial expansion of $\mathfrak{m}_{\alpha}$. Then we define the
$\tmmathbf{\Sigma}^{\sharp}$-labeled~tree
\[ T \assign
   \raisebox{-0.5\height}{\includegraphics[width=6.7597730552276cm,height=2.77653154925882cm]{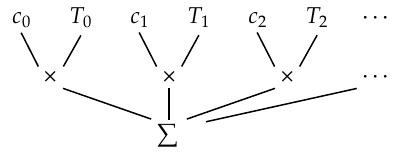}}
\]
and call it the {\tmem{standard expansion}} of $a$. Note that the height of
$T$ is at most $6$, there exists a unique evaluation $v : N_T \longrightarrow
\mathbf{No}$ of $T$, and $v (\rho_T) = a$.

Now consider two trees $T$ and~$T'$ with respective labelings \ $\lambda : N_T
\longrightarrow \tmmathbf{\Sigma}^{\sharp}$ and {$\lambda' : N_{T'}
\longrightarrow \tmmathbf{\Sigma}^{\sharp}$}. We say that $T'$
{\tmem{refines}} $T$ if $N_{T'} \supseteq N_T$ and there exist evaluations {$v
: N_T \longrightarrow \mathbf{No}$} and {$v' : N_{T'} \longrightarrow
\mathbf{No}$} such that $v (\nu) = v' (\nu)$ for all $\nu \in N_T$ and
$\lambda_{\nu} = \lambda'_{\nu}$ whenever $\lambda_{\nu} \nin
\open_{\mathbf{No}}$. Now assume that $v (\rho_T) = a$ for some evaluation $v
: N_T \longrightarrow \mathbf{No}$. Then we say that $T$ is a {\tmem{tree
expansion}} of $a$ if for every $\nu \in N_T$ with $\lambda_{\nu} =
\op{\sum}$, the subtree $T'$ of $T$ with root $\nu$ refines the standard
expansion of $v (\nu)$. In particular, a tree expansion $T$ of a number $a \in
\mathbf{No}$ with $\lambda_{\rho_T} \nin \open_{\mathbf{No}}$ always refines
the standard expansion of $a$.

\begin{lemma}
  \label{unique-tree-expansion-lem}Any $a \in \mathbf{No}$ has a unique tree
  expansion with labels in $\tmmathbf{\Sigma}$.
\end{lemma}

\begin{proof}
  Given $n \in \mathbb{N}$, we say that an
  $\tmmathbf{\Sigma}^{\sharp}$-labeled tree $T$ is {\tmem{$n$-settled}} if
  $\lambda_{\nu} \nin \open_{\mathbf{No}}$ for all nodes $\nu \in N_T$ of
  height $< n$. Let us show how to construct a sequence $(T_n)_{n \in
  \mathbb{N}}$ of $\tmmathbf{\Sigma}^{\sharp}$-labeled tree expansions of $a$
  such that the following statements hold for each $n \in \mathbb{N}$:
  \begin{description}
    \item[S1] \label{S1-cond}$T_n$ is an $n$-settled and of finite height;
    
    \item[S2] \label{S2-cond}$v_n (\rho_{T_n}) = a$ for some (necessarily
    unique) evaluation {$v_n : N_{T_n} \longrightarrow \mathbf{No}$} of $T_n$;
    
    \item[S3] \label{S3-cond}If $n > 0$, then $T_n$ refines $T_{n - 1}$;
    
    \item[S4] \label{S4-cond}If $T$ is a tree expansion of $a$ with labels in
    $\tmmathbf{\Sigma}$, then $T$ refines $T_n$.
  \end{description}
  We will write $\lambda_n : N_{T_n} \longrightarrow
  \tmmathbf{\Sigma}^{\sharp}$ for the labeling of $T_n$.
  
  We take $T_0$ such that $N_{T_0} = \{ \rho_{T_0} \}$ and
  $\lambda_{\rho_{T_0}} = \open_a$. Setting $v_0 (\rho_{T_0}) \assign a$, the
  conditions~\hyperref[S1-cond]{$\mathbf{S1}$}--\hyperref[S4-cond]{$\mathbf{S4}$} are naturally
  satisfied.
  
  Assume now that $T_n$ has been constructed and let us show how to construct
  $T_{n + 1}$. Let~$S$ be the subset of $N_{T_n}$ of nodes $\nu$ of level $n$
  with $v_n (\nu) \in \open_{\mathbf{No}}$. Given $\nu \in S$, let $T_{\nu}$
  be the standard expansion of $v_n (\nu)$ and let $v_{\nu}$ be the unique
  evaluation of $T_{\nu}$. We define $T_{n + 1}$ to be the tree that is
  obtained from $T_n$ when replacing each node $\nu \in S$ by the
  tree~$T_{\nu}$.
  
  Since each tree $T_{\nu}$ is of height at most $6$, the height of~$T_{n +
  1}$ is finite. Since $T_{n + 1}$ is clearly $(n + 1)$-settled, this proves
  \hyperref[S1-cond]{$\mathbf{S1}$}. We define an evaluation $v_{n + 1} : N_{T_{n +
  1}} \longrightarrow \tmmathbf{\Sigma}^{\sharp}$ by setting $v_{n + 1}
  (\sigma) = {\nobreak} v_n (\sigma)$ for any {$\sigma \in N_{T_n}$} and $v_{n
  + 1} (\sigma) = v_{\nu} (\sigma)$ for any~$\nu \in S$ and~$\sigma \in
  N_{T_{\nu}}$ (note that $v_{n + 1}$ is well defined since $v_{\nu}
  (\rho_{T_{\nu}}) = (\lambda_n)_{\nu} = {\nobreak} v_n (\nu)$ for all~{$\nu
  \in S$}). We have $v_{n + 1} (\rho_{T_{n + 1}}) = v_n (\rho_{T_n}) = a$, so
 \hyperref[S2-cond]{$\mathbf{S2}$} holds for $v_{n + 1}$. By construction, $N_{T_{n
  + 1}} \supseteq N_{T_n}$ and the evaluations $v_n$ and~$v_{n + 1}$ coincide
  on $N_{T_n}$; this proves \hyperref[S3-cond]{$\mathbf{S3}$}. Finally, let $T$ be a
  tree expansion of $a$ with labels in~$\tmmathbf{\Sigma}$ and let $v$ be the
  unique evaluation of $T$ with $v (\rho_T) = a$. Then $T$ refines~$T_n$, so
  $v$ coincides with $v_n$ on $N_{T_n}$. Let $\nu \in S$. Since $T$ is a tree
  expansion of $a$, the subtree~$T'$ of $T$ with root $\nu$ refines $T_{\nu}$,
  whence $N_T \supseteq N_{T_{\nu}}$. Moreover, \ $v (\nu) = v_{n + 1} (\nu)$,
  so $v$ coincides with $v_{\nu}$ on $T_{\nu}$. Altogether, this shows that
  $T$ refines $T_{n + 1}$.
  
  Having completed the construction of our sequence, we next define a
  $\tmmathbf{\Sigma}$-labeled tree~$T_{\infty}$ and a map $v_{\infty} :
  N_{T_{\infty}} \longrightarrow \mathbf{No}$ by taking $N_{T_{\infty}} =
  \bigcup_{n \in \mathbb{N}} N_{T_n}$ and by setting $(\lambda_{\infty})_{\nu}
  \assign (\lambda_n)_{\nu}$ and $v_{\infty} (\nu) = v_n (\nu)$ for any $n \in
  \mathbb{N}$ and $\nu \in N_{T_n}$ such that $(\lambda_n)_{\nu} \nin
  \open_{\mathbf{No}}$. By construction, we have $v_{\infty}
  (\rho_{T_{\infty}}) = a$ and $T_{\infty}$~refines~$T_n$ for every $n \in
  \mathbb{N}$.
  
  We claim that $T_{\infty}$ is a tree expansion of $a$. Indeed, consider a
  node $\nu \in N_{T_{\infty}}$ of height $n$ with $\lambda_{\nu} =
  \op{\sum}$. Then $\nu \in N_{T_{n + 1}}$ and {$(\lambda_{n + 1})_{\nu} =
  \op{\sum}$}, since $T_{n + 1}$ is $(n + 1)$-settled. Consequently, the
  subtree of $T_{n + 1}$ with root $\nu$ refines the standard expansion of
  $v_{n + 1} (\nu)$. Since~$T_{\infty}$ refines~$T_{n + 1}$, it follows that
  the subtree of $T_{\infty}$ with root $\nu$ also refines the standard
  expansion of $v_{\infty} (\nu) = v_{n + 1} (\nu)$. This completes the proof
  of our claim.
  
  It remains to show that $T_{\infty}$ is the unique tree expansion of $a$
  with labels in $\tmmathbf{\Sigma}$. So let~$T$ be any tree expansion of $a$
  with labeling $\lambda : N_T \longrightarrow \tmmathbf{\Sigma}$. For every
  $n \in \mathbb{N}$, it follows from \hyperref[S4-cond]{$\mathbf{S4}$} that $N_T
  \supseteq N_{T_n}$. Moreover, since $T_n$ is $n$-settled, $\lambda$
  coincides with both $\lambda_n$ and $\lambda_{\infty}$ on those nodes in
  $N_{T_n}$ that are of height $< n$. Consequently, $N_T \supseteq
  N_{T_{\infty}}$ and $\lambda$ coincides with $\lambda_{\infty}$
  on~$N_{T_{\infty}}$. Since every node in $N_T$ has finite height, we
  conclude that $T = T_{\infty}$.
\end{proof}

\subsection{Hyperserial descriptions}

From now on, we only consider tree expansions with labels in
$\tmmathbf{\Sigma}$, as in \Cref{unique-tree-expansion-lem}. Given a class
$\mathbf{Ne}$ of nested numbers as in~\Cref{section-nested-series}, it can be
verified that every element in $\mathbf{Ne}$ has the same tree expansion. We
still need a notational way to distinguish numbers with the same expansion.

Let $a \in \mathbf{No}$ be a pre-nested number. By
Theorem~\ref{th-eventually-nested}, we get a smallest $k \in \mathbb{N}$ such
that $(\Sigma_P)_{\nearrow k}$ is nested. Hence $a_{P, k} \in \mathbf{Ne}$ for
the class $\mathbf{Ne}$ of $(\Sigma_P)_{\nearrow k}$-nested numbers.
\Cref{th-3} implies that there exists a unique number $c$ with $a_{P, k} =
\Xi_{\mathbf{Ne}} (\sigma_{; k} c)$. We call $c$ the {\tmem{nested rank}}
of~$a$ and write~$\xi_a \assign c$. By \Cref{cor-nested-k}, we note that
$\xi_{u_{P, i}} = \sigma_{; i} \xi_a$ for all $i \in \mathbb{N}$. Given an
arbitrary infinite path $P$ in a number $a \in \mathbf{No}$, there exists a $k
> 0$ such that $P_{\nearrow k}$ has no bad indices for $a_{P, k}$ (modulo a
further increase of $k$, we may even assume $a_{P, k}$ to be nested). Let
$\sigma_{P, k} \assign \tmop{sign} (r_{P, 0} \cdots r_{P, k - 1}) \iota_{P, 0}
\cdots \iota_{P, k - 1} \in \{ - 1, 1 \}$. We call $\xi_P \assign \sigma_{P,
k} \xi_{u_{P, k}}$ the {\tmem{nested rank}} of $P$, where we note that the
value of $\sigma_{P, k} \xi_{u_{P, k}}$ does not depend on the choice~of~$k$.

Let $T$ be the tree expansion of a number $a \in \mathbf{No}$ and let $v : N_T
\longrightarrow \mathbf{No}$ be the evaluation with $a = v (\rho_T)$. An
{\tmem{infinite path}} in $T$ is a sequence {$\nu_0, \nu_1, \ldots$} of nodes
in $N_T$ with $\nu_0 = \rho_T$ and $\nu_{i + 1} \in C_{\nu_i}$ for all $i \in
\mathbb{N}$. Such a path induces an infinite path $P$ in $a$: let $i_1 < i_2 <
\cdots$ be the indices with $\lambda_{\nu_{i_k}} = \op{\sum}$; then we take
$\tau_{P, k} = v (\nu_{i_k + 1})$ for each $k \in \mathbb{N}$. It is easily
verified that this induces a one-to-one correspondence between the infinite
paths in $T$ and the infinite paths in $a$. We call $\xi_{\nu} \assign \xi_P$
the {\tmem{nested rank}} of the infinite path $\nu = (\nu_n)_{n \in
\mathbb{N}}$ in $T$. Denoting by~$I_T$ the set of all infinite paths in $T$,
we thus have a map $\xi : I_T \longrightarrow \mathbf{No} ; \nu \longmapsto
\xi_{\nu}$. We call~$(T, \xi)$ the {\tmem{hyperserial description}} of $a$.

We are now in a position to prove the final theorem of this paper.

\begin{proof*}{Proof of \Cref{th-hyperserial-representation}}
  Consider two numbers $a, a' \in \mathbf{No}$ with the same hyperserial
  description {$(T, \xi)$} and let $v, v' : N_T \longrightarrow \mathbf{No}$
  be the evaluations of $T$ with $v (\rho_T) = a$ and {$v' (\rho_T) = a'$}. We
  need to prove that $a = a'$. Assume for contradiction that $a \neq a'$. We
  define an infinite path $\nu_0, \nu_1, \ldots$ in $T$ with $v (\nu_n) \neq
  v' (\nu_n)$ for all $n$ by setting $\nu_0 \assign \rho_T$ and {$\nu_{n + 1}
  \assign \nu_n [m]$}, where $m \in \mathbb{N}$ is minimal such that $v (\nu_n
  [m]) \neq v' (\nu_n [m])$. (Note that such a number $m$ indeed exists, since
  otherwise $v (\nu_n) = v' (\nu_n)$ using the rules
  \hyperref[E1-rule]{$\mathbf{E1}$}--\hyperref[E7-rule]{$\mathbf{E7}$}.) This infinite path
  also induces infinite paths $P$ and $P'$ in $a$ and $a'$ with $a_{P, n} = v
  (\nu_{i_n})$ and $a_{P', n} = v' (\nu_{i_n})$ for a certain sequence $i_1 <
  i_2 < \cdots$ and all $n \in \mathbb{N}$. Let $n > 0$ be such that
  $P_{\nearrow n}$ and $P'_{\nearrow n}$ have no bad indices for $a_{P, n}$
  and $a_{P', n}$. The way we chose $\nu_0, \nu_1, \ldots$ ensures that the
  coding sequences associated to the paths $P_{\nearrow n}$ and $P'_{\nearrow
  n}$ coincide, so they induce the same nested surreal
  substructure~$\mathbf{Ne}$. It follows that $v (\nu_{i_n}) = a_{P, n} =
  \Xi_{\mathbf{Ne}} (\sigma_{; n} \xi_{\nu}) = a_{P', n} = v' (\nu_{i_n})$,
  which contradicts our assumptions. We conclude that $a$ and $a'$ must be
  equal.
\end{proof*}

\printindex

\section*{Glossary}

\begin{theglossary}{gly}
  \glossaryentry{$\mathbb{R} [[\mathfrak{M}]]$}{field of well-based series
  with real coefficients over $\mathfrak{M}$}{\pageref{autolab1}}
  
  \glossaryentry{$\tmop{supp} f$}{support of a series}{\pageref{autolab2}}
  
  \glossaryentry{$\tmop{term} f$}{set of terms of a
  series}{\pageref{autolab3}}
  
  \glossaryentry{$\mathfrak{d}_f$}{$\max \tmop{supp} f$}{\pageref{autolab4}}
  
  \glossaryentry{$f_{\succ \mathfrak{m}}$}{truncation $\sum_{\mathfrak{n}
  \succ \mathfrak{m}} f_{\mathfrak{n}} \mathfrak{n}$ of
  $f$}{\pageref{autolab5}}
  
  \glossaryentry{$f_{\succ}$}{$f_{\succ 1}$}{\pageref{autolab6}}
  
  \glossaryentry{$h = f \pplus g$}{$h = f + g$ and $\tmop{supp} f \succ
  g$}{\pageref{autolab7}}
  
  \glossaryentry{$f \trianglelefteqslant g$}{$\tmop{supp} f \succ g -
  f$}{\pageref{autolab8}}
  
  \glossaryentry{$f \prec g$}{$\mathbb{R}^{>}  | f | < | g
  |$}{\pageref{autolab9}}
  
  \glossaryentry{$f \preccurlyeq g$}{$\exists r \in \mathbb{R}^{>}, | f | < r
  | g |$}{\pageref{autolab10}}
  
  \glossaryentry{$f \asymp g$}{$f \preccurlyeq g$ and $g \preccurlyeq
  f$}{\pageref{autolab11}}
  
  \glossaryentry{$\mathbb{S}_{\succ}$}{series $f \in \mathbb{S}$ with
  $\tmop{supp} f \succ 1$}{\pageref{autolab12}}
  
  \glossaryentry{$\mathbb{S}^{\prec}$}{series $f \in \mathbb{S}$ with $f \prec
  1$}{\pageref{autolab13}}
  
  \glossaryentry{$\mathbb{S}^{>, \succ}$}{series $f \in \mathbb{S}$ with $f
  \geqslant 0$ and $f \succ 1$}{\pageref{autolab14}}
  
  \glossaryentry{$\mathbb{L}$}{field of logarithmic
  hyperseries}{\pageref{autolab15}}
  
  \glossaryentry{$\mathfrak{L}_{< \alpha}$}{group of logarithmic
  hypermonomials of force $< \: \alpha$}{\pageref{autolab16}}
  
  \glossaryentry{$\mathbb{L}_{< \alpha}$}{field of logarithmic hyperseries of
  force $< \: \alpha$}{\pageref{autolab17}}
  
  \glossaryentry{$g^{\uparrow \gamma}$}{unique series in $\mathbb{L}$ with $g
  = \left( g^{\uparrow \gamma} \right) \circ
  \ell_{\gamma}$}{\pageref{autolab18}}
  
  \glossaryentry{${\text{\tmtextrm{\tmtextbf{\tmtextup{On}}}}}$}{class of
  ordinals}{\pageref{autolab19}}
  
  \glossaryentry{$\sqsubseteq$}{simplicity relation}{\pageref{autolab20}}
  
  \glossaryentry{$\alpha \dotplus \beta$}{ordinal sum of $\alpha$ and
  $\beta$}{\pageref{autolab21}}
  
  \glossaryentry{$a \dottimes \beta$}{ordinal product of $\alpha$ and
  $\beta$}{\pageref{autolab22}}
  
  \glossaryentry{$\omega^{\gamma}$}{ordinal exponentiation with base $\omega$
  at $\gamma$}{\pageref{autolab23}}
  
  \glossaryentry{$\rho \ll \sigma$}{$\rho \prec \omega^{\eta}$ for each
  exponent $\eta$ of $\sigma$}{\pageref{autolab24}}
  
  \glossaryentry{$\rho \lleq \sigma$}{$\rho \preccurlyeq \omega^{\eta}$ for
  each exponent $\eta$ of $\sigma$}{\pageref{autolab25}}
  
  \glossaryentry{$\gamma < \beta$}{$\gamma n < \beta$ for all $n \in
  \mathbb{N}$}{\pageref{autolab26}}
  
  \glossaryentry{$\gamma \preccurlyeq \beta$}{$\exists n \in \mathbb{N},
  \gamma \leqslant \beta n$}{\pageref{autolab27}}
  
  \glossaryentry{$\gamma \asymp \beta$}{$\gamma \preccurlyeq \beta
  \preccurlyeq \gamma$}{\pageref{autolab28}}
  
  \glossaryentry{$\mu_-$}{$\mu = \mu_- + 1$ if $\mu$ is a successor and $\mu_-
  = \mu$ if $\mu$ is a limit}{\pageref{autolab29}}
  
  \glossaryentry{$\alpha_{/ \omega}$}{$\omega^{\mu_-}$ for $\alpha =
  \omega^{\mu}$}{\pageref{autolab30}}
  
  \glossaryentry{$\circ$}{composition law $\circ : \mathbb{L} \times
  {\text{\tmtextrm{\tmtextbf{\tmtextup{No}}}}}^{>, \succ} \longrightarrow
  {\text{\tmtextrm{\tmtextbf{\tmtextup{No}}}}}$}{\pageref{autolab31}}
  
  \glossaryentry{$\mathbb{T}_{\succ, \beta}$}{class of $\beta$-truncated
  series}{\pageref{autolab32}}
  
  \glossaryentry{$\sharp_{\beta} (s)$}{$\trianglelefteqslant$-maximal
  $\beta$-truncated truncation of $s$}{\pageref{autolab33}}
  
  \glossaryentry{$\mathcal{G} [a]$}{class of numbers $b$ with $\exists g, h
  \in \mathcal{G}, g a \leqslant b \leqslant h a$}{\pageref{autolab34}}
  
  \glossaryentry{${\text{\tmtextrm{\tmtextbf{\tmtextup{Smp}}}}}_{\mathcal{G}}$}{class
  of $\mathcal{G}$-simple elements}{\pageref{autolab35}}
  
  \glossaryentry{$\pi_{\mathcal{G}}$}{projection
  ${\text{\tmtextrm{\tmtextbf{\tmtextup{S}}}}} \longrightarrow
  {\text{\tmtextrm{\tmtextbf{\tmtextup{Smp}}}}}_{\mathcal{G}}$}{\pageref{autolab36}}
  
  \glossaryentry{$\leqangle$}{comparison between sets of strictly increasing
  bijections}{\pageref{autolab37}}
  
  \glossaryentry{$X \legeangle Y$}{$X$ and $Y$ are mutually pointwise
  cofinal}{\pageref{autolab38}}
  
  \glossaryentry{$\langle X \rangle$}{function group generated by
  $X$}{\pageref{autolab39}}
  
  \glossaryentry{$T_r$}{translation $a \longmapsto a +
  r$}{\pageref{autolab40}}
  
  \glossaryentry{$H_s$}{homothety $a \longmapsto sa$}{\pageref{autolab41}}
  
  \glossaryentry{$P_s$}{power function $a \longmapsto
  a^s$}{\pageref{autolab42}}
  
  \glossaryentry{$\mathcal{T}$}{function group $\{ T_r \suchthat r \in
  \mathbb{R} \}$}{\pageref{autolab43}}
  
  \glossaryentry{$\mathcal{H}$}{function group $\{ H_s \suchthat s \in
  \mathbb{R}^{>} \}$}{\pageref{autolab44}}
  
  \glossaryentry{$\mathcal{P}$}{function group $\{ P_s \suchthat s \in
  \mathbb{R}^{>} \}$}{\pageref{autolab45}}
  
  \glossaryentry{$\mathcal{E}'$}{function group $\langle E_n H_s L_n : n \in
  \mathbb{N}, s \in \mathbb{R}^{>} \rangle$}{\pageref{autolab46}}
  
  \glossaryentry{$\mathcal{E}^{\ast}$}{function group $\{ E_n, L_n \suchthat n
  \in \mathbb{N} \}$}{\pageref{autolab47}}
  
  \glossaryentry{$\tau_{P, i}$}{value $\tau_{P, i} = P (i)$ of the path $P$ at
  $i < 1 + | P |$}{\pageref{autolab48}}
  
  \glossaryentry{$\mathfrak{m}_{P, i}$}{dominant monomial of $\tau_{P,
  i}$}{\pageref{autolab49}}
  
  \glossaryentry{$r_{P, i}$}{constant coefficient of $\tau_{P,
  i}$}{\pageref{autolab50}}
  
  \glossaryentry{$| P |$}{length of a path $P \in (\mathbb{R}^{\neq} 
  {\text{\tmtextrm{\tmtextbf{\tmtextup{Mo}}}}})^{1 + | P
  |}$}{\pageref{autolab51}}
  
  \glossaryentry{$P \ast Q$}{concatenation of paths}{\pageref{autolab52}}
  
  \glossaryentry{$\nonconverted{blacktrianglelefteqslant}_{\tmop{BM}}$}{Berarducci
  and Mantova's nested truncation relation}{\pageref{autolab53}}
  
  \glossaryentry{${\text{\tmtextrm{\tmtextbf{\tmtextup{Ad}}}}}$}{class of
  admissible numbers}{\pageref{autolab54}}
  
  \glossaryentry{${\text{\tmtextrm{\tmtextbf{\tmtextup{Ad}}}}}_{\nearrow
  k}$}{class of $\Sigma_{\nearrow k}$-admissible numbers}{\pageref{autolab55}}
  
  \glossaryentry{${\text{\tmtextrm{\tmtextbf{\tmtextup{Ne}}}}}$}{class of
  $\Sigma_{\nearrow k}$-nested numbers}{\pageref{autolab56}}
  
  \glossaryentry{${\text{\tmtextrm{\tmtextbf{\tmtextup{Ne}}}}}$}{class of
  $\Sigma$-nested numbers}{\pageref{autolab57}}
\end{theglossary}


\begin{thebibliography}{10}
  \bibitem[1]{vdH:mt}M.~Aschenbrenner, L.~van~den Dries, and  J.~van~der
  Hoeven. {\newblock}\textit{Asymptotic Differential Algebra and}{\newblock} \textit{Model
  Theory of Transseries}. {\newblock} Annals of Mathematics studies, 15, Princeton University Press, 2017.{\newblock}
  
  \bibitem[2]{vdH:icm}M.~Aschenbrenner, L.~van~den Dries, and  J.~van~der
  Hoeven. {\newblock}On numbers, germs, and transseries. {\newblock}In
  \tmtextit{Proc. Int. Cong. of Math. 2018},  volume~1,  pages  1--24. Rio de
  Janeiro, 2018.{\newblock}

  \bibitem[3]{vdH:hivp}M.~Aschenbrenner, L.~van~den Dries, and  J.~van~der
  Hoeven. {\newblock}Hardy fields, the intermediate value property, and $\omega$-freeness, In F. Delon, M. Dickmann, D. Gondard and T. Servi, editors,\textit{Structures algébriques ordonnées, séminaire 2017-2018}, volume 93 of \textit{Prépublications de l'équipe de logique mathématique de l'Université Paris Diderot}, 2019.{\newblock}
  
  \bibitem[4]{vdH:bm}M.~Aschenbrenner, L.~van~den Dries, and  J.~van~der
  Hoeven. {\newblock}The surreal numbers as a universal H-field.
  {\newblock}\textit{Journal of the European Mathematical Society}, 21(4),
  2019.{\newblock}
  
  \bibitem[5]{BvdH19}V.~Bagayoko  and  J.~van~der Hoeven. {\newblock}Surreal
  substructures.{\newblock}\url{https://arxiv.org/abs/2305.02001}, 2019.{\newblock}
  
  \bibitem[6]{vdH:hypno}V.~Bagayoko  and  J.~van~der Hoeven. {\newblock}The
  hyperserial field of surreal numbers.
  {\newblock}\url{https://arxiv.org/abs/????.?????}, 2021.{\newblock}
  
  \bibitem[7]{BvdHK:hyp}V.~Bagayoko, J.~van~der Hoeven, and  E.~Kaplan.
  {\newblock}Hyperserial fields.
  {\newblock}\url{https://hal.science/hal-03196388},
  2021.{\newblock}

  \bibitem[8]{BvdHM:surhyp}V.~Bagayoko, J.~van~der Hoeven, and  V. L.~Mantova.
  {\newblock}Defining a surreal hyperexponential.
  {\newblock}\url{https://hal.science/hal-02861485},
  2021.{\newblock}
  
  \bibitem[9]{Ber20}A.~Berarducci. {\newblock}Surreal numbers, exponentiation
  and derivations. {\newblock}\url{https://arxiv.org/abs/2008.06878},
  2020.{\newblock}
  
  \bibitem[10]{BM18}A.~Berarducci  and  V.~L.~Mantova. {\newblock}Surreal
  numbers, derivations and transseries. {\newblock}\tmtextit{JEMS},
  20(2):339--390, 2018.{\newblock}
  
  \bibitem[11]{BM19}A.~Berarducci  and  V.~L.~Mantova. {\newblock}Transseries
  as germs of surreal functions. {\newblock}\tmtextit{Transactions of the
  American Mathematical Society}, 371:3549--3592, 2019.{\newblock}

  \bibitem[12]{Bou61}N.~Bourbaki. {\newblock}\textit{Fonctions d'une variable réelle}. {\newblock}Éléments de Mathématiques (Chap 5). Hermann, 2-nd edition, 1961.{\newblock}

  \bibitem[13]{Cantor}G.~Cantor. {\newblock}\textit{Sur les fondements de la théorie des ensembles transfinis}. {\newblock}Jacques Gabay, 1899, Reprint from Les Mémoires de la Société des Sciences physiques et naturelles de Bordeaux.{\newblock}
  
  \bibitem[14]{Con76}J.~H.~Conway. {\newblock}\tmtextit{On numbers and games}.
  {\newblock}Academic Press, 1976.{\newblock}
  
  \bibitem[15]{DG87}B.~Dahn  and  P.~Göring. {\newblock}Notes on
  exponential-logarithmic terms. {\newblock}\textit{Fundamenta
  Mathematicae}, 127(1):45--50, 1987.{\newblock}

  \bibitem[16]{vdDE01}L.~van~den Dries and P.~Ehrlich.
  {\newblock}Fields of surreal numbers and exponentiation. {\newblock}\tmtextit{Fundamenta Mathematicae}, 167(2):173--188, 2001.{\newblock}
  
  \bibitem[17]{vdH:loghyp}L.~van~den Dries, J.~van~der Hoeven, and  E.~Kaplan.
  {\newblock}Logarithmic hyperseries. {\newblock}\tmtextit{Transactions of the
  American Mathematical Society}, 372, 2019.{\newblock}
  
  \bibitem[18]{vdDMM01}L.~van~den Dries, A.~Macintyre, and  D.~Marker.
  {\newblock}Logarithmic-exponential series. {\newblock}\tmtextit{Annals of
  Pure and Applied Logic}, 111:61--113, 07 2001.{\newblock}

  \bibitem[19]{DBR1870}P.~du~Bois-Reymond.
  {\newblock}Sur la grandeur relative des infinis des fonctions. {\newblock}\tmtextit{Annali di Matematica Pura e Applicata (1867-1897)}, 4(1):338--353, 1870.{\newblock}

  \bibitem[20]{DBR1875}P.~du~Bois-Reymond.
  {\newblock}Über asymptotische Werte, infinitäre Approximationen und infinitäre Auflösung von Gleichungen {\newblock}\tmtextit{Math. Ann.}, 8:363--414, 1875.{\newblock}

  \bibitem[21]{DBR1877}P.~du~Bois-Reymond.
  {\newblock}Über die Paradoxen der Infinitärscalcüls. {\newblock}\tmtextit{Math. Ann.}, 11:149--167, 1877.{\newblock}
  
  \bibitem[22]{Ec92}J.~Écalle. {\newblock}\tmtextit{Introduction aux fonctions
  analysables et preuve constructive}{\newblock}\textit{de la conjecture de Dulac}.
  {\newblock}Actualités Mathématiques. Hermann, 1992.{\newblock}
  
  \bibitem[23]{Ec16}J.~Écalle. {\newblock}The natural growth scale.{\newblock}\textit{Journal of the European Mathematical Society}, CARMA,
  vol 1:93--223, 2020.{\newblock}
  
  \bibitem[24]{Fel02}U.~Felgner. {\newblock}Die Hausdorffsche Theorie der $\eta_{\alpha}$-Mengen und ihre Wirkungsgeschichte. In Springer-Verlag, editor,{\newblock}\tmtextit{Félix Hausdorff -- gesamte Werke}, volume II, pages 645--674, Berlin, 2002.{\newblock}

  \bibitem[25]{Fish}G.~Fisher. {\newblock}The infinite and infinitesimal quantities of du Bois-Reymond and their reception.
  {\newblock}\tmtextit{Arch. Hist. Exact Sci.},  24:101--163,
  1981.{\newblock}
  
  \bibitem[26]{Gon86}H.~Gonshor. {\newblock}\tmtextit{An Introduction to the
  Theory of Surreal Numbers}. {\newblock}Cambridge Univ. Press,
  1986.{\newblock}
  
  \bibitem[27]{Hahn1907}H.~Hahn. {\newblock}Über die nichtarchimedischen
  Grö$\beta$ensysteme. {\newblock}\textit{Sitz. Akad. Wiss. Wien},
  116:601--655, 1907.{\newblock}
  
  \bibitem[28]{H1910}G.~H.~Hardy. {\newblock}\tmtextit{Orders of infinity, the
  'Infinitärcalcül' of Paul du}{\newblock}\textit{Bois-Reymond}. {\newblock}Cambridge University
  Press  edition, 1910.{\newblock}
  
  \bibitem[29]{H1912}G.~H.~Hardy. {\newblock}Properties of
  logarithmico-exponential functions. {\newblock}\textit{Proceedings of the
  London Mathematical Society}, s2-10(1):54--90, 1912.{\newblock}
  
  \bibitem[30]{vdH:phd}J.~van~der Hoeven. {\newblock}\tmtextit{Automatic
  asymptotics}. {\newblock}PhD thesis, Ecole polytechnique, Palaiseau, France,
  1997.{\newblock}
  
  \bibitem[31]{vdH:gentr}J.~van~der Hoeven. {\newblock}Transséries fortement
  monotones. {\newblock}Chapter 1 of unpublished CNRS activity report,
  2000.{\newblock}
  
  \bibitem[32]{vdH:ln}J.~van~der Hoeven. {\newblock}\tmtextit{Transseries and
  real differential algebra},  volume  1888  of \tmtextit{Lecture Notes in
  Mathematics}. {\newblock}Springer-Verlag, 2006.{\newblock}

  \bibitem[33]{Kn49}H.~Kneser. {\newblock}Relle analytische Lösungen der Gleichug $\phi(\phi(x))=\operatorname{e}^x$ und verwandter Funktionalgleichungen. \tmtextit{Jour. f. d. reine und angewandte Math.}, 187(1/2):56--67, 1950.{\newblock}

  \bibitem[35]{Lemire96}D.~Lemire. {\newblock}\textit{Nombres surréels: décompositions de la
  forme normale, extrema, sommes infinies et fonctions d'ensembles}. PhD thesis, Université de Montréal, 1996.{\newblock}
  
  \bibitem[35]{Lemire97}D.~Lemire. {\newblock}Décompositions successives de la
  forme normale d'un surréel et généralisation des $\varepsilon$-nombres.
  {\newblock}\tmtextit{Ann. Sci. Math. Québec}, 21(2):133--146,
  1997.{\newblock}
  
  \bibitem[36]{Lemire99}D.~Lemire. {\newblock}Remarques sur les surréels dont
  la forme normale se décompose indéfiniment. {\newblock}\tmtextit{Ann. Sci.
  Math. Québec}, 23(1):73--85, 1999.{\newblock}

  \bibitem[37]{MM17}V.~Mantova  and  M.~Matusinski. {\newblock}Surreal numbers
  with derivation, Hardy fields and transseries: a survey.
  {\newblock}\tmtextit{Contemporary Mathematics},  pages  265--290,
  2017.{\newblock}
  
  \bibitem[38]{Schm01}M.~C.~Schmeling. {\newblock}\tmtextit{Corps de
  transséries}. {\newblock}PhD thesis, Université Paris-VII, 2001.{\newblock}
  
\end{thebibliography}
\end{document}